  \documentclass{article}
  \usepackage{amsmath, amsthm, amssymb}
  
 \usepackage{tikz}
 \usepackage{amsmath}
 \usepackage{graphicx}
 \usepackage{multirow}
 \usepackage{caption}
 \usepackage{setspace}
 \usepackage{titlesec}
 \usepackage{color}
 \usepackage[left=1in,right=1in,top=1in,bottom=1in]{geometry}

  \usepackage{amssymb}
 \usepackage{graphicx}
  \usepackage{setspace}
  \usepackage{rotating}
  \usepackage{float}
  \usepackage{multirow}
  \usepackage{color}
  \usepackage{caption}
  \usepackage{rotating}
 \usepackage{eepic}
 \usepackage{colortbl}
 \usepackage[numbers]{natbib}
 \usepackage {multirow}
 \usepackage{setspace}
 \usepackage{indentfirst}

 \usepackage{textcomp}
 \usepackage{array}
 \usepackage{listings}
 \usepackage{setspace}
 \usepackage{colortbl}
 \usepackage{graphicx}
 \usepackage{amssymb, amsmath}
 \usepackage{subfig}
 \usepackage{epsfig}
 \usepackage{times}
 \usepackage{float}
 \usepackage{rotating}
 \usepackage{makeidx}
 \usepackage{url}
 \usepackage{multirow}
 \usepackage{booktabs}
 \usepackage[subfigure, titles]{tocloft}
 \usepackage{acronym}
 \usepackage{datetime}
 \usepackage{algorithm}
 \usepackage{algorithmic}
 \usepackage{etoolbox}

 \usepackage{enumitem}
 \usetikzlibrary {arrows.meta}
 \usetikzlibrary{arrows}
 \usetikzlibrary{shapes.misc}
 \usetikzlibrary{calc, backgrounds}
 \tikzset{vertex/.style={draw=none, line width = 2.4pt, fill, inner sep
 = .09cm, circle}}
 \tikzset{svertex/.style={draw=none, line width = 2.4pt, fill, inner sep
 = .067cm, circle}}
 \tikzset{edge/.style={very thick, gray, line cap=round}}
 \tikzset{bedge/.style={very thick, black, line cap=round}}
 \tikzset{cross/.style={cross out, draw=black, minimum
 size=2*(#1-\pgflinewidth), inner sep=0pt, outer sep=0pt},
 cross/.default={1pt}}
 \tikzset{
   arrow/.pic={\path[tips,every arrow/.try,->,>=#1] (0,0) --
 +(.1pt,0);},
   pics/arrow/.default={triangle 90}
 }
 \usepackage{changepage} 
  \def\fixspacing{%
 	\oddsidemargin=0.25truein
 	\evensidemargin=\oddsidemargin
 	\textwidth=6.0truein
 	\textheight=8.5truein
 	\topmargin=-0.25truein
  }%
  \fixspacing
  \let\reallabel\label
  \def\labelshow#1{\reallabel{#1}{\rm [#1]}}
  \def\showlabels{\let\label\labelshow}
  \theoremstyle{plain}
   \newtheorem{theorem}{Theorem}[section]
   
   \newtheorem{corollary}[theorem]{Corollary}
   \newtheorem{lemma}[theorem]{Lemma}
   \newtheorem{proposition}[theorem]{Proposition}
   \newtheorem*{fanoframework}{Fano Framework}
   \newtheorem*{metaA}{Metatheorem A}
   \newtheorem*{metaB}{Metatheorem B}
   \newtheorem*{metaC}{Metatheorem C}
  \theoremstyle{definition}
   \newtheorem{definition}[theorem]{Definition}
   \newtheorem{observation}[theorem]{Observation}
   \newtheorem{claim}[theorem]{Claim}
   \newtheorem{operation}[theorem]{Operation}

  \def\sG{{\cal G}}

  \let\Si\Sigma
  \let\ga\gamma

  \def\iv{^{-1}}

 \long\def\ignore#1{}
 
 \let\bp\oplus 
 \let\du* 
 \def\pe{\mskip-1.5mu\times\mskip-1.5mu} 
 \let\wi\odot 

 \def\sdu{{}^*} 
 \def\spe{{}^\times} 
 \def\swi{{}^\odot} 

 \def\edu{^{[*]}}
 \def\epe{^{[\times]}}
 \def\ewi{^{[\odot]}}
 
 \def\edpd{^{[*\times*]}}
 
 \def\epdp{^{[\times*\times]}}

 \def\ct{/} 
 \def\dt{\mskip-2mu\setminus\mskip-2mu} 

 \def\mZ{\mathbb{Z}}
 \let\wt\widetilde
 \let\ovl\overline

 \makeatletter
 \newcommand{\superimpose}[3][\mathord]{#1{\mathpalette\superimpose@{{#2}{#3}}}}
 \newcommand{\superimpose@}[2]{\superimpose@@{#1}#2}
 \newcommand{\superimpose@@}[3]{%
  \ooalign{%
    \hfil$\m@th#1#2$\hfil\cr
    \hfil$\m@th#1#3$\hfil\cr
  }%
 }
 \makeatother
 \def\cg{\mskip-2mu\superimpose[\mathbin]/\circ\mskip-2mu}

 \def\mathbi#1{\textbf{\em #1}}
 \def\cs#1{{\mathbf{#1}}} 
 \def\eg#1{{\textnormal{\mathbi{#1}}}}
 \def\eB{\eg{B}}
 \def\eG{\eg{G}}
 \def\eH{\eg{H}}
 \def\eJ{\eg{J}}
 \def\eM{\eg{M}}
 \def\cB{\cs{B}}
 
 \def\cH{\cs{H}}
 \def\cJ{\cs{J}}
 \def\cL{\cs{L}}
 \def\cM{\cs{M}}
 \def\cN{\cs{N}}
 \def\ca{\mathsf{a}}
 \def\cv{\mathsf{v}}
 \def\cf{\mathsf{f}}
 \def\cz{\mathsf{z}}
 \def\ov{\ovl{v}}
 \def\of{\ovl{f}}
 \def\oz{\ovl{z}}

 \let\la\lambda

 \def\tline{\hbox to \hsize}

 \def\citeemm#1{\cite{twisteddualEMM, GraphsonSurfEMM}}

\begin{document}

 \title{A Fano Framework for Embeddings of Graphs in Surfaces
 }
 \author{%
 Blake Dunshee\thanks{%
 Email: \texttt{blake.dunshee@belmont.edu}.
 College of Sciences and Mathematics,
 Belmont University, Nashville, TN 37212, U.~S.~A.}
 \and
 M. N. Ellingham\thanks{%
 Email: \texttt{mark.ellingham@vanderbilt.edu}.
 Department of Mathematics,
 Vanderbilt University, Nashville, TN 37240, U.~S.~A.
 Supported by Simons Foundation awards 429625 and MPS-TSM-00002760.}
 	}
 \date{31 December 2024}

 \maketitle
 \begin{abstract}
 We consider seven fundamental properties of cellular embeddings of
graphs in compact surfaces, and show that each property can be
associated with a point of the Fano plane $F$, in such a way that
allowable combinations of properties correspond to projective subspaces
of $F$.
 This Fano framework allows us to deduce a number of implications
involving the seven properties, providing new results and unifying
existing ones.
 For each property, we provide a correspondence between embeddings with
that property and an associated structure for $4$-regular graphs, using
the medial graph of the graph embedding.  We apply this to characterize
when a graph embedding has a twisted dual with one of the properties.
 For each allowable combination of properties, we show that a graph
embedding with these properties exists.
 We investigate connections between the seven properties and three
weaker `Eulerian' properties.
 Our proofs involve parity conditions on closed walks in an extended
version of the `gem' (graph-encoded map) representation of a graph
embedding.
 \end{abstract}

 \section{Introduction}\label{sec:intro}

 \subsection{Summary of main results}

 Consider the following seven possible properties of a cellular
embedding $\eG$ of a finite graph in a compact surface:
 \begin{enumerate}[label=(\arabic*)]\setlength{\itemsep}{0pt}
  \item
 The underlying surface of $\eG$ is orientable. 
  \item
 A cycle in $\eG$ is 1-sided if and only if it is odd.
  \item
 The underlying graph of $\eG$ is bipartite.
  \item
 The embedding $\eG$ is \emph{directable}, meaning there exists an
orientation (direction) of the edges of $\eG$ so that every face
boundary is a directed walk. 
  \item
 The embedding $\eG$ is 2-face-colorable.
  \item
 If $\eG$ and its dual $\eG\sdu$ are embedded together in their common
underlying surface $\Si$ in the natural way, then the resulting regions
of $\Si - (\eG \cup \eG\sdu)$ can be properly 2-colored.
  \item
 The edges of $\eG$ can be 2-colored so that the colors alternate around
every vertex and around every face.
 \end{enumerate}
 These properties seem somewhat arbitrary, although several of them
involve some kind of $2$-colorability or parity condition. 
 In fact, the above properties are closely related and form a coherent
framework.
 Each property can be described using bipartiteness of a particular
graph derived from an embedded graph $\eG$, or as a parity condition on
cycles or closed walks in graphs known as the `gem' and `jewel' of
$\eG$.
 Using these descriptions we can map these properties to points in the
binary vector space $\mZ_2^3$, where the zero vector corresponds to a
trivial property satisfied by all graph embeddings, so that only the
nonzero vectors are of interest.  The nonzero vectors form a Fano plane,
and the Fano plane structure gives connections between the properties. 

 Our main result can be stated in general terms as follows.

 \begin{fanoframework}
 The seven properties above can be mapped to the points of the Fano
plane $F$, so that for every embedded graph $\eG$, the properties
satisfied by $\eG$ form a projective subspace of $F$.
 \end{fanoframework}

 In particular, the lines of the Fano plane can be described as follows.
 Given $a, b \in \{0, 1, 2, \dots, 7\}$, compute $a \bp b$ by turning
$a$ and $b$ into $3$-digit base $2$ numbers, adding these as vectors in
$\mZ_2^3$, and then converting back to an integer.
 For example, $6 \bp 3 = 110_2 \bp 011_2 = 101_2 = 5$.
 This operation is commutative and associative, and is a special case of
the well-known \emph{Nim sum}.
 Now three properties ($a$), ($b$), and ($c$) with $a, b, c \in \{1, 2,
\dots, 7\}$ form a line in the Fano
plane if and only if $a \bp b \bp c = 0$.

 The Fano Framework has a number of consequences, such as the following
metatheorem.

 \begin{metaA} Every embedded graph has exactly zero, one, three, or
seven of the seven properties above.  Thus, every combination of two
properties implies a third property, and if four properties are
satisfied then all seven properties hold.
 \end{metaA}

 There are another two metatheorems which together give a specific
theorem for any three of the seven properties.

 \begin{metaB} Let $(a)$, $(b)$, and $(c)$ be distinct properties from
the seven properties above, such that $a \bp b \bp c = 0$.  Then any two
of $(a)$, $(b)$, or $(c)$ imply the third.
 \end{metaB}

 \begin{metaC} Let $(a)$, $(b)$, and $(c)$ be distinct properties from
the seven properties above, such that $a \bp b \bp c \ne 0$.  Then the
three properties $(a)$, $(b)$, and $(c)$ together imply all seven
properties.
 \end{metaC}

 The rest of this section discusses important previous ideas that form
the background for our work.
 Section \ref{sec:embeddings} defines the machinery that we need for our
proofs.
 Section \ref{sec:parityconditions} establishes the Fano Framework in
terms of parity conditions for closed walks in gems or jewels, and
Section \ref{sec:fanoproperties} verifies that these conditions correspond
to the seven properties above.
 Section \ref{sec:bidirections} shows how bidirections of the medial
checkerboard (a colored version of the medial graph) are related to the
seven properties.
 These can be applied to give characterizations of when a twisted dual
is orientable or bipartite.
 Section \ref{sec:eulerian} investigates three weaker `Eulerian'
properties, which can also be characterized using parity conditions for
a special type of closed walk, or using bidirections of the medial
checkerboard.
 Section \ref{sec:examples} provides examples of embeddings with all
allowable combinations of the seven properties.
 While most of the results in this paper involve parity (i.e., modulo
$2$) conditions for closed walks in gems or jewels, Section
\ref{sec:allpds} provides a result involving a congruence condition
modulo $4$ which characterizes embeddings all of whose partial duals are
Eulerian, and also embeddings all of whose partial duals have a single
vertex.
 Some concluding remarks appear in Section \ref{sec:conclusion}.

 The major contributions of this paper are the overall framework, and
the development of suitable tools to make the proofs easy.  Our
arguments are mostly elementary, but this is due to the careful choice
and refinement of appropriate machinery.
 In addition to providing new results, our work unifies a number of
results in the literature.

 \subsection{Historical background}

 Our work rests on three key ideas.  The first is \emph{dualities} for
graph embeddings.  The notion of \emph{geometric duality} was originally
considered for polyhedra, and then transferred to graph embeddings
(maps).  This concept may go back to antiquity, and was certainly known
to Kepler in the early 17th century \cite[p.~181]{Ke1619} in connection
with the Platonic solids.  Another relevant concept is \emph{Petrie
duality}, which originated with ideas of skew polygons and zigzag walks
\cite[p.~202]{Co31} in regular polyhedra, attributed by Coxeter to his
collaborator J.~F.~Petrie.
 The interactions between geometric duality and Petrie duality were
described independently by Wilson \cite{Wilson} and Lins
\cite{gemslins}, giving an \emph{action of the symmetric group $S_3$ on
graph embeddings}, which includes a third duality operation, called
`opposite' by Wilson and `phial' by Lins, and now usually known as
\emph{Wilson duality}.

 The second key idea is representation of graph embeddings using
$3$-edge-colored cubic ($3$-regular) graphs.  This was discovered
independently by a number of researchers in the 1970s and early 1980s.
 Robertson \cite{Rob71} and Lins \cite{gemslins} represented embeddings
of graphs in surfaces using $3$-colored $3$-regular graphs, which Lins
called \emph{graph-encoded maps} or \emph{gems}.
 Bonnington and Little \cite{BoLi95} showed how a number of aspects of
topological graph theory can be developed in terms of gems.
 Others, such as Ferri \cite{Fer76} (building on earlier work by Pezzana
\cite{Pez74} and Ferri and Gagliardi \cite{FeGa75}) and Vince
\cite{Vin83} used more general $(n+1)$-edge-colored $(n+1)$-regular
graphs to represent decompositions of $n$-dimensional manifolds; Ferri
called these \emph{crystallizations}.
 References \cite{FeGaGr86, LiVi90} provide useful background, and
\cite{LiVi90} discusses connections with other combinatorial and
algebraic descriptions of graph embeddings.

 The contribution of Lins \cite{gemslins} is particularly important
because it ties together the two key ideas above.  Lins extended gems to
$4$-edge-colored $4$-regular graphs (still representing an embedding of
a graph in a surface), which we will call \emph{jewels}, in which the
$S_3$-action induced by geometric and Petrie dualities can be
implemented simply by permuting three of the edge colors.

 The third key idea is \emph{partial dualities} for graph embeddings.
 Partial Petrie duality has been used for a long time.  It was realized
independently by a number of people in the late 1970s, including Alpert,
Haggard (who attributed the idea to Edmonds around 1970), Ringel, and Stahl, that
the rotation system representation of orientable embeddings could be
extended to represent nonorientable embeddings by the use of what are
now called \emph{edge signatures}, taking values in $\{-1,1\}$ (or
sometimes $\{0,1\}$).
 Partial Petrie duality just consists of `twisting' some edges by
flipping their signatures.  This was used, for example, by Stahl
\cite{Sta78}, in one of the papers that introduced edge signatures, to
show that the genera of nonorientable embeddings of a given graph form
an interval.
 However, it is less obvious that there is also a partial version of
geometric duality, and it was a major breakthrough when Chmutov
\cite{GenDualityChmutov} pointed this out in 2009.  Chmutov's work was
anticipated to some extent by work of Bouchet, who defined
delta-matroids in \cite{Bou87}, along with an operation $S \triangle N$
for a delta-matroid $S$ and a subset $N$ of its ground set.
 Bouchet showed in \cite{Bou89} that every graph embedding has an
associated delta-matroid, but he did not investigate what changes in a
graph embedding correspond to the $\triangle$ operation (which is often
called `twisting' or `pivoting'); essentially these are partial
geometric duals.
 Ellis-Monaghan and Moffatt \citeemm. combined partial geometric duality
and partial Petrie duality, forming a \emph{twisted duality} framework
for graph embeddings, including an \emph{action of $S_3^E$ on embeddings
of graphs with a given edge set $E$}, which refines the $S_3$-action of
Wilson and Lins.

 The simple way in which partial geometric duality can be implemented in
gems was pointed out by one of the authors (Ellingham) and Zha
\cite{ElZh17}, and by Chmutov and Vignes-Tourneret in the more general
context of hypermaps \cite{ChVT22}.
 Ellingham and Zha also observed that Lins' approach to implementing
dualities by permuting edge colors in jewels can be extended to the
whole twisted duality framework.

 \section{Graphs and Cellular Graph Embeddings}\label{sec:embeddings}

 Here we summarize a number of definitions and known results (without
proof).  We define gems, jewels, and medial checkerboards, explain how
they correspond to embedded graphs, and how twisted duality can be
implemented in simple ways using gems and jewels.

 We need to discuss structures with varying levels of detail, so we
introduce some notational conventions. An abstract graph $H$ is denoted
using italics, an embedded graph $\eH$ is denoted using bold italics
(and its underlying graph is $H$), and a colored structure $\cH$
(possibly including embedding information) is denoted using a bold
upright font (with underlying graph $H$, and if appropriate, with
associated embedded graph $\eH$).
 If we mention an embedding property or substructure for a colored
structure $\cH$ we mean that property or substructure of the embedded
graph $\eH$.  And if we mention a graph property or substructure of
$\cH$ or $\eH$ we mean that property or substructure of the abstract
graph $H$.

 \subsection{Graphs}

 In this paper all graphs are finite and may contain loops and multiple
edges.  Graphs may be disconnected or empty.  We use standard graph
theory terminology following \cite{bondymurty} or \cite{Wes96}.
 As we allow loops and multiple edges, it will sometimes be convenient
to consider an edge to consist of two \emph{half-edges}, each incident
with a vertex.

 A graph $G$ is \emph{Eulerian} if it has a circuit (closed trail)
containing all the edges and vertices of $G$.  It is well known that a
graph is Eulerian if and only if it is connected and every vertex has
even degree.  More generally, $G$ is \emph{even-vertex} if every vertex
has even degree, whether or not $G$ is connected.

 A \emph{direction} (or \emph{orientation}) of a graph assigns a
direction to each edge.  (Later we also consider bidirections, which
allow the halves of an edge to be directed independently.)

 \subsection{Embedded graphs and duality}

 We assume the reader has a general familiarity with embeddings of
graphs in surfaces.  All surfaces in this paper are compact (except that
we sometimes discuss embeddings in the plane, which we think of as a
topological subspace of a sphere).
 Surfaces may be disconnected.  A graph embedding is \emph{cellular} if
every face is homeomorphic to an open $2$-cell, i.e., an open disk.  All
embeddings in this paper are cellular unless otherwise indicated. 
 In a cellular embedding $\eG$, each isolated vertex of $G$ is embedded in its
own sphere, with one face, which we call an \emph{isolated vertex
component} of $\eG$.
 The set of faces of $\eG$ is denoted by $F(\eG)$.
 Cellular embeddings of graphs can be represented in various ways,
including by rotation systems with edge signatures, or as ribbon graphs
(also known as fatgraphs or reduced band decompositions).  We refer the
reader to \cite{GraphsonSurfEMM, grosstucker, MoTh01} for details.

 A fundamental property of graph embeddings is orientability.
 A surface is \emph{orientable} if a consistent clockwise direction can
be assigned everywhere on each connected component of the surface.
 A simple closed curve in a surface is either \emph{$1$-sided}, having a
neighborhood homeomorphic to a M\"obius strip, or \emph{$2$-sided},
having a neighborhood homeomorphic to a cylinder.  A surface is
orientable if and only if every simple closed curve in the surface is
$2$-sided.
 A graph embedding $\eG$ is \emph{orientable} if its underlying surface
is orientable, which for cellular embeddings is equivalent to every
cycle of $G$ being embedded as a $2$-sided curve.
 In the rotation system/edge signature representation of a graph
embedding $\eG$, a cycle is $1$-sided if and only if it has an odd
number of edges of signature $-1$ (this does not depend on the
particular rotation system/edge signature representation).

 An embedded graph $\eG$ is \emph{directable} if there is a direction of
$G$ such that each face of $\eG$ is bounded by a directed closed walk. 

 We recall the definition of the \emph{geometric dual} $\eG\sdu$
(usually just called the \emph{dual}) of an embedded graph $\eG$ in a
surface $\Si$.
 Insert a new vertex $v_f$ in each face $f$ of $\eG$, and add one edge
$e\sdu$ crossing each edge $e$ of $\eG$, joining the vertices $v_f$ and
$v_g$ corresponding to the (possibly equal) faces $f$ and $g$ on either
side of $e$.
 The dual $\eG\sdu$ consists of the vertices $v_f$ and the edges $e\sdu$
in the surface $\Si$.
 Each isolated vertex component of $\eG$ corresponds to an isolated
vertex component of $\eG\sdu$.
 The dual embedding $\eG\sdu$ is orientable if and only if the primal
(original) embedding $\eG$ is orientable.

 We can also take the dual of some colored structures $\cH$.  Suppose
$\cH$ is a colored structure consisting of an embedded graph $\eH$ with
colors assigned to vertices, edges, or faces (possibly all three).
 Then to find $\cH\sdu$ we dualize $\eH$ and
 transfer edge colors to corresponding dual edges, and
 vertex colors to corresponding dual faces and vice versa.
 This will apply to barycentric subdivisions and gems (defined below),
and subgraphs of these.  We sometimes call $\cH\sdu$ a \emph{colored
dual} to emphasize that colors are transferred.

 We can also form the \emph{Petrie dual} of a cellularly embedded graph
$\eG$ by twisting every edge of $\eG$ (flipping the
signature of every edge in a rotation scheme/edge signature
representation); it is denoted by $\eG\spe$.
 Alternatively, we can define $\eG\spe$ in terms of the \emph{zigzags}
or \emph{Petrie walks} of $\eG$, closed walks that alternately turn
right and left at each vertex (relative to a local orientation of the
traverser of the walk) as we trace them around $\eG$ until they close
up.  We start with the graph $G$, take a closed disk $D_Z$ for each
zigzag $Z$ of $\eG$, and identify the boundary of $D_Z$ with $Z$,
forming a surface in which $G$ is embedded as $\eG\spe$.
 The zigzags of $\eG$ are the faces of $\eG\spe$, and conversely, the
faces of $\eG$ are the zigzags of $\eG\spe$.
 Each isolated vertex component of $\eG$ corresponds to an isolated
vertex component of $\eG\spe$.
 The underlying graph of $\eG\spe$ is the same as the underlying graph
of $\eG$.

 The \emph{degree} of a face or zigzag of $\eG$ is the length of the
associated closed walk in $G$.  We say $\eG$ is \emph{even-face} or
\emph{even-zigzag} if every face or zigzag, respectively, has even
degree.

 \subsection{Gems, jewels, and medial checkerboards}\label{gjm}

 The main representations we will use for a graph embedding $\eG$ are
gems, jewels, and medial checkerboards.
 Our treatment of gems and jewels loosely follows \cite{BoLi95, ElZh17,
gemslins}.  However, we modify some notation and terminology, and to
link gems, jewels, and medial checkerboards to embedded graphs we also
use barycentric subdivisions, which is an approach based on 
\cite{BoLi95, Fer76, FeGa75, Pez74, Vin83}.

 Isolated vertex components cause some minor difficulties in
representing an embedded graph using a gem, jewel, medial checkerboard,
or barycentric subdivision.  For much of this subsection and the next
subsection we will therefore focus on embedded graphs without isolated
vertex components.

 A \emph{graph-encoded map} or \emph{gem} $\cJ$ consists of a finite
cubic graph $J$ with a proper $3$-edge-coloring $\ga : E(J) \to \{\cv,
\cf, \ca\}$, such that the components of the $2$-factor induced by edges
of colors $\cv$ and $\cf$ are $4$-cycles.  The colors $\cv, \cf, \ca$
are constants, and in figures we will follow the convention of
\cite{BoLi95} and indicate them by red, blue, and yellow, respectively.
 We think of $\cv, \cf, \ca$ as vertex, facial, and auxiliary colors
that describe the correspondence between $\cJ$ and an embedded graph
$\eG$.

 A \emph{bigon} in $\cJ$ is a cycle whose edges are colored alternately
with two of the three colors. The bigons have three types:
 \begin{enumerate}[label=(\alph*)]\setlength{\itemsep}{0pt}
     \item \emph{e-squares}, whose edges are colored with $\cv$ and
$\cf$;
     \item \emph{v-gons}, whose edges are colored with $\cv$ and $\ca$;
and
     \item \emph{f-gons}, whose edges are colored with $\cf$ and $\ca$.
 \end{enumerate}
 These three sets of bigons will correspond naturally to the edges,
vertices, and faces, respectively, of a cellularly embedded graph.

 Every gem $\cJ$ has an associated embedded graph $\eJ$ which is
obtained by taking a closed disk $D_C$ for each bigon $C$ in $J$ and
identifying the boundary of $D_C$ with $C$, forming a surface in which
$J$ is embedded. 
 Faces associated with e-squares, v-gons, and f-gons are called
\emph{e-faces}, \emph{v-faces} and \emph{f-faces}, respectively.
 So, although gems are purely combinatorial objects, we can regard them
as embedded graphs when convenient.

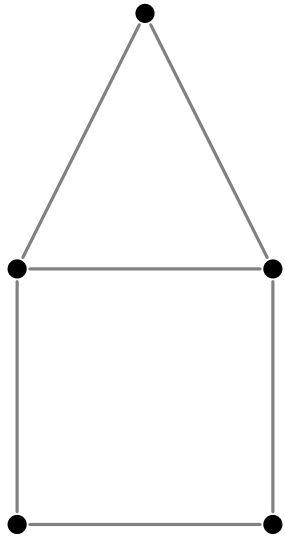
\begin{figure}

\begin{center}
   \begin{tikzpicture}[scale=0.17]
\node[vertex] at (0,0) (v0) {};
 \node[vertex] at (20,0) (v2) {};
 \node[vertex] at (0,20) (v6) {};
 \node[vertex] at (20,20) (v8) {};
 \node[vertex] at (10,40) (v13) {};

	\begin{scope}[nodes={sloped,allow upside down}][on background layer]
		\draw[edge] (v0) to (v2);
		\draw[edge] (v2) to (v8);
        \draw[edge] (v6) to (v8);
		\draw[edge] (v0) to (v6);
  	\draw[edge] (v13) to (v6);
		\draw[edge] (v13) to (v8);
	\end{scope}
\end{tikzpicture}
   \end{center}

  \caption{Embedded graph $\eG$ for which we will take the barycentric
subdivision and gem.}
   \label{fig:barycentricoriginalgraph}
 
\end{figure}
\hfill

\begin{figure}
 \tline{%
 \kern-20pt
   \begin{tikzpicture}[scale=0.17]
   
\node[vertex] at (0,0) (v0) {};
 \node[vertex] at (10,0) (v1) {};
 \node[vertex] at (20,0) (v2) {};
 \node[vertex] at (0,10) (v3) {};
 \node[vertex] at (10,10) (v4) {};
 \node[vertex] at (20,10) (v5) {};
 \node[vertex] at (0,20) (v6) {};
 \node[vertex] at (10,20) (v7) {};
 \node[vertex] at (20,20) (v8) {};
 \node[vertex] at (5,25) (v9) {};
 \node[vertex] at (10,24) (v10) {};
 \node[vertex] at (15,25) (v11) {};
  \node[vertex] at (10,30) (v12) {};
 \node[vertex] at (10,40) (v13) {};

	\begin{scope}[nodes={sloped,allow upside down}][on background layer]
		\draw[edge][red] (v0) to (v1);
		\draw[edge][red] (v1) to (v2);
        \draw[edge][red] (v0) to (v3);
		\draw[edge][red] (v2) to (v5);
  	\draw[edge][red] (v3) to (v6);
		\draw[edge][red] (v5) to (v8);
  	\draw[edge][red] (v6) to (v7);
		\draw[edge][red] (v7) to (v8);
        \draw[edge][blue] (v1) to (v4);
		\draw[edge][blue] (v3) to (v4);
  	\draw[edge][blue] (v4) to (v5);
		\draw[edge][blue] (v4) to (v7);
        \draw[edge][yellow!60!orange] (v0) to (v4);
		\draw[edge][yellow!60!orange] (v2) to (v4);
  	\draw[edge][yellow!60!orange] (v4) to (v6);
		\draw[edge][yellow!60!orange] (v4) to (v8);
  	\draw[edge][red] (v6) to (v9);
		\draw[edge][red] (v8) to (v11);
        \draw[edge][red] (v9) to (v12);
		\draw[edge][red] (v11) to (v12);
        \draw[edge][blue] (v7) to (v10);
		\draw[edge][blue] (v9) to (v10);
  	\draw[edge][blue] (v10) to (v11);
        \draw[edge][yellow!60!orange] (v6) to (v10);
		\draw[edge][yellow!60!orange] (v8) to (v10);
  	\draw[edge][yellow!60!orange] (v10) to (v12);
        \draw[edge, yellow!60!orange] (v0) .. controls (-15, 20) .. (v13);
         \draw[edge, yellow!60!orange] (v2) .. controls (35, 20) .. (v13);
          \draw[edge, blue] (v1) .. controls (5, -15) and (-40,15).. (v13);
        \draw[edge, blue] (v3) .. controls (-10, 20) .. (v13);
         \draw[edge, blue] (v5) .. controls (30, 20) .. (v13);
           \draw[edge, yellow!60!orange] (v6) .. controls (-5, 20) .. (v13);
         \draw[edge, yellow!60!orange] (v8) .. controls (25, 20) .. (v13);
           \draw[edge, blue] (v9) .. controls (3, 25) .. (v13);
         \draw[edge, blue] (v11) .. controls (17, 25) .. (v13);
           \draw[edge, yellow!60!orange] (v12) to (v13);
           
	\end{scope}
\end{tikzpicture}
\hfill
}
\vspace{-15mm}
  \caption{Barycentric subdivision of $\eG$.}
   \label{fig:barycentric}
\end{figure}
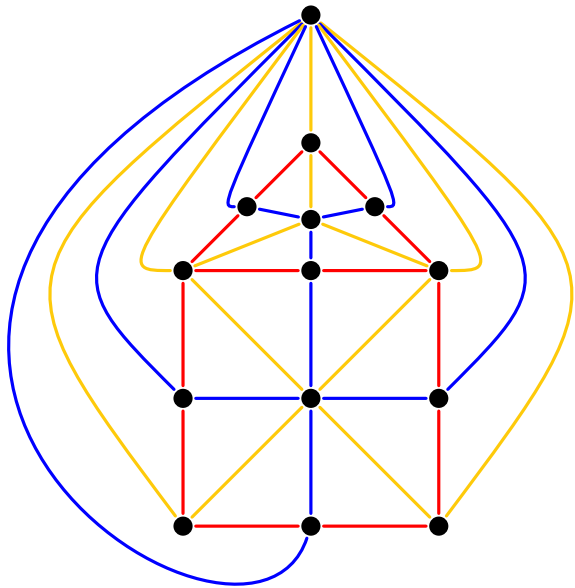
\hfill

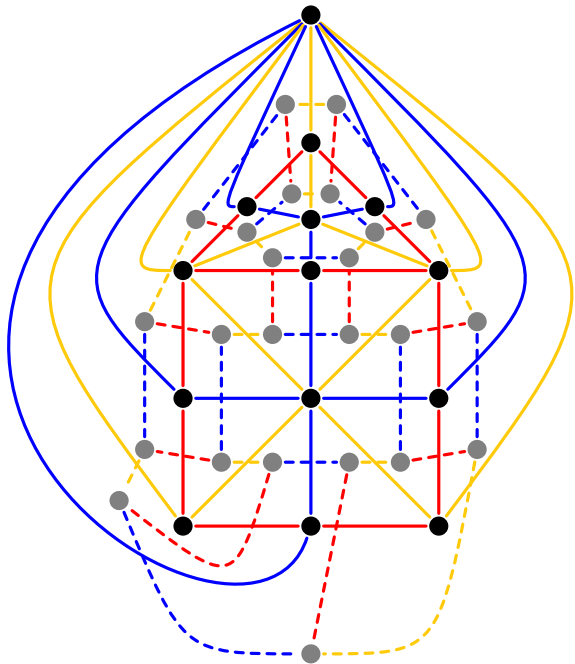
\begin{figure}
 \tline{\kern-20pt
   \begin{tikzpicture}[scale=0.17]
   
\node[vertex] at (0,0) (v0) {};
 \node[vertex] at (10,0) (v1) {};
 \node[vertex] at (20,0) (v2) {};
 \node[vertex] at (0,10) (v3) {};
 \node[vertex] at (10,10) (v4) {};
 \node[vertex] at (20,10) (v5) {};
 \node[vertex] at (0,20) (v6) {};
 \node[vertex] at (10,20) (v7) {};
 \node[vertex] at (20,20) (v8) {};
 \node[vertex] at (5,25) (v9) {};
 \node[vertex] at (10,24) (v10) {};
 \node[vertex] at (15,25) (v11) {};
  \node[vertex] at (10,30) (v12) {};
 \node[vertex] at (10,40) (v13) {};
 \node[vertex, gray] at (3,5) (dv1) {};
  \node[vertex, gray] at (7,5) (dv2) {};
   \node[vertex, gray] at (13,5) (dv3) {};
  \node[vertex, gray] at (17,5) (dv4) {};
 \node[vertex, gray] at (3,15) (dv5) {};
  \node[vertex, gray] at (7,15) (dv6) {};
   \node[vertex, gray] at (13,15) (dv7) {};
  \node[vertex, gray] at (17,15) (dv8) {};
   \node[vertex, gray] at (5,23) (dv9) {};
  \node[vertex, gray] at (7,21) (dv10) {};
   \node[vertex, gray] at (13,21) (dv11) {};
  \node[vertex, gray] at (15,23) (dv12) {};
     \node[vertex, gray] at (8,33) (dv13) {};
  \node[vertex, gray] at (8.5,26) (dv14) {};
   \node[vertex, gray] at (12,33) (dv15) {};
  \node[vertex, gray] at (11.5,26) (dv16) {};
\node[vertex, gray] at (1,24) (dv17) {};
  \node[vertex, gray] at (19,24) (dv18) {};
   \node[vertex, gray] at (-3,16) (dv19) {};
  \node[vertex, gray] at (23,16) (dv20) {};
  \node[vertex, gray] at (-3,6) (dv21) {};
  \node[vertex, gray] at (23,6) (dv22) {};
   \node[vertex, gray] at (-5,2) (dv23) {};
  \node[vertex, gray] at (10,-10) (dv24) {};

  	\begin{scope}[nodes={sloped,allow upside down}][on background layer]
\draw[edge, dashed, yellow!60!orange] (dv1) to (dv2);
\draw[edge, dashed, blue] (dv2) to (dv3);
\draw[edge, dashed, yellow!60!orange] (dv3) to (dv4);
\draw[edge, dashed, yellow!60!orange] (dv5) to (dv6);
\draw[edge, dashed, blue] (dv6) to (dv7);
\draw[edge, dashed, yellow!60!orange] (dv7) to (dv8);
\draw[edge, dashed, yellow!60!orange] (dv9) to (dv10);
\draw[edge, dashed, blue] (dv10) to (dv11);
\draw[edge, dashed, yellow!60!orange] (dv11) to (dv12);
\draw[edge, dashed, red] (dv13) to (dv14);
\draw[edge, dashed, yellow!60!orange] (dv14) to (dv16);
\draw[edge, dashed, red] (dv15) to (dv16);
\draw[edge, dashed, blue] (dv1) to (dv5);
\draw[edge, dashed, red] (dv1) to (dv21);
\draw[edge, dashed, red] (dv2) .. controls (4,-5) .. (dv23);
\draw[edge, dashed, yellow!60!orange] (dv21) to (dv23);
\draw[edge, dashed, blue] (dv23) .. controls (0, -10) .. (dv24);
\draw[edge, dashed, red] (dv4) to (dv22);
\draw[edge, dashed, yellow!60!orange] (dv22) .. controls (20,-10) .. (dv24);
\draw[edge, dashed, red] (dv3) to (dv24);
\draw[edge, dashed, yellow!60!orange] (dv18) to (dv20);
\draw[edge, dashed, red] (dv8) to (dv20);
\draw[edge, dashed, blue] (dv20) to (dv22);
\draw[edge, dashed, blue] (dv4) to (dv8);
\draw[edge, dashed, red] (dv6) to (dv10);
\draw[edge, dashed, red] (dv7) to (dv11);
\draw[edge, dashed, yellow!60!orange] (dv19) to (dv17);
\draw[edge, dashed, blue] (dv19) to (dv21);
\draw[edge, dashed, red] (dv17) to (dv9);
\draw[edge, dashed, blue] (dv17) to (dv13);
\draw[edge, dashed, yellow!60!orange] (dv15) to (dv13);
\draw[edge, dashed, red] (dv19) to (dv5);
\draw[edge, dashed, blue] (dv9) to (dv14);
\draw[edge, dashed, blue] (dv12) to (dv16);
\draw[edge, dashed, blue] (dv15) to (dv18);
\draw[edge, dashed, red] (dv12) to (dv18);
	\end{scope}
	\begin{scope}[nodes={sloped,allow upside down}][on background layer]
		\draw[edge][red] (v0) to (v1);
		\draw[edge][red] (v1) to (v2);
        \draw[edge][red] (v0) to (v3);
		\draw[edge][red] (v2) to (v5);
  	\draw[edge][red] (v3) to (v6);
		\draw[edge][red] (v5) to (v8);
  	\draw[edge][red] (v6) to (v7);
		\draw[edge][red] (v7) to (v8);
        \draw[edge][blue] (v1) to (v4);
		\draw[edge][blue] (v3) to (v4);
  	\draw[edge][blue] (v4) to (v5);
		\draw[edge][blue] (v4) to (v7);
        \draw[edge][yellow!60!orange] (v0) to (v4);
		\draw[edge][yellow!60!orange] (v2) to (v4);
  	\draw[edge][yellow!60!orange] (v4) to (v6);
		\draw[edge][yellow!60!orange] (v4) to (v8);
  	\draw[edge][red] (v6) to (v9);
		\draw[edge][red] (v8) to (v11);
        \draw[edge][red] (v9) to (v12);
		\draw[edge][red] (v11) to (v12);
        \draw[edge][blue] (v7) to (v10);
		\draw[edge][blue] (v9) to (v10);
  	\draw[edge][blue] (v10) to (v11);
        \draw[edge][yellow!60!orange] (v6) to (v10);
		\draw[edge][yellow!60!orange] (v8) to (v10);
  	\draw[edge][yellow!60!orange] (v10) to (v12);
        \draw[edge, yellow!60!orange] (v0) .. controls (-15, 20) .. (v13);
         \draw[edge, yellow!60!orange] (v2) .. controls (35, 20) .. (v13);
          \draw[edge, blue] (v1) .. controls (5, -15) and (-40,15).. (v13);
        \draw[edge, blue] (v3) .. controls (-10, 20) .. (v13);
         \draw[edge, blue] (v5) .. controls (30, 20) .. (v13);
           \draw[edge, yellow!60!orange] (v6) .. controls (-5, 20) .. (v13);
         \draw[edge, yellow!60!orange] (v8) .. controls (25, 20) .. (v13);
           \draw[edge, blue] (v9) .. controls (3, 25) .. (v13);
         \draw[edge, blue] (v11) .. controls (17, 25) .. (v13);
           \draw[edge, yellow!60!orange] (v12) to (v13);
           
	\end{scope}
\end{tikzpicture}
\hfill
 }
 \vspace{-5mm}
  \caption{The barycentric subdivision (solid) and gem (dashed) for $\eG$.}
   \label{fig:barycentricwithgem}
\end{figure}
\hfill

\begin{figure}
\centering

  \begin{center}
      
   \begin{tikzpicture}[scale=0.17]

 \node[vertex, gray] at (3,5) (dv1) {};
  \node[vertex, gray] at (7,5) (dv2) {};
   \node[vertex, gray] at (13,5) (dv3) {};
  \node[vertex, gray] at (17,5) (dv4) {};
 \node[vertex, gray] at (3,15) (dv5) {};
  \node[vertex, gray] at (7,15) (dv6) {};
   \node[vertex, gray] at (13,15) (dv7) {};
  \node[vertex, gray] at (17,15) (dv8) {};
   \node[vertex, gray] at (5,23) (dv9) {};
  \node[vertex, gray] at (7,21) (dv10) {};
   \node[vertex, gray] at (13,21) (dv11) {};
  \node[vertex, gray] at (15,23) (dv12) {};
     \node[vertex, gray] at (8,33) (dv13) {};
  \node[vertex, gray] at (8.5,26) (dv14) {};
   \node[vertex, gray] at (12,33) (dv15) {};
  \node[vertex, gray] at (11.5,26) (dv16) {};
\node[vertex, gray] at (1,24) (dv17) {};
  \node[vertex, gray] at (19,24) (dv18) {};
   \node[vertex, gray] at (-3,16) (dv19) {};
  \node[vertex, gray] at (23,16) (dv20) {};
  \node[vertex, gray] at (-3,6) (dv21) {};
  \node[vertex, gray] at (23,6) (dv22) {};
   \node[vertex, gray] at (-5,2) (dv23) {};
  \node[vertex, gray] at (10,-10) (dv24) {};

  	\begin{scope}[nodes={sloped,allow upside down}][on background layer]
\draw[edge, dashed, yellow!60!orange] (dv1) to (dv2);
\draw[edge, dashed, blue] (dv2) to (dv3);
\draw[edge, dashed, yellow!60!orange] (dv3) to (dv4);
\draw[edge, dashed, yellow!60!orange] (dv5) to (dv6);
\draw[edge, dashed, blue] (dv6) to (dv7);
\draw[edge, dashed, yellow!60!orange] (dv7) to (dv8);
\draw[edge, dashed, yellow!60!orange] (dv9) to (dv10);
\draw[edge, dashed, blue] (dv10) to (dv11);
\draw[edge, dashed, yellow!60!orange] (dv11) to (dv12);
\draw[edge, dashed, red] (dv13) to (dv14);
\draw[edge, dashed, yellow!60!orange] (dv14) to (dv16);
\draw[edge, dashed, red] (dv15) to (dv16);
\draw[edge, dashed, blue] (dv1) to (dv5);
\draw[edge, dashed, red] (dv1) to (dv21);
\draw[edge, dashed, red] (dv2) .. controls (4,-5) .. (dv23);
\draw[edge, dashed, yellow!60!orange] (dv21) to (dv23);
\draw[edge, dashed, blue] (dv23) .. controls (0, -10) .. (dv24);
\draw[edge, dashed, red] (dv4) to (dv22);
\draw[edge, dashed, yellow!60!orange] (dv22) .. controls (20,-10) .. (dv24);
\draw[edge, dashed, red] (dv3) to (dv24);
\draw[edge, dashed, yellow!60!orange] (dv18) to (dv20);
\draw[edge, dashed, red] (dv8) to (dv20);
\draw[edge, dashed, blue] (dv20) to (dv22);
\draw[edge, dashed, blue] (dv4) to (dv8);
\draw[edge, dashed, red] (dv6) to (dv10);
\draw[edge, dashed, red] (dv7) to (dv11);
\draw[edge, dashed, yellow!60!orange] (dv19) to (dv17);
\draw[edge, dashed, blue] (dv19) to (dv21);
\draw[edge, dashed, red] (dv17) to (dv9);
\draw[edge, dashed, blue] (dv17) to (dv13);
\draw[edge, dashed, yellow!60!orange] (dv15) to (dv13);
\draw[edge, dashed, red] (dv19) to (dv5);
\draw[edge, dashed, blue] (dv9) to (dv14);
\draw[edge, dashed, blue] (dv12) to (dv16);
\draw[edge, dashed, blue] (dv15) to (dv18);
\draw[edge, dashed, red] (dv12) to (dv18);
	\end{scope}

\end{tikzpicture}
\end{center}
  \caption{The gem of $\eG$.}
   \label{fig:gemonly}
   
\end{figure}
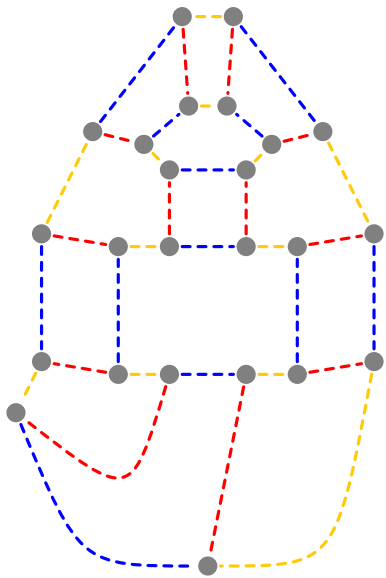

 To give the correspondence between embedded graphs and gems, we first
define the \emph{barycentric subdivision} $\cB$ of an embedded graph
$\eG$ without isolated vertex components.
 Subdivide each edge $e$ of $\eG$ with a new vertex $v_e$, and color all
edges of the resulting embedded graph $\eG'$ with $\cv$.
 Now add a new vertex $v_f$ inside each face $f$ of $\eG'$.
 Join $v_f$ to each vertex of $\eG$ on the boundary of $f$ with an edge
of color $\ca$, and to each vertex $v_e$ on the boundary of $f$ with an
edge of color $\cf$.
 The resulting embedded graph with colored edges is $\cB$, which is
embedded in the same surfaces as $\eG$.
 In Figure \ref{fig:barycentricoriginalgraph} we show a planar embedded graph $\eG$, and
in Figure \ref{fig:barycentric} we show its barycentric subdivision.

 We can identify the origin of each vertex of $\cB$ with incident edges,
calling it \emph{type {\rm V}, {\rm F}, or {\rm E}}, depending on
whether its incident edges have colors in $\{\cv, \ca\}$, $\{\cf,
\ca\}$, or $\{\cv, \cf\}$, respectively.
 The faces of $\cB$ are the \emph{flags} of $\eG$, and each flag is
bounded by a triangle whose edges use all three colors.
 We can recognize $\eG$ inside $\cB$, formed by vertices of type V
joined by edges of color $\cv$ subdivided by vertices of type E.  We can
also recognize $\eG\sdu$, formed by vertices of type F joined by edges
of color $\cf$ subdivided by vertices of type E.
 The barycentric subdivision is an embedded graph $\eB$ with added edge
colors, and the embedding is important.  The combinatorial information
(graph $B$ and edge coloring) can be the same for different embeddings
of the same graph $G$, so is not sufficient to determine $\eG$.

 Now we can show that every embedded graph without isolated vertex
components has a corresponding gem $\cJ$, and vice versa.  The gem $\cJ$
of $\eG$ is the colored dual $\cB\sdu$ of $\cB$.  Like $\cB$, this is
embedded in the same surface as $\eG$.
 Thus, each vertex of $\cJ$ represents a flag of $\eG$ (so that $\cJ$ is
sometimes called the \emph{flag graph} of $\eG$), and is incident with
three edges of different colors. 
 The v-gons, f-gons, and e-squares of $\cJ$ correspond to the vertices,
faces, and edges, respectively, of $\eG$.
 Conversely, if we have a gem $\cJ$, we can take its colored dual to
obtain $\cB$ and then recover $\eG$ from $\cB$.
 In Figure \ref{fig:barycentricwithgem} we illustrate $\cB$ and $\cJ$
together for the embedding of Figure \ref{fig:barycentricoriginalgraph},
and in Figure \ref{fig:gemonly} we show just $\cJ$.

 A \emph{jewel} is a slightly extended form of a gem, where we add edges
of a new color $\cz$, associated with the zigzags of $\eG$.
 Again following the conventions of \cite{BoLi95}, we use the color
green for $\cz$ in figures.

 A jewel $\cL$ consists of a $4$-regular graph $L$ with a proper
$4$-edge-coloring $\ga : E(L) \to \{\cv, \cf, \cz, \ca\}$, such that
each component of the $3$-factor induced by edges of colors $\cv$, $\cf$
and $\cz$ is a complete graph $K_4$, which we call an \emph{e-simplex},
and which contains an e-square consisting of its edges colored $\cv$ and
$\cf$.
 Jewels have v-gons and f-gons, defined in the same way as
for gems, but also additional bigons whose edges are alternately colored
$\ca$ and $\cz$, which we call \emph{z-gons}, and which correspond to
the zigzags of $\eG$.
 Unlike gems, a jewel $\cL$ does not provide an embedding of
its underlying graph $L$ (so $\eg{L}$ is not defined).

 Each jewel $\cL$ has a corresponding gem $\cJ$ obtained by deleting the
edges of color $\cz$, while we can obtain $\cL$ from $\cJ$ by adding an
edge of color $\cz$ joining each pair of vertices that are diagonally
opposite in an e-square of $\cJ$.  It follows that there is also a
correspondence between embedded graphs $\eG$ without isolated vertex
components and jewels $\cL$.

 \def\cc{\mskip0.5\thinmuskip}
 \def\ec#1#2{E^{\cc #1}_{#2}}

 If we have an edge-colored structure $\cH$ (barycentric subdivision,
gem, jewel) with edge-coloring $\ga$, then for a color $c \in \{\cv,
\cf, \cz, \ca\}$ we use $\ec{c}\cH$ to denote $\ga\iv(c)$, the set of
edges colored $c$ in $\cH$.
 We will write $\ec\cv\cH \cup \ec\cf\cH$ as $\ec{\cv\cc\cf}\cH$ and
make other similar notational abbreviations.

 A \emph{medial checkerboard} $\cM$ is a $4$-regular graph embedding
$\eM$ with a proper $2$-face-coloring $\ga: F(\eM) \to \{\cv, \cf\}$.
 Every embedded graph $\eG$ without isolated vertex components has an
associated medial checkerboard $\cM$, and vice versa.
 Given an embedded graph $\eG$, form the barycentric subdivision $\cB$,
and let $\eB'$ be the embedded subgraph induced by all edges colored
$\ca$.

 The vertices of $\eB'$ are the vertices of $\cB$
of types V and F, and each face of $\eB'$ has degree $4$ and its
interior contains a unique vertex of $\cB$ of type E.  Let $\eM$ be the
dual of $\eB'$, using the vertices of type E in $\cB$ as its vertices. 
 Color each face of $\eM$ with $\cv$ or $\cf$ depending on
whether it contains a vertex of $\cB$ of type V or F, respectively.  The
result is a medial checkerboard $\cM$.
 Note that $\eB'$, and hence $\cM$, is embedded in the same surface as
$\eG$.
 We can also think of $\cM$ as obtained from $\cB$ by adding colors
$\cv$ and $\cf$ to vertices of types V and F respectively, removing all
vertices of type E (and their incident edges of colors $\cv$ and $\cf$),
uncoloring edges of color $\ca$, and then taking the colored dual.

 Conversely, given a medial checkerboard $\cM$, we can take the dual of
$\eM$ and use the color information from $\cM$ to add back the remaining
vertices, edges and colors of $\cB$; from $\cB$ we can recover $\eG$.

 In the literature the term ``medial graph'' may refer to the
face-colored embedding $\cM$, the embedding $\eM$, or just the graph
$M$.
 To remove this ambiguity, we use \emph{medial graph} for $M$,
\emph{medial embedding} for $\eM$, and \emph{medial checkerboard} for
$\cM$.
 Also, the standard convention is to use the color black for faces of
$\cM$ corresponding to vertices of $\eG$, and white for those
corresponding to faces of $\eG$ (see for example \citeemm.).  However,
we use $\cv$ (red) and $\cf$ (blue) for consistency with the colors in
gems and jewels.

 Note that $\cM$ may also be obtained from the gem $\cJ$ by coloring the
v- and f-faces with $\cv$ and $\cf$, respectively, and contracting each
e-face to a single vertex.

 At this point we have developed correspondences between embedded graphs
$\eG$ without isolated vertex components, gems $\cJ$, jewels $\cL$, and
medial checkerboards $\cM$ (and also barycentric subdivisions $\cB$).
 For a given finite set $E$, the following are in one-to-one
correspondence for a given finite set $E$: embedded graphs $\eG$ with
edge set $E$ and no isolated vertex components (up to vertex relabeling
and surface homeomorphism), gems $\cJ$ with e-squares labeled by $E$ (up
to color- and label-preserving isomorphism), jewels $\cL$ with
e-simplices labeled by $E$ (up to color- and label-preserving
isomorphism), and medial checkerboards $\cM$ with vertices labeled by
$E$ (up to color- and label-preserving isomorphism and surface
homeomorphism).

 Finally, we address isolated vertex components.
 If $\eG$ has isolated vertex components, we will just ignore them and
use the other components of $\eG$ to form the corresponding barycentric
subdivision $\cB$, gem $\cJ$, jewel $\cL$, or medial checkerboard $\cM$.
 This means that we lose some information, and cannot uniquely
reconstruct $\eG$ from these other structures.
 We also lose the property that $\cB$, $\cJ$, $\cL$, and $\cM$ are
embedded in the same surface as $\eG$.
 There are ways we could encode isolated vertex components in $\cB$,
$\cJ$, $\cL$, or $\cM$:  for example, in \cite{GraphsonSurfEMM} medial
checkerboards $\cM$ can contain gadgets called `free loops' representing
isolated vertex components.
 However, as we verify later in appropriate places, the presence or
absence of isolated vertex components does not affect the properties we
consider in this paper, so it is simpler to just disregard such
components.

 \subsection{Twisted duality}\label{td}

 Here we discuss duality operations, as investigated by Wilson
\cite{Wilson} and Lins \cite{gemslins}, and their extensions to twisted
duality as defined by Ellis-Monaghan and Moffatt \citeemm., in terms of
gems and jewels.  We again follow \cite{BoLi95, ElZh17, gemslins}.

 In our work we will be applying operations such as duality or partial
duality, which may completely rearrange vertices and faces, but
essentially preserve the edge set of an embedded graph.  For example,
when we take the dual $\eG\sdu$, the primal edge $e$ of $\eG$ and the
dual edge $e\sdu$ of $\eG\sdu$ are different topological curves, but
combinatorially we may consider them to be the same edge $e$, with
different incidences in $\eG$ and in $\eG\sdu$, so that the underlying
graphs of $\eG$ and $\eG\sdu$ have the same edge sets.

 The dual $\eG\sdu$ of a cellularly embedded graph $\eG$ without
isolated vertex components can easily be described in terms of gems.
 First we move from $\eG$ to its associated gem $\cJ$. As observed by
Lins \cite{gemslins}, the gem $\cJ\edu$ corresponding to $\eG\sdu$ is
formed by swapping the colors $\cv$ and $\cf$ everywhere in $\cJ$.
 This interchanges the v-gons and f-gons in the gem which cause vertices
and faces to swap in the corresponding cellularly embedded graph.
 The dual also corresponds to swapping $\cv$ and $\cf$ in the jewel
$\cL$ to give a new jewel $\cL\edu$.

 Note that $\cJ\edu$ is different from $\cJ\sdu$ (which is $\cB$), and
$\cL\edu$ exists while $\cL\sdu$ does not ($\cL$ has no
associated embedded graph).  So our notation needs to distinguish between
applying an operator such as duality directly to a colored structure
$\cH$, or applying it to $\eG$ and making a corresponding change to
$\cH$.
 In general, when we apply an operation, such as duality, to an embedded
graph $\eG$, the corresponding operation on the associated colored
structures $\cB, \cJ, \cL, \cM$ will be denoted by enclosing the
operation in brackets $[\cdot]$.

 We can also define partial duality using gems (or jewels) \cite{ChVT22,
ElZh17}.
 Suppose $\eG$ is an embedded graph without isolated vertex components,
and with corresponding gem $\cJ$ (or
jewel $\cL$). The \emph{partial dual} $\eG \du A$ of $\eG$ with respect
to $A \subseteq E(G)$ can be obtained by swapping the colors $\cv$ and
$\cf$ on the e-squares of $\cJ$ (or the e-simplices of $\cL$)
corresponding to $A$, to give the gem $\cJ [{}\du A]$ (or jewel $\cL
[{}\du A]$) that corresponds to $\eG \du A$.
 This is equivalent to the original definition of partial duality given
by Chmutov \cite{GenDualityChmutov}, using arrow-marked ribbon graphs.
 Partial duals $\eG \du A$ generally have different underlying graphs,
but they all have the same edge set, $E(G)$.
 Note that $\eG\sdu = \eG \du E(G)$.

 The Petrie dual of an embedded graph $\eG$ without isolated vertex
components can also be interpreted very simply in terms of the
corresponding jewel $\cL$.
 The jewel $\cL\epe$ corresponding to $\eG\spe$ is obtained by swapping
the colors $\cf$ and $\cz$ everywhere in $\cL$, which exchanges the
f-gons and z-gons in $\cL$ and hence exchanges the faces and zigzags in
$\eG$.
 We can also interpret the Petrie dual in terms of the gem $\cJ$: remove
all edges colored $\cf$ and then add new edges colored $\cf$ joining
opposite corners of each e-square of $\cJ$ to obtain $\cJ\epe$. 
 Note that the underlying graph of $\cJ\epe$ is not $J$.

 We can also form the \emph{partial Petrie dual} $\eG \pe A$ with
respect to $A \subseteq E(G)$ by twisting (flipping the signatures of)
only the edges in $A$.
 The jewel $\cL[{} \pe A]$ corresponding to $\eG \pe A$ is obtained by
swapping colors $\cf$ and $\cz$ just in the e-simplices of $\cL$ that
correspond to $A$, and the corresponding gem $\cJ[{} \pe A]$ is obtained
by moving the edges of color $\cf$ just in the e-squares of $\cJ$ that
correspond to $A$.
 Partial Petrie duals $\eG \pe A$ all have the same underlying graph $G$
and hence the same edge set $E(G)$.
 Note that $\eG\spe = \eG \pe E(G)$.

 For embedded graphs $\eG$ with isolated vertex components, we can
implement all of the above operations (duality, partial duality, Petrie
duality, partial Petrie duality) by applying them to the other
components of $\eG$, leaving the isolated vertex components unchanged.

 We can apply sequences of duals and Petrie duals, which we compose left
to right, so that, for example, $\eG\sdu\spe = (\eG\sdu)\spe$.
 Since the dual and Petrie dual just permute the three colors $\cv, \cf,
\cz$ in a jewel $\cL$, they generate the $S_3$-action discovered by
Wilson \cite{Wilson} and Lins \cite{gemslins}, with the relation that
$\cL\edpd = \cL\epdp$, so that $\eG\sdu\spe\sdu =
\eG\spe\sdu\spe$ for embedded graphs $\eG$.  Thus, there are at most six
distinct embedded graphs that can be obtained from $\eG$ in this way:
$\eG$, $\eG\sdu$, $\eG\spe$, $\eG\sdu\spe$, $\eG\spe\sdu$ and
$\eG\sdu\spe\sdu = \eG\spe\sdu\spe$.

 We can also compose partial duals and partial Petrie duals, to produce,
for example, $\eG\du A \pe B \du C = ((\eG\du A)\pe B) \du C$ for $A, B,
C \subseteq E(G)$. 
 We use abbreviated notation when applying several operations to the
same set of edges $A$, such as $\eG \pe\du A$ to mean $G \pe A \du
A$.  Again the operations are composed left to right.
 A composition of partial dual and partial Petrie dual
operations is a \emph{twisted dual operation}, and an embedded graph
obtained by applying a twisted dual operation to $\eG$ is a
\emph{twisted dual} of $\eG$.
 Ellis-Monaghan and Moffatt \citeemm. developed the idea of twisted
duals and showed that twisted duals give an
action of $S_3^{E(G)}$ on the embedded graphs derived from $\eG$.
 This action can be interpreted as permuting the colors $\cv, \cf, \cz$
independently in each e-simplex of the jewel $\cL$, and extends the
$S_3$-action generated by geometric and Petrie duality.

 The \emph{partial Wilson dual} of $\eG$ with respect to $A$ is defined
as $\eG \wi A = \eG \du\pe\du A$, which is also equal to $\eG \pe\du\pe A$.  The partial Wilson dual corresponds to swapping colors
$\cv$ and $\cz$ in the e-simplices corresponding to $A$ in the jewel
$\cL$.  The \emph{Wilson dual} of $\eG$ is $\eG \swi = \eG \wi E(G) =
\eG \sdu \spe \sdu = \eG \spe \sdu \spe$, which corresponds to swapping
$\cv$ and $\cz$ everywhere in $\cL$ to give $\cL\ewi$.

 Our notation here differs a little from what is generally used, but is
more convenient for our purposes.  The standard notation introduced by
Chmutov \cite{GenDualityChmutov} for the partial dual $\eG \du A$ is
$\eG^A$, but this becomes awkward when we need to take partial Petrie
duals as well as partial duals.  Ellis-Monaghan and Moffatt \citeemm.
use $\delta$ to denote partial dual and $\tau$ to denote partial Petrie
dual, so they would write $\eG^{\delta(A)\tau(B)\delta(C)}$ for our $\eG
\du A \pe B \du C$.  Their operations $\delta(A)$ and $\tau(B)$ compose
left to right, but $\delta$ and $\tau$ by themselves compose right to
left, so that $\eG^{\delta\tau(A)} = (\eG^{\tau(A)})^{\delta(A)}$, which
is our $\eG\pe\du A = \eG\pe A \du A$.  In our notation twisted duality
operations always compose left to right.
 Our notation is
similar to that typically used for operations on delta-matroids.

 \section{Parity Conditions for Closed Walks in Gems and Jewels}
 \label{sec:parityconditions}

 \subsection{Generalized bipartiteness}

 Our results are based on a generalization of the idea of bipartiteness
in graphs.  One condition equivalent to bipartiteness is that all closed
walks have even length.  We will extend this by choosing a set of edges,
and investigating the condition that all closed walks contain an even
number of edges from this set.  We connect this with binary vertex
labelings of $G$ and with graphs obtained by contracting sets of edges
in $G$.

 \let\num\nu
 A \emph{$v_0v_\ell$-walk} $W$ in a graph $G$ is a sequence of vertices
and edges of $G$, $W = v_0 e_1 v_1 e_2 v_2 \dots v_{\ell-1} e_\ell
v_\ell$, where the incident vertices of each $e_i$ are $v_{i-1}$
and $v_i$; $\ell$ is the \emph{length} of $W$.
 If $S \subseteq E(G)$, define $\num_S(W)$ to be the number of $i$ for
which $e_i \in S$, i.e., the number of edges of $W$ that belong to $S$,
counted with their multiplicity in $W$.

 We say $W$ is \emph{trivial} if $\ell=0$, and \emph{closed} if $v_0 =
v_\ell$.  We often use $K$ to denote a closed walk.
 The following fundamental fact is well known.

 \begin{lemma}\label{oddwalkoddcyclelemma}
 In a graph $G$, a nontrivial closed walk $K =
v_0e_1v_1e_2v_2...e_\ell(v_\ell=v_0)$ with no internal repeated vertex
($v_i\neq v_j$ for $0\leq i < j < \ell$) is either a cycle or the walk
$v_0 e_1 v_1 e_1 v_0$.
 \end{lemma}

 By repeatedly breaking a closed walk into two shorter closed walks at
an internal repeated vertex, and applying Lemma
\ref{oddwalkoddcyclelemma}, we obtain the following.

 \begin{lemma}\label{oddwalkoddcyclelemma2}
 Let $G$ be a graph, let $S\subseteq E(G)$, and let $K$ be a closed walk
in $G$.  If $\num_S(K)$ is odd then there is a cycle $C$ that is a (not
necessarily consecutive) subsequence of $K$ for which $\num_S(C)$ is
also odd. 
 \end{lemma}

 A \emph{binary labeling} of a graph $G$ is a function $\la : V(G) \to
\mZ_2$.  If $S \subseteq E(G)$, then  an \emph{$S$-labeling} of $G$ is a
binary labeling such that for every edge $e$ with ends $u$ and $v$,
$\la(u) \ne \la(v)$ if $e \in S$ and $\la(u) = \la(v)$ if $e \notin S$,
or in other words, $\la(v) = (\la(u) + \num_S(uev)) \bmod 2$.
 For $u \in V(G)$ we say $\la$ is an \emph{$S$-labeling relative to $u$}
if $\la(u) = 0$.
 The following observations are easy but useful.  

 \begin{observation}\label{binlab}
 Let $G$ be a graph and $u, w \in V(G)$.
 \begin{enumerate}[label=\rm(\alph*)]\setlength{\itemsep}{0pt}
 \item\label{blS}
  Every binary labeling $\la$ of $G$ is an $S$-labeling for a unique $S
\subseteq E(G)$, namely $S = \{e \in E(G) \;|$ $\la(u) \ne \la(w)$ where
$u$ and $w$ are the endvertices of $e\}$.
 \item\label{blW}
 If $\la$ is an $S$-labeling of $G$ and $W$ is a
$uw$-walk in $G$, then $\la(w) = (\la(u) + \num_S(W)) \bmod 2$.
 \item\label{blU}
 An $S$-labeling of $G$ is unique up to flipping values (swapping $0$
and $1$) on components of $G$.
 \end{enumerate}
 \end{observation}

 If $U \subseteq V(G)$ then $S$ is the \emph{coboundary of $U$ in $G$}
if $S$ is the set of edges of $G$ joining $U$ to $V(G)-U$.

 The graph $G \cg S$ is obtained from $G$ by \emph{contracting $S$ in
$G$}, by which we mean identifying vertices connected by edges of $S$,
and then deleting the edges of $S$.
 Contracting $S$ has the same result as contracting the edges in $S$
individually, in any order.  Contracting a loop is the same as deleting
it.  Every edge in $E(G)-S$ remains an edge of $G \cg S$, although
non-loops of $G$ may become loops of $G \cg S$.
 Contracting a set of edges $S$ in a graph $G$ does not correspond
exactly to contracting $S$ in an embedded graph $\eG$ to form $\eG\ct S$
(see the definition in \cite[Section 4.2]{GraphsonSurfEMM}), so we use
different notation.  However, if $S$ does not induce any noncontractible
cycles in $\eG$, then the underlying graph of $\eG\ct S$ is $G \cg S$.

 \let\implies\Rightarrow
 We can now define a generalization of bipartiteness via several simple
equivalent conditions, stated in Theorem \ref{oddwalkoddcycletheorem}.
 Conditions \ref{sbipK} and \ref{sbipW} are easily seen to be equivalent,
 \ref{sbipK} and \ref{sbipC} are equivalent by Lemma
\ref{oddwalkoddcyclelemma2}, and it is straightforward to prove that
\ref{sbipK} $\implies$ \ref{sbiplabel} $\implies$ \ref{sbipcoboundary}
$\implies$ \ref{sbipcontract} $\implies$ \ref{sbipK}, so we omit the
details.

 \begin{theorem}\label{oddwalkoddcycletheorem}
 Let $G$ be a graph with $S\subseteq E(G)$.
 Then the following are equivalent:
 \begin{enumerate}[label=\rm(\alph*)]\setlength{\itemsep}{0pt}
     \item\label{sbipK}
 	$\num_S(K)$ is even for every closed walk $K$ in $G$.
     \item\label{sbipW}
 	For all $u, w \in V(G)$, $\num_S(W)$ has the same parity for all
$uw$-walks $W$ in $G$.
     \item\label{sbipC}
 	$\num_S(C)$ is even for every cycle $C$ in $G$.
     \item\label{sbiplabel}
 	$G$ has an $S$-labeling.
     \item\label{sbipcoboundary}
 	$S$ is a coboundary in $G$.
     \item\label{sbipcontract}
 	$G\cg\overline{S}$ is bipartite, where $\overline{S}=E(G)-S$.
 \end{enumerate}
 \end{theorem}

 If $G$ and $S$ satisfy any of the conditions in Theorem
\ref{oddwalkoddcycletheorem}, we say that $G$ is \emph{$S$-bipartite}. 
 Note that $E(G)$-bipartiteness is just bipartiteness, and we refer to
an $E(G)$-labeling, with the two vertex classes of the bipartition
labeled $0$ and $1$, as a \emph{bipartite labeling} of $G$.

 \subsection{Main idea of this paper}

 At this point we can summarize our fundamental approach.
 The main idea of this paper is to translate properties of embedded
graphs into parity conditions for closed walks in gems or jewels, of the
form of Theorem \ref{oddwalkoddcycletheorem}\ref{sbipK}.
 These conditions can then be combined using addition.
 We use other $S$-bipartiteness conditions as appropriate.
 We apply Theorem \ref{oddwalkoddcycletheorem} frequently, usually
without explicit reference.

 The sets of edges $S$ that we will be concerned with are the edges of
prescribed colors in a gem or jewel.
 We simplify the notation for the function for counting the number of
edges colored $\cv$ in a walk in a jewel $\cL$, namely
$\num_{\ec\cv\cL}$, to $v_\cL$. 
 Similarly, functions $f_\cL$, $z_\cL$, and $a_\cL$ count the number of
edges colored $\cf$, $\cz$ and $\ca$, respectively.
 We use the same notation for a gem $\cJ$, except that we do not have
$z_\cJ$.
 We will write $v_\cL(K) +f_\cL(K)$ as $(v+f)_\cL(K)$, and use other
similar abbreviations.

 \subsection{Some known results}

 Before discussing our full Fano Framework, we illustrate our main idea
by re-proving two known results.  We begin with the following.

 \begin{proposition}[Wilson \cite{Wilson}, Ellis-Monaghan and Moffatt
\citeemm.]\label{oribipori}
 Let $\eG$ be an orientable graph embedding.  Then $\eG$ is bipartite if
and only if $\eG\spe$ is orientable.
 \end{proposition}

 Wilson \cite[Theorem 3]{Wilson} stated Theorem \ref{oribipori} without
proof for regular maps (graph embeddings with certain symmetry
properties).  Ellis-Monaghan and Moffatt, \cite[Proposition
4.30(2)]{twisteddualEMM} and \cite[Proposition 3.27]{GraphsonSurfEMM},
gave an incorrect statement of Wilson's result (corrected in the errata
for \cite{GraphsonSurfEMM}, available from Moffatt's website), but a
correct proof that applies to all graph embeddings, not just regular
maps.

 Applying Proposition \ref{oribipori} with $\eG\spe$ replacing $\eG$
shows that if $\eG\spe$ is orientable and $\eG\spe$ is bipartite
(equivalent to $\eG$ being bipartite), then $\eG$ is orientable.
 Combining this with Proposition \ref{oribipori} gives the following
stronger statement.

 \begin{theorem}\label{twoimplythethird}
 Let $\eG$ be an embedded graph.  Any two of the following properties
imply the third:
 \begin{enumerate}[label=\rm(\alph*)]\setlength{\itemsep}{0pt}
     \item\label{Gori}
 	$\eG$ is orientable.
     \item\label{Gbip}
 	$\eG$ is bipartite (or equivalently $\eG\spe$ is bipartite).
     \item\label{Gpeori}
 	$\eG\spe$ is orientable.
 \end{enumerate}
 \end{theorem}

 Theorem \ref{twoimplythethird} is in fact one case of Metatheorem B,
since conditions \ref{Gori}, \ref{Gbip}, and \ref{Gpeori} are just
properties (1), (2), and (3), respectively, of our original properties.
 (The equivalence between \ref{Gpeori} and property (3) is discussed in
Subsubsection \ref{cond011}.)
 We give a proof of the theorem using gems and generalized
bipartiteness.  Most of the work involves translating the properties of
Theorem \ref{twoimplythethird} into parity conditions for closed walks.

 We rely on the following result to translate conditions \ref{Gori} and
\ref{Gpeori} of Theorem \ref{twoimplythethird}.

 \begin{theorem}[Lins {\cite{gemslins}}]
 \label{orientablebip}
 Let $\eG$ be a graph embedding with corresponding gem $\cJ$.  Then
$\eG$ is orientable if and only if $\cJ$ is bipartite.
 \end{theorem}

 If $\eG$ has no isolated vertex components, Theorem \ref{orientablebip}
can be regarded as a consequence of a 1949 result of Tutte
\cite[p.~481]{Tu49}, namely that a $3$-face-colorable embedding of a
cubic graph is orientable if and only if the graph is bipartite.  This
has been rediscovered a number of times.  To apply it here, note that
the partition of the faces of a gem into v-, f-, and e-faces provides a
proper $3$-face-coloring, and that $\cJ$ is orientable if and only if
$\eG$ is orientable, because they are embedded in the same surface.
 If $\eG$ has isolated vertex components, they affect neither the
orientability of $\eG$ nor the bipartiteness of $\cJ$, so the result
also holds in that case.

 Since $(v+f+a)_\cJ$ counts all edges in a gem $\cJ$, Theorem
\ref{orientablebip} is equivalent to the following.

 \begin{corollary}\label{gemori}
 Let $\eG$ be a graph embedding with corresponding gem $\cJ$.  Then
$\eG$ (or equivalently $\eG\sdu$) is orientable if and only if
$(v+f+a)_\cJ(K)$ is even for all closed walks $K$ in $\cJ$.
 \end{corollary}

 We can also derive the following from Theorem \ref{orientablebip}.

 \begin{theorem}\label{gempeori}
 Let $\eG$ be a graph embedding with corresponding gem $\cJ$.
 Then $\eG\spe$ (or equivalently $\eG\spe\sdu$) is orientable if and
only if $(v+a)_\cJ(K)$ is even for all closed walks in $\cJ$.
 \end{theorem}

 \begin{proof}
 Suppose that $\eG\spe$ is orientable.  By Theorem \ref{orientablebip}, 
$\cJ\epe$ is bipartite and so has a bipartite labeling $\la$.
 The value of $\la$ flips as we go along edges of color $\cv$ or $\ca$
in $\cJ$, as they are also edges of $\cJ\epe$.  But the edges of color
$\cf$ in $\cJ$ join vertices that are distance $2$ in $\cJ\epe$, and so
the value of $\la$ does not change on edges of color $\cf$ in $\cJ$. 
Therefore, by Observation \ref{binlab} we see that $\la$ is an
$\ec{\cv\cc\ca}\eJ$-labeling of $\cJ$, and so $(v+a)_\cJ(K)$ is even for
all closed walks $K$ in $\cJ$.

 Conversely, if $(v+a)_\cJ(K)$ is even for all closed walks in $K$, then
there is an $\ec{\cv\cc\ca}\cJ$-labeling $\la$ in $\cJ$.  By similar
reasoning to the above, $\la$ is a bipartite labeling of $\cJ\epe$, so
$\cJ\epe$ is bipartite.  Thus, $\eG\spe$ is orientable by Theorem
\ref{orientablebip}.
 \end{proof}

 The following addresses Theorem \ref{orientablebip}\ref{Gbip}.

 \begin{theorem}\label{gembip}
 Let $\eG$ be a graph embedding with corresponding gem $\cJ$.
 Then $\eG$ (or equivalently $\eG\spe$) is bipartite if and only if
$f_\cJ(K)$ is even for all closed walks in $\cJ$.
 \end{theorem}

 \begin{proof}
 Suppose first that $\eG$ has no isolated vertex components.
 From Theorem \ref{oddwalkoddcycletheorem}\ref{sbipK} and
\ref{sbipcontract}, $f_\cJ(K)$ is even for all closed walks $K$ in $\cJ$
if and only if $H = J \cg \ec{\cv\cc\ca}\cJ$ is bipartite.  But $H$ is
just the result of contracting all v-gons in $J$, and is just $G$ with
each edge $e$ replaced by two parallel edges (originally the two edges
of color $\cf$ in the e-square of $\cJ$ corresponding to $e$).
 Therefore, $H$ is bipartite if and only if $G$ is bipartite, i.e.,
$\eG$ is bipartite.

 If $\eG$ has isolated vertex components, they can be discarded without
affecting either $\cJ$ or the bipartiteness of $\eG$, so the result also
holds in that case.
 \end{proof}

 Notice that in the proofs above we used different parts of Theorem
\ref{oddwalkoddcycletheorem}, involving parity conditions for closed
walks, labelings, and contractions.
 Now that we have translated the conditions of Theorem
\ref{twoimplythethird} into parity conditions for closed walks, its
proof is easy.

 \begin{proof}[Proof of Theorem \ref{twoimplythethird}]
 Let $\cJ$ be the gem of $\eG$.  By Theorems \ref{gemori},
\ref{gempeori}, and \ref{gembip}, the conditions of Theorem
\ref{twoimplythethird} become, respectively, the following.

 \begin{enumerate}[label=\rm(\alph*)]\setlength{\itemsep}{0pt}
     \item
 	$(v+a+f)_\cJ(K)$ is even for all closed walks $K$ in $\cJ$.
     \item
 	$f_\cJ(K)$ is even for all closed walks $K$ in $\cJ$.
     \item
 	$(v+a)_\cJ(K)$ is even for all closed walks $K$ in $\cJ$.
 \end{enumerate}
 But given two of the functions $(v+a+f)_\cJ, f_\cJ$, and $(v+a)_\cJ$,
the third can be obtained by either addition or subtraction, so any two
of (a), (b), or (c) imply the third.
 \end{proof}

 The other known result that we re-prove here is the following.  Recall
that an embedding is \emph{directable} if the edges can be directed so
that every face boundary is a directed walk.

 \begin{theorem}[{R.-X. Hao \cite[Lemma 4.1]{hao}}]\label{hao}
 Let $\eG$ be an orientable graph embedding.  Then $\eG$ is directable
if and only if $\eG$ is $2$-face-colorable.
 \end{theorem}

 Again, it turns out that this result can be strengthened, to the
following.

 \begin{theorem}\label{hao2imply3}
 Let $\eG$ be an embedded graph.  Any two of the following properties
imply the third:
 \begin{enumerate}[label=\rm(\alph*)]\setlength{\itemsep}{0pt}
     \item\label{hori}
 	$\eG$ is orientable.
     \item\label{h2fc}
 	$\eG$ is $2$-face-colorable.
     \item\label{hdirectable}
 	$\eG$ is directable.
 \end{enumerate}
 \end{theorem}

 This is also a special case of Metatheorem B, involving properties (1),
(5), and (4).  Condition \ref{hori} is equivalent to $\eG\sdu$ being
orientable, \ref{h2fc} is equivalent to $\eG\sdu$ being bipartite, and
(see Proposition \ref{directedembeddingdualpetrieori} below)
\ref{hdirectable} is equivalent to $\eG\sdu\spe$ being orientable.
Therefore, Theorem \ref{hao2imply3} is just Theorem
\ref{twoimplythethird} with $\eG$ replaced by $\eG\sdu$.

 Alternatively, we could give an independent proof for Theorem
\ref{hao2imply3} based on parity conditions for closed walks in the
jewel $\cJ$ of $\eG$.  As we have shown (for \ref{hori}) or will show
below (for \ref{h2fc} and \ref{hdirectable}), conditions \ref{hori},
\ref{h2fc}, and \ref{hdirectable} turn out to be equivalent to evenness
of $(v+f+a)_\cJ(K)$, $v_\cJ(K)$, and $(f+a)_\cJ(K)$, respectively, for
all closed walks $K$ in $\cJ$.  Thus, a proof similar to the proof of
Theorem \ref{twoimplythethird} above can be provided.

 Not every theorem implied by Metatheorem B can be obtained just by
applying Theorem \ref{twoimplythethird} to one of the six embeddings
$\eG$, $\eG\sdu$, $\eG\spe$, $\eG\sdu\spe$, $\eG\spe\sdu$, or
$\eG\sdu\spe\sdu$.  Metatheorem B implies seven theorems, and only three
of them (Theorems \ref{twoimplythethird} and \ref{hao2imply3}, and one
other) can be obtained in this way.  The other four theorems are, as far
as we know, new, and will be discussed in more detail later.

 \subsection{Gems versus jewels}\label{ss:gemsvsjewels}

 \let\al\alpha
 \def\alv{\al_\cv}
 \def\alf{\al_\cf}
 \def\alz{\al_\cz}
 \def\ala{\al_\ca}
 \def\alvfa{\alv\alf\ala}
 \def\alvfza{\alv\alf\alz\ala}
 \def\CW{\textrm{\rm CW}}
 \newenvironment{condition}%
 	{\par\smallskip\advance\leftskip by 2\parindent\noindent}%
 	{\par\smallskip\advance\leftskip by -2\parindent\noindent}

 In the previous subsection, our proofs related embedding properties to
parity conditions for closed walks in gems.  In fact, all seven of our
properties for an embedded graph $\eG$ can be described by a condition
of the form
 \begin{condition}%
 $\CW(\cJ, \alvfa)$:\quad
 $(\alv v + \alf f + \ala a)_\cJ(K)$ is even for all closed walks $K$ in
$\cJ$,
 \end{condition}%
 where $\cJ$ is the gem of $\eG$ and $\alvfa$ is a $01$-string of length
$3$.
 This actually gives eight conditions, but if  $\alv = \alf = \ala = 0$,
we have a property trivially satisfied by all embeddings.
 Our theory can be developed using conditions for gems, as can be seen
in an earlier version of this work in \cite[Chapter 4]{BlakeDiss}. 
However, a more symmetric approach, providing easier proofs of some
results, uses conditions in the jewel $\cL$ instead of the gem $\cJ$.

 In the jewel $\cL$ of $\eG$ we naturally consider conditions of the
form
 \begin{condition}%
 $\CW(\cL, \alvfza)$:\quad
 $(\alv v + \alf f + \alz z + \ala a)_\cL(K)$ is even for all closed
walks $K$ in $\cL$,
 \end{condition}%
 where $\alvfza$ is a $01$-string of length $4$.
 But some choices of $\alvfza$ give rise to conditions that are
satisfied only in trivial ways.  The remaining choices correspond to
parity conditions for closed walks in gems, as we now show.

 \begin{theorem}\label{gem2jewel}
 Suppose that $\cL$ is a jewel and $\cJ$ is the corresponding gem.
 Let $\alv, \alf, \alz, \ala \in \{0,1\}$.
 \begin{enumerate}[label=\rm(\alph*)]\setlength{\itemsep}{0pt}
  \item\label{jeweloddafz}
  If $\alv + \alf + \alz$ is odd, then
 $\CW(\cL, \alvfza)$ holds if and only if $\cL$ is empty.
  \item\label{jewelevenafz}
  If $\alv + \alf + \alz$ is even, then
 $\CW(\cL, \alvfza)$ holds 
 if and only if
 $\CW(\cJ, \alvfa)$ holds.
 \end{enumerate}
 \end{theorem}

 \begin{proof}
 (a) Assume that $\alv + \alf + \alz$ is odd.

 Suppose that $\CW(\cL, \alvfza)$ holds.  If $\cL$ is nonempty then it
has an e-simplex.  Let $K$ be any triangle in the e-simplex.  Then $K$
uses one edge of each color $\cv, \cf, \cz$, and so
   $(\alv v + \alf f + \alz z + \ala a)_\cL(K) = \alv + \alf + \alz$,
which is odd, a contradiction.  Therefore $\cL$ is empty.

 Conversely, if $\cL$ is empty, $\CW(\cL, \alvfza)$ is satisfied
vacuously.

 \smallskip
 (b) Assume that $\alv + \alf + \alz$ is even.

 Suppose that $\CW(\cL, \alvfza)$ holds. Every closed walk $K$ in $\cJ$
is also a closed walk in $\cL$, and moreover $z(K)=0$, so $(\alv v +
\alf f + \ala a)_\cJ(K) = (\alv v + \alf f + \alz z + \ala a)_\cL(K)$,
which is even.  Hence $\CW(\cJ, \alvfa)$ holds.

 Now suppose that $\CW(\cJ, \alvfa)$ holds.
 There is $S\subseteq E(J)$ with $(\alv v + \alf f + \ala a)_\cJ =
\num_S$, and by Theorem \ref{oddwalkoddcycletheorem}, $J$ has an
$S$-labeling $\la$.
 Since $V(L) = V(J)$, by Observation \ref{binlab}, $\la$ is an
$T$-labeling of $L$ for some $T \subseteq E(L)$, and furthermore $T = S
\cup S'$ for some $S' \subseteq E(L)-E(J) = \ec\cz\cL$.
 We know that $\num_{T}(K)$ is even for all closed walks $K$ in $\cL$
because the $T$-labeling $\la$ of $L$ exists.  So to prove that
$\CW(\cL, \alvfza)$ holds it suffices to show that  $\num_{T} = (\alv v
+ \alf f + \alz z + \ala a)_\cL$.

 If $\alv = \alf$, then $\alz = 0$.
 Moreover, diagonally opposite pairs of vertices of a given e-square in
$J$ have equal $\la$ values, and so the edges between such pairs in
$\cL$ do not belong to $T$.
 In other words, $\ec\cz\cL \cap T = \emptyset$ and so $T = S$.
 Therefore, in $L$ we have $\num_{T} = \num_S = (\alv v + \alf f + \alz
z + \ala a)_\cL$ because $\alz = 0$.

 If $\alv \ne \alf$, then $\alz = 1$.
 Moreover, diagonally opposite pairs of vertices of a given e-square in
$J$ have different $\la$ values, and so the edges between such pairs in
$\cL$ belong to $T$.
 In other words, $\ec\cz\cL \subseteq T$ and so $T = S \cup \ec\cz\cL$.
 Therefore, in $L$ we have $\num_{T} = \num_S + z_\cL = (\alv v + \alf f
+ \alz z + \ala a)_\cL$ because $\alz = 1$.
 \end{proof}

 In light of part \ref{jeweloddafz}, we will consider only conditions
$\CW(\cL, \alvfza)$ where $\alv + \alf + \alz$ is even.  In that case,
when moving from jewels to gems we just delete the term $\alz z$.  When
moving from gems to jewels we add a term $\alz z$ where $\alz$ is an
even parity check bit for $\alv$ and $\alf$ (i.e., $\alz$ makes $\alv +
\alf + \alz$ even).

 \let\ifoif\Leftrightarrow

 Using Theorem \ref{gem2jewel} we can, for example, give a very easy
proof of Theorem \ref{gempeori}, as follows:

 \smallskip

 \def\rgp#1#2{\leavevmode\hangindent=#2\parindent\hangafter=1%
 	{}\hskip-\parindent\hskip#1\parindent}

 \rgp13 $\eG\spe$ is orientable

 \rgp13 $\ifoif$ (Theorem \ref{gemori})\quad
 $\CW(\cJ\epe, 111)$ holds:
 $(v+f+a)_{\cJ\epe}(K)$ is even for all closed walks $K$ in $\cJ\epe$

 \rgp13 $\ifoif$ (Theorem \ref{gem2jewel})\quad
 $\CW(\cL\epe, 1101)$ holds:
 $(v+f+a)_{\cL\epe}(K)$ is even for all closed walks $K$ in $\cL\epe$

 \rgp13 $\ifoif$ (Petrie dual swaps colors $\cf$ and $\cz$ in
jewels)\quad
 $\CW(\cL, 1011)$ holds:
 $(v+z+a)_{\cL}(K)$ is even for all closed walks $K$ in $\cL$

 \rgp13 $\ifoif$ (Theorem \ref{gem2jewel})\quad
 $\CW(\cJ, 101)$ holds:
 $(v+a)_\cJ(K)$ is even for all closed walks $K$ in $\cJ$.
 \smallskip

 As this proof shows, a condition $\CW(\cL, \alvfza)$ has an advantage
over a condition $\CW(\cJ, \alvfa)$ when taking some combination of
duals and Petrie duals: we just modify $\CW(\cL, \alvfza)$ using the
corresponding permutation of the colors $\cv, \cf, \cz$. 

 The following observation is trivial but we state it for later
reference.

 \begin{observation}\label{asmod2}
 Suppose $\al, \al' \in \{0,1\}$.  Let $\al \bp \al' = (\al + \al')
\bmod 2$ be the sum of $\al$ and $\al'$ considered as elements of
$\mZ_2$.  Then $\al+\al'$, $\al-\al'$, and $\al \bp \al'$ all have the
same parity.
 \end{observation}

 \def\alcl#1{\alv#1\alf#1\alz#1\ala#1}
 \let\be\beta
 \def\bev{\be_\cv}
 \def\bef{\be_\cf}
 \def\bez{\be_\cz}
 \def\bevfz{\bev \bef \bez}
 Suppose we know that two conditions $\CW(\cL, \al{})$ and $\CW(\cL, \al')$
hold, where $\al = \alcl{}$ and $\al' = \alcl'$, so that
 \begin{condition}%
 $(\alv{} v + \alf{} f + \alz{} z + \ala{} a)_\cL(K)$ and
 $(\alv' v + \alf' f + \alz' z + \ala' a)_\cL(K)$
 \end{condition}%
 are even for all closed walks $K$ in $\cL$.  Then clearly
 \begin{condition}%
 $((\alv{}+\alv') v + (\alf{}+\alf') f + (\alz{}+\alz') z
 	+ (\ala{}+\ala') a)_\cL(K)$
 \end{condition}%
 is even for all closed walks $K$ in $\cL$.  But by Observation
\ref{asmod2}, this is equivalent to
 \begin{condition}%
 $((\alv{}\bp \alv') v + (\alf{} \bp \alf') f + (\alz{} \bp \alz') z
 	+ (\ala{} \bp \ala') a)_\cL(K)$
 \end{condition}%
 being even for closed walks $K$ in $\cL$.  In other words, the
condition $\CW(\cL, \al{} \bp \al')$ holds, where $\al{} \bp \al'$ is
the sum of $\al{}$ and $\al'$ considered as vectors in the vector space
$\mZ_2^4$.  Similarly, if $\CW(\cJ, \wt\al)$ and $\CW(\cJ, \wt\al')$
hold, where $\wt\al, \wt\al' \in \mZ_2^3$, then $\CW(\cJ, \wt\al \bp
\wt\al')$ holds.
  
 The conditions $\CW(\cL, \alvfza)$ are of most interest when
$\alv+\alf+\alz$ is even, i.e., when $\alv\bp\alf\bp\alz = 0$.  The
corresponding vectors $\al =\alvfza$ form a $3$-dimensional subspace $Y$
of $\mZ_2^4$.  There is a basis for $Y$ that is invariant under
permutations of the first three coordinates, namely $\{0111, 1011,
1101\}$.  These vectors correspond to the three functions
 \begin{condition}%
 $\ov_\cL = (f+z+a)_\cL$,\quad
 $\of_\cL = (v+z+a)_\cL$,\quad
 and\quad
 $\oz_\cL = (v+f+a)_\cL$,
 \end{condition}%
 which count all edges other than those of color $\cv$, $\cf$, or $\cz$,
respectively.  We therefore define another condition for a jewel $\cL$
and $\be = \bevfz \in \mZ_2^3$, namely
 \begin{condition}%
 $\ovl\CW(\cL, \bevfz)$:\quad
 $(\bev\ov + \bef\of + \bez\oz)_\cL(K)$ is even for all closed walks $K$
in $\cL$,
 \end{condition}%
 The conditions $\ovl\CW(\cL, \be)$ have the same good behavior as
conditions $\CW(\cL, \al)$ under duals and Petrie duals: we can just
modify the condition using the corresponding permutation of the colors
$\cv,\cf,\cz$.
 If $\ovl\CW(\cL, \be)$ and $\ovl\CW(\cL, \be')$ hold for $\be, \be' \in
\mZ_2^3$, then $\ovl\CW(\cL, \be \bp \be')$ also holds, by similar
reasoning to that used for conditions of the form $\CW(\cL, \al)$.
 The correspondence between conditions $\CW(\cL, \al)$ and $\ovl\CW(\cL,
\be)$ is easily determined, as follows.

 \begin{observation}\label{jewel2jewel}
 Suppose $\cL$ is a jewel, $\al = \alvfza \in \mZ_2^4$ with
$\alv\bp\alf\bp\alz = 0$, and $\be = \bevfz \in \mZ_2^3$.  Then
 \begin{enumerate}[label=\rm(\alph*)]\setlength{\itemsep}{0pt}
  \item
    $\CW(\cL, \al)$\quad $\ifoif$\quad
    $\ovl\CW(\cL, (\alv\bp\ala)(\alf\bp\ala)(\alz\bp\ala))$,\quad and
  \item
    $\ovl\CW(\cL, \be)$\quad $\ifoif$\quad
    $\CW(\cL,
 	(\bef\bp\bez)(\bev\bp\bez)(\bev\bp\bef)(\bev\bp\bef\bp\bez))$.
 \end{enumerate}
 \end{observation}

 \subsection{Proof of a Fano structure}

 We have now developed a vector space structure for embedded graph
properties that can be represented as $\ovl\CW(\cL, \be)$ (or,
equivalently, as $\CW(\cL, \al)$ or $\CW(\cJ, \wt\al)$).
 This gives the overall Fano framework we wish to prove, subject to the
following claim which describes the link between the properties of $\eG$
from Section \ref{sec:intro} and the properties $\ovl\CW(\cL, \be)$.
 Define property (0) of a graph embedding $\eG$ to be the always-true
property.
 Let $b : \{0,1,2,\dots,7\} \to \mZ_2^3$ be the bijection mapping each
number to its representation as a $3$-digit binary number, considered as
a vector in $\mZ_2^3$.  For example, $b(3)=011$.

 \begin{claim}\label{propclaim}
 Suppose $\eG$ is an embedded graph with corresponding jewel $\cL$.
 If $p \in \{0,1,2,\dots,7\}$ then property $(p)$ for $\eG$, from
Section \ref{sec:intro} or the previous paragraph, is equivalent to
$\ovl\CW(\cL, b(p))$.
 \end{claim}

 We postpone the proof, although at this point we could easily prove it
for some values of $p$, such as $p=1$.

 \begin{observation}\label{transadd}
 If $p, p' \in \{0,1,2,\dots,7\}$ and $p \bp p'$ is defined as in
Section \ref{sec:intro}, then $b(p \bp p') = b(p) \bp b(p')$.
 \end{observation}

 Now we can introduce the Fano plane.  The \emph{Fano plane $F$} has a
standard representation as a Desarguesian projective plane over $\mZ_2$:
the \emph{points} of $F$ are the nonzero vectors of $\mZ_2^3$,
\emph{$d$-dimensional projective subspaces} of $F$ are
$(d+1)$-dimensional vector subspaces of $\mZ_2^3$ with $0$ removed, and
in particular a \emph{line} of $F$ consists of three nonzero vectors in
$\mZ_2^3$ whose sum is $0$.

 \begin{theorem}\label{subspace}
 Let $\eG$ be an embedded graph with corresponding jewel $\cL$.  Then
 $X = \{\be \in \mZ_2^3 \;|\; \ovl\CW(\cL, \be)$ holds$\}$ is a vector
subspace of $\mZ_2^3$, and $X^-=X-\{0\}$ is a projective subspace of the
Fano plane $F$.
 \end{theorem}

 \begin{proof}
 In a binary vector space a subset is a subspace if it contains $0$ and
is closed under addition.  Since $\ovl\CW(\cL, 000)$ holds trivially for
all $\cL$, we have $0 = 000 \in X$.  From discussion above after the
definition of $\ovl\CW(\cL, \be)$, we know that if $\be, \be' \in X$,
then $\be \bp \be' \in X$.  Thus, $X$ is a vector subspace of $\mZ_2^3$
and hence $X^-$ is a projective subspace of $F$.
 \end{proof}

 \begin{proof}[Proof of results in Section \ref{sec:intro}, subject to Claim
\ref{propclaim}]
 The Fano Framework as stated in Section \ref{sec:intro} is obtained by
translating the statement about $X^-$ in Theorem \ref{subspace} into a
statement about properties of an embedded graph $\eG$ using Claim
\ref{propclaim}.
 Metatheorem A follows because every $d$-dimensional projective subspace
of $F$ has $2^{d+1}-1$ elements, where $d \in \{-1,0,1,2\}$.

 If $\ovl\CW(\cL, \cdot)$ holds for certain vectors of $\mZ_2^3$, then
by Theorem \ref{subspace} it must also hold for all vectors in their
span.
 If three nonzero vectors $\be,\be',\be'' \in \mZ_2^3$ satisfy
$\be\bp\be'\bp\be'' = 0$, then they are on a line in $F$, and each one
is in the span of the other two.  Thus, if $\ovl\CW(\cL, \cdot)$ holds
for two of them, then it holds for the third.  Translating this to
properties of $\eG$ using Claim \ref{propclaim} and Observation
\ref{transadd} yields Metatheorem B.
 If three nonzero vectors $\be,\be',\be'' \in \mZ_2^3$ satisfy
$\be\bp\be'\bp\be'' \ne 0$, then they are not on a line in $F$, and they
span all of $\mZ_2^3$, which translates to Metatheorem C.
 \end{proof}

 \section{Graph embedding properties}\label{sec:fanoproperties}

 \def\cd(#1){\ovl\CW(\cL, {#1})}

 In this section we consider the individual conditions $\ovl\CW(\cL,
\be)$ for $\be \ne 0$, and establish their links to the graph embedding
properties of Section \ref{sec:intro}, verifying Claim \ref{propclaim}, and
thus completing the proof of the Fano Framework. In some cases other
equivalent properties are described as well.  We also characterize the
graphs $G$ that possess an embedding satisfying each condition.

 From now on we assume $\eG$ is an embedded graph, $\cJ$ is its gem, and
$\cL$ is its jewel.  We write `$\eG$ satisfies condition $(\be)$', for
$\be \in \mZ_2^3$, to mean that $\cL$ satisfies $\cd(\be)$.  All
embedded graphs satisfy condition $(000)$, and since $b(0)=000$, Claim
\ref{propclaim} holds for $p=0$.

 \subsection{Orientability and bipartiteness of graph embeddings}

 Because we have carefully chosen our conditions $(\be)$ so that the
effect of duals and Petrie duals is easy to track by permuting colors in
$\cL$, and hence permuting coordinates in $\be$, we can see that three
of our conditions $(\be)$ describe orientability of two embeddings, and
three describe bipartiteness of two embeddings.  To see which, we use
Theorem \ref{gem2jewel} and Observation \ref{jewel2jewel} to convert
conditions for $\cJ$ into conditions $(\be)$ based on $\cL$.

 \def\hquad{\hskip0.5em\relax}
 \def\vstrut#1#2{\vrule height#1 depth#2 width0pt}
 \def\hdstrut{\vstrut{13pt}{3pt}}
 \begin{table}[tb]
 \begin{center}\rm
 \bgroup\def\arraystretch{0.15}
 \begin{tabular}{|l||l|l|l|}
 \hline
 \hdstrut Condition & Conditions, functions, and
 				         & Condition, function, and
 						& Orientability/ \\
                    & coboundary in $\cL$  & coboundary in $\cJ$
 						& bipartiteness \\
 \hline &&&\\ \hline
 \vstrut{12pt}{4pt}
 (001) & $\cd(001) = \CW(\cL,1101)$ & $\CW(\cJ,111)$ & $\eG, \eG\sdu$\\
 \vstrut{0pt}{6pt}
       & $\oz_\cL = (v+f+a)_\cL$    & $(v+f+a)_\cJ$  & orientable \\
 \vstrut{0pt}{6pt}
       & $\ec{\cv\cc\cf\cc\ca}\cL$ is coboundary
		& $\ec{\cv\cc\cf\cc\ca}\cJ$ is coboundary & \\
 \hline
 \hdstrut
 \vstrut{12pt}{4pt}
 (010) & $\cd(010) = \CW(\cL,1011)$ & $\CW(\cJ,101)$ & $\eG\spe,
 						     \eG\spe\sdu$\\
 \vstrut{0pt}{6pt}
       & $\of_\cL = (v+z+a)_\cL$    & $(v+a)_\cJ$    & orientable \\
 \vstrut{0pt}{6pt}
       & $\ec{\cv\cc\cz\cc\ca}\cL$ is coboundary
		& $\ec{\cv\cc\ca}\cJ$ is coboundary & \\
 \hline
 \vstrut{12pt}{4pt}
 (011) & $\cd(011) = \CW(\cL,0110)$ & $\CW(\cJ,010)$ & $\eG, \eG\spe$\\
 \vstrut{0pt}{6pt}
       & $(\of+\oz)_\cL$ or $(f+z)_\cL$    & $f_\cJ$  & bipartite \\
 \vstrut{0pt}{6pt}
       & $\ec{\cf\cc\cz}\cL$ is coboundary
		& $\ec{\cf}\cJ$ is coboundary & \\
 \hline
 \vstrut{12pt}{4pt}
 (100) & $\cd(100) = \CW(\cL,0111)$ & $\CW(\cJ,011)$ & $\eG\sdu\spe,
 						     \eG\sdu\spe\sdu$\\
 \vstrut{0pt}{6pt}
       & $\ov_\cL = (f+z+a)_\cL$    & $(f+a)_\cJ$    & orientable \\
 \vstrut{0pt}{6pt}
       & $\ec{\cf\cc\cz\cc\ca}\cL$ is coboundary
		& $\ec{\cf\cc\ca}\cJ$ is coboundary & \\
 \hline
 \vstrut{12pt}{4pt}
 (101) & $\cd(101) = \CW(\cL,1010)$ & $\CW(\cJ,100)$ & $\eG\sdu,
 						     \eG\sdu\spe$\\
 \vstrut{0pt}{6pt}
       & $(\ov+\oz)_\cL$ or $(v+z)_\cL$    & $v_\cJ$  & bipartite \\
 \vstrut{0pt}{6pt}
       & $\ec{\cv\cc\cz}\cL$ is coboundary
		& $\ec{\cv}\cJ$ is coboundary & \\
 \hline
 \vstrut{12pt}{4pt}
 (110) & $\cd(110) = \CW(\cL,1100)$ & $\CW(\cJ,110)$ & $\eG\spe\sdu,
 						     \eG\sdu\spe\sdu$\\
 \vstrut{0pt}{6pt}
       & $(\ov+\of)_\cL$ or $(v+f)_\cL$    & $(v+f)_\cJ$  & bipartite \\
 \vstrut{0pt}{6pt}
       & $\ec{\cv\cc\cf}\cL$ is coboundary
		& $\ec{\cv\cc\cf}\cJ$ is coboundary & \\
 \hline
 \vstrut{12pt}{4pt}
 (111) & $\cd(111) = \CW(\cL,0001)$ & $\CW(\cJ,001)$ & $\eM$\\
 \vstrut{0pt}{6pt}
       & $(\ov+\of+\oz)_\cL$ or $a_\cL$    & $a_\cJ$  & bipartite \\
 \vstrut{0pt}{6pt}
       & $\ec{\ca}\cL$ is coboundary
		& $\ec{\ca}\cJ$ is coboundary & \\
 \hline
 \end{tabular}
 \egroup
 \caption{Mapping between conditions and orientability/bipartiteness
properties.}\label{tab:condoribip}
 \end{center}
 \end{table}

 By Theorem \ref{orientablebip}, $\eG$ (and $\eG\sdu$) are orientable
$\ifoif$ $(v+f+a)_\cJ$ is even for all closed walks $K$ in $\cJ$
$\ifoif$ $(v+f+a)_\cL(K) = \oz(K)$ is even for all closed walks $K$ in
$\cL$, i.e., $(001)$ holds. 
 Since the Petrie dual switches colors $\cf$ and $\cz$ in $\cL$, we see
that $\eG\spe$ (and $\eG\spe\sdu$) are orientable if and only if $(010)$
holds.
 Since the Wilson dual switches colors $\cv$ and $\cz$ in $\cL$, we see
that $\eG\swi = \eG\sdu\spe\sdu$ (and $\eG\swi\sdu = \eG\sdu\spe$) are
orientable if and only if $(100)$ holds.  Thus, the three vectors of
weight $1$ in $\mZ_2^3$ correspond to orientability properties.

 By Theorem \ref{gembip}, $\eG$ (and $\eG\spe$) are bipartite $\ifoif$
$f_\cJ(K)$ is even for all closed walks $K$ in $\cJ$ $\ifoif$
$(f+z)_\cL(K)$ is even for all closed walks $K$ in $\cL$, i.e.,
$\CW(\cL, 0110)$ holds $\ifoif$ $(011)$ holds.
 Since duality switches colors $\cv$ and $\cf$ in $\cL$, we see that
$\eG\sdu$ (and $\eG\sdu\spe$) are bipartite if and only if $(101)$ holds.
 Since the Wilson dual switches colors $\cv$ and $\cz$ in $\cL$, we see
that $\eG\swi = \eG\spe\sdu\spe$ (and $\eG\swi\spe = \eG\spe\sdu$) are
bipartite if and only if $(110)$ holds.
 Thus, the three vectors of weight $2$ in $\mZ_2^3$ correspond to
bipartiteness properties.

 This leaves one condition, namely $(111)$, which says that
$(\ov+\of+\oz)_\cL(K) = a_\cL(K)$ is even for all closed walks $K$ in
$\cL$.  By Theorem \ref{oddwalkoddcycletheorem}, this is equivalent to
$L \cg \ec{\cv\cc\cf\cc\cz}\cL$ being bipartite, but that graph is
exactly the medial graph $M$.  So condition (111) is bipartiteness of
$M$ (or $\eM$ or $\cM$): the weight $3$ vector in $\mZ_2^3$ represents
medial bipartiteness.

 Table \ref{tab:condoribip} summarizes these conclusions and shows the
conditions in a number of forms, the corresponding edge-counting
functions for closed walks in $\cL$ and $\cJ$, and the corresponding
binary labelings of $\cL$ and $\cJ$.  The conditions are in
lexicographic order.

 At this point the Fano structure shown in Figure
\ref{fig:FanoFramework} is apparent.  The points represent the seven
conditions.  The lines (including the circle) represent the theorems
coming from Metatheorem B: the points on each line are labeled by three
vectors that sum to $0$.  The central point represents medial
bipartiteness, and each of the six outer points represents either
orientability or bipartiteness of two of the embedded graphs $\eG,
\eG\sdu, \eG\spe, \eG\sdu\spe, \eG\spe\sdu, \eG\sdu\spe\sdu$.  Each
embedded graph is listed between the two points that represent its
orientability and bipartiteness.

 \begin{figure}
     \centering
  \includegraphics{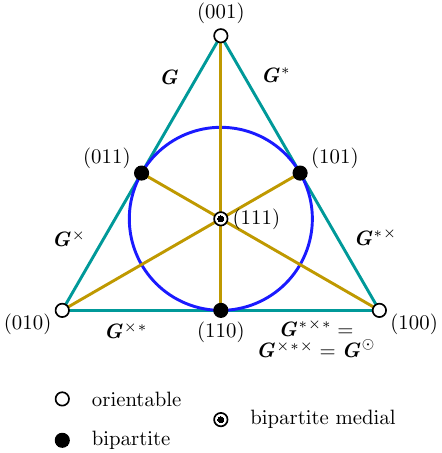}
     \caption{Fano Framework for Graph Embeddings.}
         \label{fig:FanoFramework}
 \end{figure}

  The three sides of the triangle give theorems relating two
orientability properties and one bipartiteness property: this includes
Theorems \ref{twoimplythethird} and \ref{hao2imply3}, and a theorem
derived by applying Theorem \ref{twoimplythethird} to $\eG\swi =
\eG\spe\sdu\spe$ or $\eG\swi\spe = \eG\spe\sdu$.

 The three medians of the triangle give theorems relating one
orientability property, medial bipartiteness and one bipartite property,
and are to our knowledge new.  One such result is Theorem
\ref{omb2imply3} below; the other two can be obtained by applying this
theorem to $\eG\spe$ or $\eG\spe\sdu$, and to $\eG\swi =
\eG\sdu\spe\sdu$ or $\eG\swi\sdu = \eG\sdu\spe$.

 \begin{theorem}\label{omb2imply3}
 Let $\eG$ be an embedded graph with medial checkerboard $\cM$.  Any two
of the following properties imply the third:
 \begin{enumerate}[label=\rm(\alph*)]\setlength{\itemsep}{0pt}
     \item
 	$\eG$ (or $\eG\sdu$) is orientable.
     \item
 	$\cM$ is bipartite.
     \item
 	$\eG\swi=\eG\spe\sdu\spe$ (or $\eG\swi\spe = \eG\spe\sdu$) is
bipartite.
 \end{enumerate}
 \end{theorem}

 The circle gives one theorem relating three bipartiteness properties,
which to our knowledge is also new.

 \begin{theorem}\label{bbb2imply3}
 Let $\eG$ be an embedded graph.  Any two of the following properties
imply the third:
 \begin{enumerate}[label=\rm(\alph*)]\setlength{\itemsep}{0pt}
     \item
 	$\eG$ (or $\eG\spe$) is bipartite.
     \item
 	$\eG\sdu$ (or $\eG\sdu\spe)$ is bipartite.
     \item
 	$\eG\swi=\eG\spe\sdu\spe$ (or $\eG\spe\sdu$) is bipartite.
 \end{enumerate}
 \end{theorem}

 \subsection{Individual conditions and graph embedding properties}

 In this subsection we consider each condition $(\be)$ with $\be \in
\mZ_2^3$, $\be \ne 0$, for an embedded graph with gem $\cJ$ and jewel
$\cL$.
 We verify Claim \ref{propclaim} for $p=b\iv(\be)$ by showing that
$(\be)$ corresponds to the property $(p)$ from Section 1, describe
$(\be)$ in terms of bipartiteness of one or more graphs associated with
$\eG$, and in some cases discuss other properties equivalent to $(\be)$.
 (Each condition is also associated with a particular type of
bidirection of the medial checkerboard $\cM$, but we discuss this later,
in Section \ref{sec:bidirections}.)
 We consider the values of $\be$ in increasing order by weight, and
within each weight in an order that is convenient for our proofs.

 \definecolor{dkgreen}{rgb}{0,0.6,0.1}
\definecolor{ltdkgreen}{rgb}{0.3,0.8,0.5}
 \def\nc[#1]{{\color{dkgreen}{(#1)}}}

 \subsubsection{Condition $(001)$}\label{cond001}

 From Table \ref{tab:condoribip}, condition $(001)$ means that $\eG$ is
orientable, which is property $(1)$ from Section \ref{sec:intro}, as
required.

 By Theorem \ref{orientablebip}, condition $(001)$ is equivalent to the
gem $\cJ$ being bipartite.

 There is also another graph whose bipartiteness is equivalent to
condition $(001)$.
 From Table \ref{tab:condoribip}, $(v+f+a)_\cL(K)$ is even for all closed
walks $K$ in $\cL$, which by Theorem \ref{oddwalkoddcycletheorem} is
equivalent to $L'=L \cg \ec\cz\cL$ being bipartite.  The vertices of
this graph correspond to the edges colored $\cz$ in $\cL$, the
\emph{diagonals} of the e-squares of $\cJ$.  Each diagonal vertex is
adjacent by four edges of colors $\cv, \cf$ to the other diagonal of the
same e-square, and by two edges of color $\cz$ to two diagonals in
adjacent e-squares.  Replacing each group of four parallel edges of
colors $\cv, \cf$ in $L'$ by a single edge gives a $3$-regular graph
called the \emph{diagonal graph} of $\eG$, whose bipartiteness is
equivalent to condition $(001)$.

 Every graph $G$ has an orientable embedding, i.e., an embedding
satisfying condition $(001)$.

 \subsubsection{Condition $(010)$}\label{cond010}

 From Table \ref{tab:condoribip}, condition $(010)$ means that $\eG\spe$ is
orientable.
 Thus, every cycle $C$ of $\eG\spe$ is $2$-sided, i.e., if we take a
rotation system/edge signature representation, each cycle has an even
number of edges of signature $-1$.  This is true if and only if, after
we flip the signature of each edge to obtain a representation of $\eG$,
the even cycles of $G$ have an even number of edges of signature $-1$,
and the odd cycles have an odd number of edges of signature $-1$.  In
other words, a cycle in $\eG$ is $1$-sided if and only if it is odd. 
Thus, condition $(010)$ is equivalent to property (2) of Section
\ref{sec:intro}, as required.

 The condition that $\eG\spe$ be orientable has been used in the
analysis of regular maps, beginning with Wilson \cite{Wil78}, who called
$\eG$ `pseudo-orientable' if this condition holds.  More recently this
property has been called `bi-orientable' \cite{BrCaSi19}.  It can also
be described by saying that each vertex can be given a rotation so that
the rotations at the end of each edge are opposite.  In other words,
$\eG$ has a representation using a rotation system and edge signatures
in which every edge has signature $-1$.

 By Theorem \ref{orientablebip}, condition $(010)$ is equivalent to the
gem of $\eG\spe$, i.e., $\cJ\epe$, being bipartite.

 There is also another graph whose bipartiteness is equivalent to
condition $(010)$.
 From Table \ref{tab:condoribip}, $(v+a)_\cJ(K)$ is even for all closed
walks $K$ in $\cJ$, which by Theorem \ref{oddwalkoddcycletheorem} is
equivalent to $J'=J \cg \ec\cf\cJ$ being bipartite.  The vertices of
$J'$ correspond to the edges colored $\cf$ in $\cJ$, which represent the
sides of the edges of $\eG$.  After reducing parallel edges in $J'$ we
obtain a $3$-regular graph in which each side vertex is adjacent to the
other side of the same edge and to the two sides consecutive with it
around a face.  This is called the \emph{side graph} of $\eG$, whose
bipartiteness is equivalent to condition $(010)$.

 Since every graph $G$ has an orientable embedding $\eG$, it also has an
embedding whose Petrie dual is orientable, namely $\eG\spe$.  So every
graph has an embedding satisfying condition $(010)$.

 \subsubsection{Condition $(100)$}\label{cond100}

 From Table \ref{tab:condoribip}, condition $(100)$ means that $\eG\sdu\spe$
is orientable.
 The following proposition shows that this is equivalent to property
(4), which says that $\eG$ is directable.
 This result is due to the authors and Joanna A. Ellis-Monaghan,
and is equivalent to \cite[Theorem 1]{NS24} due to Nakamoto and Suzuki.
 We will give another proof in a future paper on directed embeddings;
the proof in \cite{NS24} is different from either of ours.

 \begin{proposition}\label{directedembeddingdualpetrieori}
 A graph embedding $\eG$ is directable if and only if $\eG\sdu\spe$ is
orientable.
 \end{proposition}

 \begin{proof}
 Suppose $\eG$ is directable, and take a direction of its edges such
that each face boundary is a directed walk.  We can apply the direction
of each edge of $\eG$ to the corresponding edges of color $\cf$ in
$\cJ$.  Then those edges in each f-gon point in the same direction
around the f-gon, and those edges of each e-square point towards the
same edge of color $\cv$.
 Each vertex of $\cJ$ can be labeled $0$ or $1$ according to whether its
incident edge of color $\cf$ is directed inwards or outwards,
respectively.
 This is an $\ec{\cf\cc\ca}\cJ$-labeling of $J$, and so, by Theorem
\ref{oddwalkoddcycletheorem}, $(f+a)_\cJ(K)$ is even for all closed
walks $K$ in $\cJ$, and hence, from Table \ref{tab:condoribip}, $\eG\sdu\spe$
is orientable.

 This reasoning can be reversed, using an $\ec{\cf\cc\ca}\cJ$-labeling
to define a suitable direction on edges of color $\cf$ in $\cJ$ and
hence on edges of $\eG$.
 \end{proof}

 Condition $(100)$ is also equivalent to $\eG\swi$ being orientable,
which by Theorem \ref{orientablebip} is equivalent to the gem of
$\eG\swi$, which we denote as $\cJ\ewi$, being bipartite.

 There is also another graph whose bipartiteness is equivalent to
condition $(100)$.
 From Table \ref{tab:condoribip}, $(f+a)_\cJ(K)$ is even for all closed
walks $K$ in $\cJ$, which by Theorem \ref{oddwalkoddcycletheorem} is
equivalent to $J''=J \cg \ec\cv\cJ$ being bipartite.
 The vertices of $J''$ correspond to the edges colored $\cv$ in $\cJ$,
which represent the ends of the edges (half-edges) of $\eG$.  After
reducing parallel edges in $J''$ we obtain a $3$-regular graph in which
each end vertex is adjacent to the other end of the same edge and to the
two ends consecutive with it around a vertex.  This is called the
\emph{end graph} of $\eG$, whose bipartiteness is equivalent to
condition $(100)$.

 If $\eG$ satisfies condition $(100)$, then for each $v$-gon $K$ in
$\cJ$, $(f+a)_\cJ(K) = a_\cJ(K)$ is even, which means that the
corresponding vertex has even degree.  Thus, a connected graph with an
embedding satisfying condition $(100)$ must be Eulerian.  This was
pointed out by Lins \cite[Theorem 2.9(b)]{LinsPhD}.
 
 To show the converse, that every Eulerian graph has an embedding
satisfying condition $(100)$, we use the following lemma.
 This is implicit in work of Bonnington, Conder, Morton and McKenna
\cite{BCMM02}, and is closely related to Theorems \ref{hao} and
\ref{hao2imply3}.

 \begin{lemma}\label{oridiremb}
 Let $D$ be an Eulerian digraph (having a closed directed trail that
uses every arc exactly once).  Then $D$ has an orientable embedding in
which the boundary of each face is a directed walk, and this embedding
is $2$-face-colorable.
 \end{lemma}

 \begin{corollary}\label{euleriangraph}
 Every Eulerian graph $G$ has an embedding that is orientable,
$2$-face-colorable, and directable.
 \end{corollary}

 \begin{proof}
 By following an Euler circuit $T$ in $G$ and orienting each edge of $G$
in the direction of $T$, we obtain an Eulerian digraph $D$.
 Take an embedding $\eg{D}$ of $D$ provided by Lemma \ref{oridiremb},
and discard the directions on the arcs to obtain an embedding $\eG$ of
$G$.  This is clearly orientable, $2$-face-colorable, and directable
(add back directions to obtain $\eg{D}$).
 \end{proof}

 By Corollary \ref{euleriangraph}, an Eulerian graph $G$ has an
embedding that is directable, i.e., it satisfies condition $(100)$.
 Thus, a connected graph $G$ has an embedding satisfying condition
$(100)$ if and only if it is Eulerian.

 \subsubsection{Condition $(110)$}\label{cond110}

 From Table \ref{tab:condoribip}, condition $(110)$ means that $\eG\spe\sdu$
(or $\eG\spe\sdu\spe$) is bipartite.  Since an embedded graph is
bipartite if and only if its dual is $2$-face-colorable, this is
equivalent to $\eG\spe$ (or $(\eG\spe\sdu\spe)\sdu =
(\eG\sdu\spe\sdu)\sdu = \eG\sdu\spe$) being $2$-face-colorable.

 Since the faces of $\eG\spe$ are the zigzags of $\eG$ (and of
$\eG\sdu$), condition $(110)$ is also equivalent to $\eG$ being
$2$-zigzag-colorable, which means we can $2$-color the zigzags so that
the two zigzags sharing each edge are different colors (which implies
that a zigzag cannot use the same edge twice).

 We wish to show that condition $(110)$ is equivalent to property (6)
from Section \ref{sec:intro}, which says that if $\eG$ and its dual
$\eG\sdu$ are embedded together in their common underlying surface $\Si$
in the natural way, then the resulting regions of $\Si - (G \cup G\sdu)$
can be properly 2-colored.  Embedding $\eG$ and $\eG\sdu$
simultaneously, adding a new vertex at the intersection of each edge and
its dual edge, gives $\cB\dt \ec\ca\cB$, which we will call $\eG \cup
\eG\sdu$.
 We must show that this is $2$-face-colorable.

 We may assume that $\eG$ has no isolated vertex components.
 When we form $\cB$ we ignore isolated vertex components of $\eG$, so by
the above definition $\eG \cup \eG\sdu$ also has no information
regarding such components.
 While we could take the union of an isolated vertex component and its
dual in a single sphere, this would consist of two vertices embedded in
a sphere, which is problematic because it is not a cellular embedding.

 A face of $\cB \dt \ec\ca\cB$ is called a \emph{corner} of the embedded
graph $\eG$; it represents an incidence between a vertex and a face in $\eG$.
 Each  corner contains two flags with a common edge of color $\ca$.  The
corners of $\eG$ therefore correspond to edges of color $\ca$ in $\cB$,
edges of color $\ca$ in $\cJ$, or edges of $\cM$.  The \emph{corner
graph} of $\eG$ is the underlying graph of $(\cB\dt \ec\ca\cB)\sdu$, a
$4$-regular graph in which two vertices representing corners are joined
by an edge when those corners are consecutive around a vertex or a face.
 From this it is apparent that the corner graph is also the medial graph
of the medial embedding $\eM$ of $\eG$.

 Using the duality between deletion and contraction, and noting that the
edges in $\cJ$ of color $\ca$ do not induce any cycles, we see that the
corner graph is $J \cg \ec\ca\cJ$.  But from Table \ref{tab:condoribip},
condition $(110)$ means that $(v+f)_\cJ$ is even for all closed walks
$K$ in $\cJ$, and by Theorem \ref{oddwalkoddcycletheorem} this is
equivalent to the corner graph $J \cg \ec\ca\cJ$ being bipartite, which
is equivalent to $\eG \cup \eG\sdu = \cB\dt \ec\ca\cB$ being
$2$-face-colorable, as required.

 Condition $(110)$ is also equivalent to the Petrie dual of the medial
embedding, $\eM\spe$, being orientable.  To see this, observe that the
medial embedding is always $2$-face-colorable, so that $\eM^*$ is always
bipartite.  Therefore, applying Figure \ref{fig:FanoFramework} to
$\eM$, the medial graph of $\eM$ (i.e., the corner graph of $\eG$) being
bipartite is equivalent to $\eM\spe$ being orientable.

 If a graph $G$ has an embedding satisfying condition $(110)$, then
$\eG\spe\sdu$ is bipartite, which means that $\eG\spe$ is
$2$-face-colorable, which implies that each vertex of its underlying
graph, which is just $G$, has even degree.  Thus, a connected graph with
an embedding satisfying condition $(110)$ must be Eulerian.

 Conversely, if $G$ is Eulerian then by Corollary \ref{euleriangraph}
$G$ has a $2$-face-colorable embedding $\eG$.
 Now $\eG\spe$ is also an embedding of $G$, and $(\eG\spe)\spe = \eG$ is
$2$-face-colorable.  In other words, $\eG\spe$ is an embedding of $G$
satisfying condition $(110)$.

 Thus, a connected graph has an embedding satisfying condition $(110)$
if and only if it is Eulerian.

 \subsubsection{Condition $(011)$}\label{cond011}

 From Table \ref{tab:condoribip}, condition $(011)$ is equivalent to $\eG$
(or $\eG\spe$) being bipartite, which is property (3) from Section
\ref{sec:intro}.  Since an embedded graph is bipartite if and only if its
dual is $2$-face-colorable, this is equivalent to $\eG\sdu$ (or
$(\eG\spe\sdu$) being $2$-face-colorable.

 Condition $(011)$ is also equivalent to condition $(110)$ after
swapping colors $\cv$ and $\cz$ in the jewel, i.e., after taking the
Wilson dual.  Therefore, it is equivalent to bipartiteness of the corner
graph of $\eG\swi$, or the \emph{Wilson corner graph}.  This graph has
the same vertex set as the corner graph of $\eG$, but two vertices
representing corners are joined by an edge when those corners are
consecutive around a face or a zigzag of $\eG$.

 The class of graphs that have an embedding satisfying condition $(011)$
is the set of all bipartite graphs; any embedding satisfies the
condition.

 \subsubsection{Condition $(101)$}\label{cond101}

 From Table \ref{tab:condoribip}, condition $(101)$ is equivalent to $\eG\sdu$
(or $\eG\sdu\spe$) being bipartite.  Since an embedded graph is bipartite if
and only if its dual is $2$-face-colorable, this is equivalent to
$\eG$ (or $(\eG\sdu\spe\sdu$) being $2$-face-colorable, which is
property (5) from Section \ref{sec:intro}.

 Condition $(101)$ is also equivalent to condition $(110)$ after
swapping colors $\cf$ and $\cz$ in the jewel, i.e., after taking the
Petrie dual.  Therefore, it is equivalent to bipartiteness of the corner
graph of $\eG\spe$, or the \emph{Petrie corner graph}.  This graph has
the same vertex set as the corner graph of $\eG$, but two vertices
representing corners are joined by an edge when those corners are
consecutive around a vertex or a zigzag of $\eG$.

 If a graph $G$ has an embedding $\eG$ satisfying condition $(101)$,
then $\eG$ is $2$-face-colorable, which obviously means that every
vertex has even degree.  Thus, if $G$ is connected, $G$ is Eulerian.
 Conversely, suppose that $G$ is Eulerian.  By Corollary
\ref{euleriangraph}, $G$ has an embedding that is $2$-face-colorable,
i.e., it satisfies condition $(101)$.  Thus, a connected graph has an
embedding satisfying condition $(101)$ if and only if it is Eulerian.

 \subsubsection{Condition $(111)$}\label{cond111}

 As indicated in Table \ref{tab:condoribip}, condition $(111)$ is equivalent
to bipartiteness of the medial checkerboard $\cM$ of an embedded graph
$\eG$.  A proper $2$-coloring of the vertices of $\cM$ translates to a
coloring of the edges of $\eG$ where colors alternate around each vertex
and around each face, so condition $(111)$ is equivalent to property (7)
of Section \ref{sec:intro}.

 If $G$ has an embedding $\eG$ satisfying property (7), with edges
colored (say) black and white around each vertex and each face, then the
degree of each vertex is even.  Thus, if $G$ is connected then it is
Eulerian.  Moreover, the numbers of half-edges colored black and white
must be equal, so the numbers of edges colored black and white must be
equal, and so $|E(G)|$ must be even.

 Conversely, suppose $G$ is an Eulerian graph and $|E(G)|$ is even.  Let
$T$ be an Euler circuit in $G$.  Because $|E(T)|=|E(G)|$ is even, we can
color the edges of $T$ alternately black and white, resulting in an
equal number of black and white half-edges incident with each vertex. 
Choose a rotation scheme at each vertex that alternates black and white
half-edges, and assign all edges signature $1$.  The result is an
embedding $\eG$ that satisfies property (7), and is moreover orientable.

 Thus, a connected graph has an embedding satisfying condition $(111)$,
or equivalently property (7), if and only if it is Eulerian and has an
even number of edges.

 \subsubsection{Summary}\label{fanosummary}

 We have now verified Claim \ref{propclaim} by showing that that seven
properties listed in the introduction (Section \ref{sec:intro})
correspond to the seven Fano conditions in Table \ref{tab:condoribip}. 
We will henceforth refer to these as the seven `Fano properties'.

 In Table \ref{tab:summindcond} we give a full overview of the results
in this section.  The first column gives the appropriate condition and
corresponding property from Section \ref{sec:intro}.  The second column
gives equivalent embedding properties.  The third column gives
associated graphs, not already mentioned in the second column, whose
bipartiteness corresponds to the condition.  The fourth column
summarizes some results from Section \ref{sec:bidirections}
regarding bidirections of the medial checkerboard.  The final column
describes the connected graphs that have an embedding satisfying the
given condition.

 We have given a number of different ways to describe condition $(110)$,
which appear in the table.  We can similarly describe conditions $(011)$
and $(101)$ in several ways, but we omit these for brevity.

 We will provide examples to show that all combinations of the Fano
properties allowed by Metatheorem A actually occur.
 However, we postpone this to Section \ref{sec:examples},
since it is helpful to first discuss three further `Eulerian' properties
and their connections with the Fano properties.

 \begin{table}[th]
  \centering
  \begingroup
  \def\arraystretch{0.18}%
  \def\vstrut#1#2{\vrule height#1 depth#2 width0pt}%
  \def\vs#1#2{\vstrut{#1pt}{#2pt}}%
  \begin{tabular}{|l||l|l|l|l|}
  \hline
  \vs{10}3 Condition & Embedding properties & (Other) & Bidirection of & Connected \\
  \vs83 and & (all equivalent) & associated & $\cM$     & graphs having \\
  \vs85 property     &    & bipartite graphs &          & embedding \\
  \hline
  &&&&\\
  \hline
  \vs{12}3 $(001)$
	& $\eG, \eG\sdu$ orientable
	& gem $\cJ$,
	& b-direction
	& all
	\\
  \vs86 \kern0.5em(1)
	&
	& diagonal graph
	&&
	\\
  \hline
  \vs{12}3 $(010)$
	& $\eG\spe, \eG\spe\sdu$ orientable;
	& $\cJ\epe$,
	& d-direction
	& all
	\\
  \vs83 \kern0.5em(2)
	& cycle $1$-sided $\ifoif$ odd;
	& side graph
	&&
	\\
  \vs83
	& vertex rotations that
	&&&
	\\
  \vs86
	& \quad reverse along edges
	&&&
	\\
  \hline
  \vs{12}3 $(011)$
	& $\eG, \eG\spe$ bipartite;
	& Wilson corner
	& c-antidirection
	& bipartite
	\\
  \vs86 \kern0.5em(3)
	& $\eG\sdu, \eG\spe\sdu$ $2$-face-colorable
	& \quad graph
	&&
	\\
  \hline
  \vs{12}3 $(100)$
	& $\eG\sdu\spe, \eG\sdu\spe\sdu$ orientable;
	& $\cJ\ewi$,
	& c-direction
	& Eulerian
	\\
  \vs86 \kern0.5em(4)
	& directable
	& end graph
	&&
	\\
  \hline
  \vs{12}3 $(101)$
	& $\eG\sdu, \eG\sdu\spe$ bipartite;
	& Petrie corner
	& d-antidirection
	& Eulerian
	\\
  \vs86 \kern0.5em(5)
	& $\eG, \eG\sdu\spe\sdu$ $2$-face-colorable
	& \quad graph
	&&
	\\
  \hline
  \vs{12}3 $(110)$
	& $\eG\spe\sdu, \eG\sdu\spe\sdu$ bipartite;
	& corner graph =
	& b-antidirection
	& Eulerian
	\\
  \vs83 \kern0.5em(6)
	& $\eG\spe, \eG\sdu\spe$ $2$-face-colorable;
	& medial graph
	&&
	\\
  \vs83
	& $\eG, \eG\sdu$ $2$-zigzag-colorable;
	& \quad of $\eM$
	&&
	\\
  \vs83
	& $\eG \cup \eG\sdu$ $2$-face-colorable;
	&&&
	\\
  \vs86
	& $\eM\spe, \eM\spe\sdu$ orientable
	&&&
	\\
  \hline
  \vs{12}3 $(111)$
	& $\cM$ bipartite;
	&
	& t-direction
	& Eulerian with
	\\
  \vs83 \kern0.5em(7)
	& can $2$-color edges so alternate
	&&& \quad even number
	\\
  \vs86
	& \quad around vertices and faces
	&&& \quad of edges
	\\
  \hline
  \end{tabular}
  \endgroup
  \caption{Summary of individual conditions.}
  \label{tab:summindcond}
 \end{table}

 \let\ti\times
 \let\tofrom\leftrightarrow
 \def\ms{${}^-$}
 \def\ps{${}^+$}
 \def\cQ{\cs{Q}}
 \let\de\delta
 \let\tbdbl\tau 
 \let\gtype\gamma
 \let\stype\sigma

 \section{Medial checkerboard bidirections}\label{sec:bidirections}

 In this section we link the seven Fano plane conditions for an embedded
graph $\eG$ to structures in the medial checkerboard $\cM$.
 We use correspondences as follows:

 \smallskip
 \rgp24 Fano condition(s) for $\eG$

 \rgp24 $\tofrom$\quad binary labeling(s) of gem $\cJ$ or jewel $\cL$

 \rgp24 $\tofrom$\quad condition for edges colored $\ca$ in $\cJ$ or $\cL$

 \rgp24 $\tofrom$\quad structure in medial checkerboard $\cM$.
 \smallskip

 \noindent
 The last step here relies on the fact that edges of $\cM$ can be
identified with edges of color $\ca$ in $\cJ$ or $\cL$, because we can
think of $M$ as either $J \cg (E(J)-\ec\ca\cJ)$ or $L \cg
(E(L)-\ec\ca\cL)$.
 A similar approach appears in the proof of Proposition
\ref{directedembeddingdualpetrieori} above: there we used
correspondences between condition $(100)$, an
$\ec{\cf\cc\ca}\cJ$-labeling of the gem $\cJ$, and a condition for edges
colored $\cf$ in $\cJ$.

 Our work extends and systematizes some previous results in the
literature, which involve directions of the medial graph $M$.
 We work with \emph{bidirections}, which assign a direction to each
half-edge of a graph $G$.  An edge is \emph{directed} if its two halves
are directed the same way, and \emph{antidirected} if they are directed
oppositely.  We think of a \emph{direction} as a bidirection where all
edges are directed, and an \emph{antidirection} is a bidirection where
all edges are antidirected.  An antidirected edge is \emph{introverted}
or \emph{extraverted} depending on whether the halves are directed,
respectively, into or out of the center of the edge, or equivalently,
out of or into their incident vertices.

 Medial graphs are $4$-regular, and the bidirections that matter to us
have special properties.  In a $4$-regular graph with a bidirection, a
vertex is \emph{balanced} if it has two incoming half-edges and two
outgoing half-edges, and \emph{total} if all half-edges are incoming or
all half-edges are outgoing.  A bidirection is \emph{balanced},
\emph{total}, or \emph{balanced-total} if every vertex is balanced,
every vertex is total, or every vertex is balanced or total,
respectively.

 We will show that each individual Fano plane condition corresponds to a
particular subclass of balanced-total directions or antidirections of
$M$, with conditions involving the structure added to $M$ to form the
medial checkerboard $\cM$ (i.e., vertex rotations and face colorings).
 In Section \ref{sec:eulerian} we will expand this to include three
additional properties.
 We also discuss how these bidirections interact with partial duality
and partial Petrie duality.
 Allowable combinations of properties corresponding to lines in the Fano
plane, or the whole Fano plane, can also be related to structures in the
medial checkerboard: to do this we will use multiple binary labelings
simultaneously,


 \subsection{Previous results on medial graph
directions}\label{ss:previousdirections}

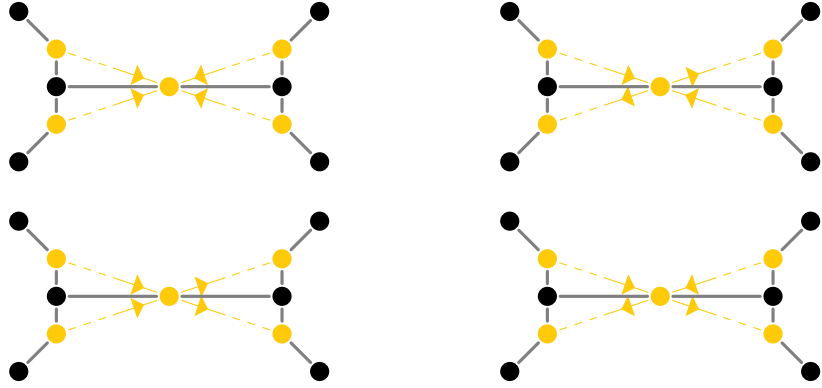
\begin{figure}
\hbox to \hsize{%
\hfill
   \begin{tikzpicture}[scale=1.0]
   
	\node[vertex] at (7,-5) (v0) {};
	\node[vertex, yellow!60!orange] at (7.5,-5.5) (v1) {};
	\node[vertex] at (7.5,-6) (v2) {};
	\node[vertex, yellow!60!orange] at (9,-6) (v3) {};
	\node[vertex] at (10.5,-6) (v4) {};
	\node[vertex, yellow!60!orange] at (10.5,-5.5) (v5) {};
	\node[vertex] at (11,-5) (v6) {};
	\node[vertex, yellow!60!orange] at (10.5,-6.5) (v7) {};
	\node[vertex] at (11,-7) (v8) {};
	\node[vertex, yellow!60!orange] at (7.5,-6.5) (v9) {};
	\node[vertex] at (7,-7) (v10) {};
	\begin{scope}[nodes={sloped,allow upside down}][on background layer]
		\draw[edge] (v0) to (v1);
		\draw[edge] (v1) to (v2);
       
            \draw[yellow!60!orange] (8.25,-5.75) -- pic[pos=.7]{arrow} (v3);
            \draw[yellow!60!orange] (9.75, -6.25) -- pic[pos=.7]{arrow} (v3);
            \draw[yellow!60!orange] (8.25,-6.25) -- pic[pos=.7]{arrow} (v3);
            \draw[yellow!60!orange] (9.75,-5.75) -- pic[pos=.7]{arrow} (v3);

            \draw[yellow!60!orange][dashed] (v1) to (8.25,-5.75);
            \draw[yellow!60!orange,dashed] (v5) to (9.75,-5.75);
            \draw[yellow!60!orange, dashed] (v7) to (9.75,-6.25);
		\draw[yellow!60!orange, dashed] (v9) to (8.25,-6.25);
        
		\draw[edge] (v2) to (v3);
		\draw[edge] (v2) to (v9);
		\draw[edge] (v3) to (v4);

		\draw[edge] (v4) to (v5);
		\draw[edge] (v4) to (v7);
		\draw[edge] (v5) to (v6);
		\draw[edge] (v7) to (v8);
		\draw[edge] (v9) to (v10);
	\end{scope}
\end{tikzpicture}
 \hfill
  \begin{tikzpicture}[scale=1.0]
   
	\node[vertex] at (7,-5) (v0) {};
	\node[vertex, yellow!60!orange] at (7.5,-5.5) (v1) {};
	\node[vertex] at (7.5,-6) (v2) {};
	\node[vertex, yellow!60!orange] at (9,-6) (v3) {};
	\node[vertex] at (10.5,-6) (v4) {};
	\node[vertex, yellow!60!orange] at (10.5,-5.5) (v5) {};
	\node[vertex] at (11,-5) (v6) {};
	\node[vertex, yellow!60!orange] at (10.5,-6.5) (v7) {};
	\node[vertex] at (11,-7) (v8) {};
	\node[vertex, yellow!60!orange] at (7.5,-6.5) (v9) {};
	\node[vertex] at (7,-7) (v10) {};
	\begin{scope}[nodes={sloped,allow upside down}][on background layer]
		\draw[edge] (v0) to (v1);
		\draw[edge] (v1) to (v2);
       
            \draw[yellow!60!orange] (8.25,-5.75) -- pic[pos=.7]{arrow} (v3);
            \draw[yellow!60!orange] (9.75, -6.25) -- pic[pos=.7]{arrow} (v3);
            \draw[yellow!60!orange] (v3) -- pic[pos=.6]{arrow} (8.25, -6.25);
            \draw[yellow!60!orange] (v3) -- pic[pos=.6]{arrow} (9.75,-5.75);

            \draw[yellow!60!orange][dashed] (v1) to (8.25,-5.75);
            \draw[yellow!60!orange,dashed] (v5) to (9.75,-5.75);
            \draw[yellow!60!orange, dashed] (v7) to (9.75,-6.25);
		\draw[yellow!60!orange, dashed] (v9) to (8.25,-6.25);
        
		\draw[edge] (v2) to (v3);
		\draw[edge] (v2) to (v9);
		\draw[edge] (v3) to (v4);

		\draw[edge] (v4) to (v5);
		\draw[edge] (v4) to (v7);
		\draw[edge] (v5) to (v6);
		\draw[edge] (v7) to (v8);
		\draw[edge] (v9) to (v10);
	\end{scope}
\end{tikzpicture}
\hfill
}

\vspace{5mm}

\hbox to \hsize{%
\hfill
  \begin{tikzpicture}[scale=1.0]
   
	\node[vertex] at (7,-5) (v0) {};
	\node[vertex, yellow!60!orange] at (7.5,-5.5) (v1) {};
	\node[vertex] at (7.5,-6) (v2) {};
	\node[vertex, yellow!60!orange] at (9,-6) (v3) {};
	\node[vertex] at (10.5,-6) (v4) {};
	\node[vertex, yellow!60!orange] at (10.5,-5.5) (v5) {};
	\node[vertex] at (11,-5) (v6) {};
	\node[vertex, yellow!60!orange] at (10.5,-6.5) (v7) {};
	\node[vertex] at (11,-7) (v8) {};
	\node[vertex, yellow!60!orange] at (7.5,-6.5) (v9) {};
	\node[vertex] at (7,-7) (v10) {};
	\begin{scope}[nodes={sloped,allow upside down}][on background layer]
		\draw[edge] (v0) to (v1);
		\draw[edge] (v1) to (v2);
       
            \draw[yellow!60!orange] (8.25,-5.75) -- pic[pos=.7]{arrow} (v3);
            \draw[yellow!60!orange] (8.25, -6.25) -- pic[pos=.7]{arrow} (v3);
            \draw[yellow!60!orange] (v3) -- pic[pos=.6]{arrow} (9.75, -6.25);
            \draw[yellow!60!orange] (v3) -- pic[pos=.6]{arrow} (9.75,-5.75);

            \draw[yellow!60!orange][dashed] (v1) to (8.25,-5.75);
            \draw[yellow!60!orange,dashed] (v5) to (9.75,-5.75);
            \draw[yellow!60!orange, dashed] (v7) to (9.75,-6.25);
		\draw[yellow!60!orange, dashed] (v9) to (8.25,-6.25);
        
		\draw[edge] (v2) to (v3);
		\draw[edge] (v2) to (v9);
		\draw[edge] (v3) to (v4);

		\draw[edge] (v4) to (v5);
		\draw[edge] (v4) to (v7);
		\draw[edge] (v5) to (v6);
		\draw[edge] (v7) to (v8);
		\draw[edge] (v9) to (v10);
	\end{scope}
\end{tikzpicture}
\hfill
  \begin{tikzpicture}[scale=1.0]
   
	\node[vertex] at (7,-5) (v0) {};
	\node[vertex, yellow!60!orange] at (7.5,-5.5) (v1) {};
	\node[vertex] at (7.5,-6) (v2) {};
	\node[vertex, yellow!60!orange] at (9,-6) (v3) {};
	\node[vertex] at (10.5,-6) (v4) {};
	\node[vertex, yellow!60!orange] at (10.5,-5.5) (v5) {};
	\node[vertex] at (11,-5) (v6) {};
	\node[vertex, yellow!60!orange] at (10.5,-6.5) (v7) {};
	\node[vertex] at (11,-7) (v8) {};
	\node[vertex, yellow!60!orange] at (7.5,-6.5) (v9) {};
	\node[vertex] at (7,-7) (v10) {};
	\begin{scope}[nodes={sloped,allow upside down}][on background layer]
		\draw[edge] (v0) to (v1);
		\draw[edge] (v1) to (v2);
       
            \draw[yellow!60!orange] (8.25,-5.75) -- pic[pos=.7]{arrow} (v3);
            \draw[yellow!60!orange] (9.75, -5.75) -- pic[pos=.7]{arrow} (v3);
            \draw[yellow!60!orange] (v3) -- pic[pos=.6]{arrow} (8.25, -6.25);
            \draw[yellow!60!orange] (v3) -- pic[pos=.6]{arrow} (9.75,-6.25);

            \draw[yellow!60!orange][dashed] (v1) to (8.25,-5.75);
            \draw[yellow!60!orange,dashed] (v5) to (9.75,-5.75);
            \draw[yellow!60!orange, dashed] (v7) to (9.75,-6.25);
		\draw[yellow!60!orange, dashed] (v9) to (8.25,-6.25);
        
		\draw[edge] (v2) to (v3);
		\draw[edge] (v2) to (v9);
		\draw[edge] (v3) to (v4);

		\draw[edge] (v4) to (v5);
		\draw[edge] (v4) to (v7);
		\draw[edge] (v5) to (v6);
		\draw[edge] (v7) to (v8);
		\draw[edge] (v9) to (v10);
	\end{scope}
\end{tikzpicture}
\hfill
}
  \caption{Top: t$^-$- and b-vertices; bottom: c- and d-vertices.}

    \label{fig:tbcd-vertices}
\end{figure}

 The use of directions (orientations) of the medial graph to
characterize embedding properties was introduced by Huggett and Moffatt
\cite{bippartdualsHM}.
 Deng, Jin, and Metsidik extended this in various ways \cite{dengjin,
MJ18, Met21} to bidirections and additional properties.
 These results characterized when partial duals of an embedded graph
$\eG$ have certain properties.
 We describe some known results of this kind here; additional results
are discussed in Subsection \ref{ss:eulerianbidirections}.

 \begin{definition}\label{bcdt}
 Suppose an embedded graph $\eG$, with medial checkerboard $\cM$, and a
bidirection $\de$ of the medial graph $M$, are given.
 A vertex $v_e$ of $\cM$, corresponding to $e \in E(G)$, is
 \begin{enumerate}[label=\rm(\alph*)]\setlength{\itemsep}{0pt}
 \item a \emph{b-vertex} if the half-edge directions around $v_e$ alternate
(in (towards $v_e$), out (away from $v_e$), in, out);
 \item a \emph{c-vertex} if the half-edge directions and face colors around
$v_e$ are in the cyclic order (in, $\cv$, in, $\cf$, out, $\cv$, out,
$\cf$);
 \item a \emph{d-vertex} if the half-edge directions and face colors around
$v_e$ are in the cyclic order (in, $\cf$, in, $\cv$, out, $\cf$, out,
$\cv$); and
 \item a \emph{t-vertex} if the half-edge directions around $v_e$ are
(in, in, in, in) (a \emph{t\ms-vertex}) or (out, out, out, out) (a
\emph{t\ps-vertex}).
 \end{enumerate}
 See Figure \ref{fig:tbcd-vertices}.
 Note that these definitions depend on both the bidirection $\de$ and
the vertex rotations and face colorings of $\cM$.
 But we consider $\de$ to be attached to the medial graph $M$, not the
complete medial checkerboard $\cM$, because we will consider the same
$\de$ for different medial checkerboards when we take twisted duals.
 If we have a balanced-total bidirection of $M$ then these four types
(with two subtypes of t-vertex) classify all vertices of $\cM$.

 An edge $e$ in $\eG$ is called a \emph{b-edge}, \emph{c-edge},
\emph{d-edge}, or \emph{t-edge} (or \emph{t\ms-edge} or
\emph{t\ps-edge}) according to the type of the corresponding vertex
$v_e$ in $\cM$.
 We use \emph{b-direction of $\cM$} (or \emph{b-direction of $\eG$}) to
refer to a direction of $M$ in which all vertices of $\cM$ are
b-vertices, \emph{bc-antidirection of $\cM$} to refer to
an antidirection of $M$ in which all vertices of $\cM$ are b- or
c-vertices, and other similar terminology.

 Note that if we reverse a balanced-total bidirection, or reverse it on
some of the components of $M$, then it remains a balanced-total
bidirection and the types of the vertices of $\cM$
or edges of $\eG$ remain unchanged.
 \end{definition}

 Huggett and Moffatt \cite{bippartdualsHM} defined c- and d-edges.
 The idea of t-edges was used implicitly by Metsidik \cite{Met21} and
defined by Deng, Jin, and Metsidik \cite{dengjin}, as part of results we
will discuss in Section \ref{sec:eulerian}.
 The definition of b-edges is new and will be used later in this
section.
 In the literature cd-bidirections are called \emph{all-crossing}
bidirections, and cdt-bidirections are called \emph{crossing-total}
bidirections.
 We now state some results from \cite{dengjin,bippartdualsHM}.

 \begin{theorem}[Huggett and Moffatt {\cite[Theorem
1.3]{bippartdualsHM}}]\label{HMplane}
 Let $\eG$ be a planar graph embedding with medial graph $M$
and $A\subseteq E(G)$. Then the partial dual $\eG \du A$ is bipartite if and
only if there exists a direction of $M$ for which all elements of $A$
are c-edges of $\eG$ and all elements of $E(G)-A$ are d-edges of $\eG$.
 \end{theorem}

 Deng, Jin, and Metsidik \cite{dengjin} extended Theorem \ref{HMplane}
to all embedded graphs, orientable or nonorientable, using directions of
a `modified medial graph', which depends on the choice of local
rotations for the vertices of $\eG$.  Their result can also be expressed
using bidirections of $M$ with both directed and antidirected edges.
 We will not discuss the details, as there is a simpler way to handle
the general case, which we describe in Subsection \ref{ss:antidirections}.
 In the orientable case, Deng, Jin, and Metsidik's modified medial graph
is just the medial graph, and they obtain the following straightforward
generalization of Theorem \ref{HMplane}.

 \begin{theorem}[Deng, Jin, and Metsidik {\cite[Corollary
3.2]{dengjin}}]\label{bippddedge} 
 Let $\eG$ be an orientable graph embedding with medial graph
$M$ and $A\subseteq E(G)$. Then $\eG\du A$ is bipartite if and only if
there exists a direction of $M$ for which all elements of $A$ are
c-edges of $\eG$ and all elements of $E(G)-A$ are d-edges of $\eG$.
 \end{theorem}

 Taking $A = \emptyset$, Theorem \ref{bippddedge} implies that an
orientable embedded graph $\eG$ is bipartite if and only there is a
d-direction of its medial graph.
 But we know that an orientable $\eG$ is bipartite if and only if
$\eG\spe$ (or $\eG\spe\sdu$) is orientable, i.e., $\eG$ satisfies
condition $(010)$.  This suggests that there is a d-direction of $\eG$
if and only if condition $(010)$ holds.  One direction of this was
verified by Yan and Jin.

 \begin{theorem}[Yan and Jin {\cite[Lemma
6.2]{YJ19}}]\label{petrieoridedge}
 If a graph embedding $\eG$ satisfies condition $(010)$, i.e., $\eG\spe$
is orientable, then there exists a d-direction of $\eG$.
 \end{theorem}

 We will unify and extend the above results in the next three
subsections.

 \subsection{Medial checkerboard bidirections and
jewels}\label{ss:bidirections}

 In this subsection we establish a general connection between medial
checkerboard bidirections and embedding properties using binary labelings
of jewels.
 Theorems \ref{HMplane} and \ref{bippddedge} involve properties of
partial duals; we will show that characterization of the properties and
the effect of taking twisted duals (partial duals, partial Petrie duals,
and their compositions) can be considered independently.

 The following observation is obvious, but stated for later reference.

 \begin{observation}\label{bidirdir}
 Let $\eG$ be an embedded graph with medial graph $M$, and let
$\de$ be a bidirection of $M$.
 Every edge $e \in E(G)$ is a b-, c-, or d-edge of $\eG$ if and only if
$\de$ is a balanced bidirection, and every edge $e \in E(G)$ is a
t-edge of $\eG$ if and only if $\de$ is a total bidirection.
 \end{observation}

 The next observation is important for translating conditions between a
medial checkerboard and a jewel.

 \def\csimp{\cs{T}}
 \begin{observation}\label{medch2jewel}
 Suppose $\eG$ is an embedded graph with medial checkerboard $\cM$ and
jewel $\cL$.  Let $e \in E(G)$, let $v_e$ be the vertex of $\cM$
corresponding to $e$, and let $\csimp_e$ be the e-simplex of $\cL$
corresponding to $e$.
 \begin{enumerate}[label=\rm(\alph*)]\setlength{\itemsep}{0pt}
 \item\label{m2ja}
 Each half-edge $g$ of $M$ incident with $v_e$ corresponds to a
half-edge $g'$ of $L$ colored $\ca$ in $\cL$, and incident with a vertex of
$\csimp_e$.
 \item\label{m2jb}
 This gives a correspondence between bidirections $\de$ of $M$ and
bidirections $\de_L$ of the edges of $L$ colored $\ca$ in $\cL$: a half-edge
$g$ of $M$ is directed into or out of $v_e$ by $\de$ if and only
if $g'$ is directed into or out of $\csimp_e$ by $\de_L$,
respectively.
 \item\label{m2jc}
 Half-edges $g$ and $h$ of $\cM$ are consecutive in the rotation
around $v_e$ if and only if $g'$ and $h'$ are incident with consecutive
vertices $x$ and $y$ of the e-square in $\csimp_e$.  Moreover, the color
of the face between $g$ and $h$ in $\cM$ is the same as the
color of the edge of $\cL$ joining $x$ and $y$.
 \end{enumerate}
 \end{observation}

 Suppose that $\eG$ is an embedded graph with medial checkerboard $\cM$
and jewel $\cL$.  Using Observation \ref{medch2jewel}\ref{m2ja} and
\ref{m2jb}, we set up a correspondence between a bidirection $\de$
of $M$ and a binary labeling $\la$ of $L$ by assigning $0$ or $1$ to
each $x \in V(L)$ according to whether the unique bidirected half-edge
incident with $x$ is directed into or out of $x$ by $\de_L$,
respectively.  We write $\la = \tbdbl(\de)$, and $\tbdbl$ is a
bijection, so $\de = \tbdbl\iv(\la)$.

 Note that the correspondence between half-edges of $M$ and certain
half-edges of $L$ is not affected by taking twisted duals.  Thus, the
correspondence between $\de$, $\de_L$, and $\la=\tbdbl(\de)$ is also not
affected by taking twisted duals.  
 The type of a vertex of the medial checkerboard $\cM$ may change under
taking a twisted dual, but the correspondence described in Observation
\ref{medch2jewel}\ref{m2jc} holds both for the original embedding and
for the modified embedding.

 Given $\de$ and $\la = \tbdbl(\de)$, we know $\la$ is an
$S$-labeling for some unique $S \subseteq E(L)$.  Then $\de$ is a
direction if and only if for each $e \in \ec\ca\cL$ with endpoints $x$
and $y$, $\la(x) \ne \la(y)$, which is equivalent to $\ec\ca\cL
\subseteq S$.  Therefore, by Table \ref{tab:condoribip} the conditions
$(001)$, $(010)$, $(100)$, and $(111)$ 
correspond to medial checkerboard directions.
 Similarly, $\de$ is an antidirection if and only if $S \cap
\ec\ca\cL = \emptyset$, so the conditions $(011)$, $(101)$, and $(110)$
correspond to medial checkerboard antidirections.  We will consider the
directions first.

 \subsection{Medial checkerboard directions and Fano
properties}\label{ss:directions}

 We begin by considering condition $(100)$.  Much of the work is done by
the following lemma.

 \begin{lemma}\label{cedgefza}
 Let $\eG$ be a graph embedding with medial checkerboard $\cM$ and jewel
$\cL$.  Let $\de$ be a bidirection of $M$ and $\la=\tbdbl(\de)$ the
corresponding binary labeling of $\cL$.
 Then $\de$ is a c-direction of $\cM$ if and only if $\la$ is an
$\ec{\cf\cc\cz\cc\ca}\cL$-labeling of $\cL$.
 \end{lemma}

 \begin{proof}
 We reason as follows.

 \smallskip
 \rgp13 Under $\de$ every edge of $M$ is directed and every edge
of $\eG$ is a c-edge.

 \rgp13 $\ifoif$
 Under $\de$ every edge of $\cM$ is directed, and around a vertex the
cyclic order of edge directions and face colors is (in (towards the
vertex), $\cv$, in, $\cf$, out, $\cv$, out, $\cf$).

 \rgp13 $\ifoif$
 (Observation \ref{medch2jewel})\quad
 Under $\de_L$ each edge of color $\ca$ in $\cL$ is directed, and around
an e-square of $\cL$ the cyclic order of directions of incident
half-edges and edge colors is (in (towards the e-square), 
 $\cv$, in, $\cf$, out, $\cv$, out, $\cf$).

 \rgp13 $\ifoif$
 The binary labeling $\la=\tbdbl(\de)$ of $L$ changes value across edges
of color $\ca$ in $\cL$, and around an e-square of $\cL$ the cyclic
order of label values and edge colors is
 ($0$, $\cv$, $0$, $\cf$, $1$, $\cv$, $1$, $\cf$).
 so that values change across edges of colors $\cf$ and $\cz$, but do
not change across edges of color $\cv$.

 \rgp13 $\ifoif$
 $\la$ is an $\ec{\cf\cc\cz\cc\ca}\cL$-labeling of $\cL$.
 \end{proof}

 \begin{theorem}\label{cedge100}
 Let $\eG$ be a graph embedding with medial checkerboard $\cM$. Then
$\eG$ satisfies condition $(100)$, i.e., $\eG\sdu\spe$ and
$\eG\sdu\spe\sdu$ are orientable, if and only if there is a c-direction
of $\cM$.  Such a direction is a balanced direction, and is unique up to
reversal on components of $M$.
 \end{theorem}

 \begin{proof}
 Let $\cL$ be the jewel of $\eG$.  By Lemma \ref{cedgefza} there is a
c-direction $\de$ of $\cM$ if and only if $\cL$ has an
$\ec{\cf\cc\cz\cc\ca}\cL$-labeling $\la =\tbdbl(\de)$, which by
Table \ref{tab:condoribip} is true if and only if $\eG$ satisfies condition
$(100)$.

 If $\de$ exists, it is a balanced direction by Observation
\ref{bidirdir}.
 There are two choices for directing the edges incident with a given
vertex $v_e$ of $\cM$ so that $v_e$ is a c-vertex, which are reverses of
each other.  But once we choose one of these we fix the choice at all
neighbors of $v_e$, and by propagation, at all vertices in the same
component as $v_e$.  Thus, $\de$ is unique up to reversal on
components of $M$.
 \end{proof}

 We can also begin with a $4$-regular graph $M$ with a balanced
direction $\de$, and ask whether we can embed and face-color $M$ to
obtain a medial checkerboard with a bidirection as in Theorem
\ref{cedge100}, corresponding to an embedded graph $\eG$ for which
condition $(100)$ holds.  In general we obtain many such $\eG$.

 \def\sG{\mathcal{G}}
 \begin{theorem}\label{4reg100}
 Let $M$ be a $4$-regular graph with a balanced direction $\de$.  We can
embed and face-color $M$ to obtain at least one medial checkerboard
$\cM$ in which every vertex of $\cM$ is a c-vertex under $\de$.  The
medial checkerboards $\cM$ obtainable in this way correspond to an
equivalence class $\sG$ of embedded graphs under partial Petrie duality. 
Any other bidirected $4$-regular graph that corresponds to $\sG$ in this
way is isomorphic to $M$ with a bidirection obtained from $\de$ by
reversal on a subset of components of $M$.
 \end{theorem}

 \let\simp=T
 \begin{proof}
 Create a new $4$-regular graph $L$ from $M$ by expanding each vertex
$v$ of $M$ to a copy $\simp_v$ of $K_4$ (a \emph{simplex}), so that each
edge of $M$ incident with $v$ becomes incident with a distinct vertex of
$\simp_v$ in $L$.  Then $L$ has the structure of the underlying graph of
a jewel, and $M = L \cg S$ where $S$ is the set of simplex edges.
 We can obtain a binary labeling $\la=\tbdbl(\de)$ of $L$ from $\de$ in
the manner described earlier for checkerboards and jewels, with an
intermediate bidirection $\de_L$ of $A= E(L)-S$.

 In each $\simp_v$ two vertices will be labeled $0$ and two vertices
will be labeled $1$.  Color the edges of $A$ in $L$ with $\ca$, and
color the two edges of each simplex with equal labels on their
endvertices with $\cv$.  The remaining four edges of the simplex form
two matchings: color the edges of one matching with $\cf$ and the edges
of the other with $\cz$.  The result is a jewel $\cL$ for which $\la$ is
a $\ec{\cf\cc\cz\cc\ca}\cL$-labeling, corresponding to an embedded graph
$\eG$ satisfying condition $(100)$.  Since $M = L \cg
\ec{\cv\cc\cf\cc\cz}\cL$, $M$ is the medial graph of $\eG$, and so the
medial checkerboard $\cM$ of $\eG$ is obtained by embedding and
face-coloring $M$.
 Now $\de$ is a bidirection of $\cM$ and $\la = \tbdbl(\de)$ is an 
$\ec{\cf\cc\cz\cc\ca}\cL$-labeling of $\cL$.
 Thus, by Lemma \ref{cedgefza}, every edge of $\eG$ is a c-edge under
$\de$, as required.

 This construction allows us to swap the colors $\cf$ and $\cz$ on any
simplices of $\cL$, which corresponds to taking arbitrary partial
Petrie duals of $\eG$, forming an equivalence class $\sG$.
 Any $\eG \in \sG$ determines $M$, and by Theorem \ref{cedge100} it also
determines $\de$ up to reversal on components of $M$.
 \end{proof}

 Every $4$-regular graph has a balanced direction (take an Euler circuit
$T$ and direct every edge in the direction of $T$), so the proof of
Theorem \ref{4reg100} gives a method for constructing graph embeddings
with a given medial graph $M$ and satisfying condition $(100)$.
 Also, note that two embedded graphs in the class $\sG$ from Theorem
\ref{4reg100} always have the same medial graph $M$, but their medial
embeddings $\eM$ will usually be distinct.

 We can apply similar analyses to conditions $(001)$ and $(010)$.  We
omit the proofs, but state the results in Theorems \ref{bcdedge} and
\ref{4regdir} below.

 The fourth condition related to a medial checkerboard direction is
condition $(111)$.  This is equivalent to $\cM$, or $M$, being
bipartite.  While an argument using an $\ec\ca\cL$-labeling of $\cL$ can
be made, it is easy to see directly that this is equivalent to $\cM$
having a bidirection in which every edge of $\eG$ is a t-edge: the
bipartition of $\cM$ is just the partition into t\ps- and t\ms-vertices.
 For the counterpart of Theorem \ref{4reg100} for condition $(111)$, we
create a jewel $\cL$ with an $\ec\ca\cL$-labeling based on a total
direction of $M$.  In any given simplex, all vertices have the same
label, and we can color the three matchings in the simplex arbitrarily
with the colors $\cv, \cf, \cz$.  We can permute these colors
arbitrarily on each simplex, which corresponds to taking arbitrary
twisted duals of $\eG$.

  The results for all four conditions $(001)$, $(010)$, $(100)$, and
$(111)$ can be summarized by the following two theorems, which refer to
Table \ref{tab:bcd}.


 \def\operation#1{$\langle$operation$\rangle$}
 \begin{theorem}\label{bcdedge}
 Let $\eG$ be a graph embedding with medial checkerboard $\cM$ and jewel
$\cL$.  Choose $\gtype$, $\stype$, $S$ and $(\be)$ from the same row of Table
\ref{tab:bcd}. 
 \begin{enumerate}[label=\rm(\alph*)]\setlength{\itemsep}{0pt}
 \item
 Let $\de$ be a bidirection of $M$ and $\la=\tbdbl(\de)$ the
corresponding binary labeling of $L$.
 Then $\de$ is a $\stype$-direction of $\cM$ if and only if $\la$ is
an $S$-labeling of $\cL$.
 \item
 Condition $(\be)$ holds for $\eG$ if and only if there is a
$\stype$-direction of $\cM$.  Such a direction is a direction of general
type $\gtype$, and is unique up to reversal on components of $M$.
 \end{enumerate}
 \end{theorem}

 \def\sO{\mathcal{O}}
 \begin{table}[h]
 \def\hquad{\hskip0.5em\relax}
 \def\vstrut#1#2{\vrule height#1 depth#2 width0pt}
 \def\hdstrut{\vstrut{13pt}{5pt}}
 \begin{center}\rm
 \begin{tabular}{|l|l|l|l|l|}
 \hline
	\multicolumn{2}{|c|}{\hdstrut Direction of $\cM$}
	&&&
	\\
 \cline{1-2}
	\hdstrut General
        & Specific
	& $S$-labeling
	& Condition $(\be)$ and
	& Operation class $\sO$
	\\
	type $\gtype$
 	& type $\stype$
	& \quad of $\cL$
	& \quad property of $\eG$
	& \quad defining $\sG$
	\\
 \hline
 \vstrut{12pt}{4pt}%
	balanced
  & \quad b &	$\ec{\cv\cc\cf\cc\ca}\cL$-labeling
			& condition $(001)$, i.e., &	partial duality \\
 &&& $\eG$, $\eG\sdu$ orientable & \\
 \vstrut{12pt}{4pt}%
	balanced
  & \quad c &	$\ec{\cf\cc\cz\cc\ca}\cL$-labeling
			& condition $(100)$, i.e., &	partial Petrie duality \\
 &&& $\eG\sdu\spe$, $\eG\sdu\spe\sdu$ orientable & \\
 \vstrut{12pt}{4pt}%
	balanced
  & \quad d &	$\ec{\cv\cc\cz\cc\ca}\cL$-labeling
			& condition $(010)$, i.e., &	partial Wilson duality \\
 &&& $\eG\spe$, $\eG\spe\sdu$ orientable & \\
 \vstrut{12pt}{4pt}%
	total
  & \quad t &	$\ec{\ca}\cL$-labeling
			& condition $(111)$, i.e., &	twisted duality \\
 &&& $\cM$ bipartite & \\
 \hline
 \end{tabular}
 \caption{Correspondences for Theorems \ref{bcdedge} and
\ref{4regdir}.}%
 \label{tab:bcd}
 \end{center}
 \end{table}

 \begin{theorem}\label{4regdir}
 Choose $\gtype$, $\stype$, $(\be)$, and $\sO$ from the same row of
Table \ref{tab:bcd}.
 Let $M$ be a $4$-regular graph with a direction $\de$ of general type
$\gtype$; $M$ always has such a direction if $\gtype$ is `balanced'.

 We can embed and face-color $M$ to obtain at least one medial
checkerboard $\cM$ in which $\de$ is a $\stype$-direction, i.e., so that
condition $(\be)$ holds for the corresponding embedded graph $\eG$.  The
medial checkerboards $\cM$ obtainable in this way correspond to an
equivalence class $\sG$ of embedded graphs under the operations in
$\sO$.  Any other bidirected $4$-regular graph that corresponds to $\sG$
in this way is isomorphic to $M$ with a bidirection obtained from $\de$
by reversal on a subset of components of $M$.
 \end{theorem}

 The results summarized in Theorems \ref{bcdedge} and \ref{4regdir} are,
as far as we are aware, new, except for part of the result for
d-directions, stated earlier as Theorem \ref{petrieoridedge}, due to Yan
and Jin.
 In particular, there are no prior results involving b-edges.

 There are other ways to prove some of these results, as we already
noted above for condition $(111)$.  For example, we can show that if
$\eG$ is orientable, then there is a b-direction of $\cM$: just choose
one color, either $\cv$ or $\cf$, and direct all faces of $\eM$ of that
color in the clockwise direction for some orientation of the surface. 
 Yan and Jin used a similar argument to prove Theorem \ref{petrieoridedge}.
 Conversely, we can apply Theorem \ref{hao2imply3} to the medial
embedding $\eM$ to show that if there is a b-direction of $\cM$, then
$\eM$, and hence $\eG$, is orientable.

 \begin{figure}[tb]
     \centering
  \includegraphics{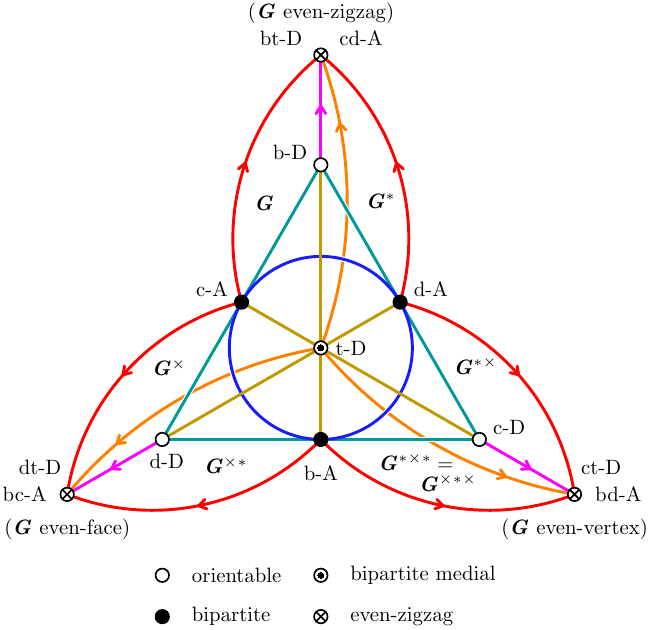}
     \caption{Extended Fano Framework (D = direction, A = antidirection).}
         \label{fig:FanoFrameworkExtended}
 \end{figure}

 We have now associated four of the seven Fano framework properties with
a property involving medial graph directions.  These results are
indicated on Figure \ref{fig:FanoFrameworkExtended}, an extended version
of Figure \ref{fig:FanoFramework}.
 Each of the central vertex and the open vertices is marked with the
type of direction that characterizes its property.  For example, the top
open vertex is marked with `b-D' to indicate that its property,
orientability of $\eG$ or $\eG\sdu$, corresponds to existence of a
b-direction of $\cM$.  Other parts of Figure
\ref{fig:FanoFrameworkExtended} will be explained later.

 \subsection{Twisted duality and balanced-total
bidirections}\label{ss:bidirtwistedduals}

 We next characterize the effect of partial duality and partial Petrie
duality on edge types for a bidirection $\de$ of the medial graph $M$ of an
embedded graph $\eG$.
 As noted in Subsection \ref{ss:directions}, a twisted duality operation
does not change $\de$ or the corresponding labeling $\la=\tbdbl(\de)$ of
the underlying graph $L$ of the jewel.
 However, the types of vertices of the medial checkerboard may change.
 Observation \ref{medch2jewel}\ref{m2jc} allows us to transfer
information about changes in the jewel to information about vertex types
in the medial checkerboard, i.e. edge types in the embedded graph.

 \begin{theorem}\label{typechange}
 Suppose we have an embedded graph $\eG$, a fixed balanced-total
bidirection $\de$ of its medial graph $M$, and $A \subseteq
E(G)$.
 In the partial dual $\eG\du A$, or partial Petrie dual $\eG\pe A$, the
type of each edge of $\eG$ not in $A$ does not change, and the type of
each edge in $A$ changes according to Table \ref{tab:typechange}.

 \begin{table}[h]
 \def\hquad{\hskip0.5em\relax}
 \def\vstrut#1#2{\vrule height#1 depth#2 width0pt}
 \def\hdstrut{\vstrut{13pt}{5pt}}
 \begin{center}\rm
 \begin{tabular}{|l||l|l|}
 \hline
 \multicolumn3{|c|}{\hdstrut Effect on edges in $A$} \\
 \hline
 \hdstrut Type in $\eG$ & Type in $\eG*A$ & Type in $\eG \ti A$ \\
 \hline
 \vstrut{13pt}{0pt}%
   \hquad b-edge & \quad b-edge & \quad d-edge \\
 \hquad c-edge & \quad d-edge & \quad c-edge \\
 \hquad d-edge & \quad c-edge & \quad b-edge \\
 \hquad t\ps-edge & \quad t\ps-edge & \quad t\ps-edge \\
 \vstrut{0pt}{5pt}%
   \hquad t\ms-edge & \quad t\ms-edge & \quad t\ms-edge \\
 \hline
 \end{tabular}
 \caption{Type changes under partial duality and partial Petrie duality.}
 \label{tab:typechange}
 \end{center}
 \end{table}
 \end{theorem}

 \begin{proof}
 We examine one case; the others are similar.
 Let $\cL$ be the jewel of $\eG$, and $\de_\cL$ the bidirection of
$\ec\ca\cL$ corresponding to $\de_\cM$.
 Suppose $e \in A$ is a b-edge in $\eG$, corresponding to an e-simplex
in $\cL$ containing the e-square $\cQ_e=(u_0 u_1 u_2 u_3)$.
 Applying Observation \ref{medch2jewel}, we may assume that edges $u_0
u_1$ and $u_2 u_3$ are colored $\cv$, edges $u_1 u_2$ and $u_3 u_0$ are
colored $\cf$, and under $\de_\cL$ 
 edges of $\ec\ca\cL$ are  directed into $u_0$ and $u_2$, and out of
$u_1$ and $u_3$.  In the jewel $\cL[\pe A]$ of $\eG \pe A$, the colors
$\cf$ and $\cz$ are swapped on the e-simplex, and the e-square of the
recolored e-simplex is $\cQ'_e = (u_0 u_1 u_3 u_2)$, where $u_0 u_1$ and
$u_2u_3$ are still colored $\cv$, while $u_1 u_3$ and $u_0 u_2$ are
colored $\cf$.  Therefore, around $\cQ'_e$ we have (in ($u_0$), $\cf$,
in ($u_2$), $\cv$, out ($u_3$), $\cf$, out ($u_1$), $\cv$).  Applying
Observation \ref{medch2jewel}, we see that $e$ is now a d-edge.
 \end{proof}

 A special case of Theorem \ref{typechange} was given by Deng, Jin, and
Metsidik \cite[Lemma 2.4]{dengjin}.  They showed that under full duality
all c-edges become d-edges, and vice versa, while t-edges remain
t-edges.

 \let\symdiff\triangle
 Using Theorem \ref{typechange} we can track the effect on edge types
(relative to a fixed medial graph direction) for any sequence of partial
dual and partial Petrie dual operations.  This has many consequences; we
provide a few examples.
 \newlength{\realparindent}
 \setlength{\realparindent}{\parindent}
 \begin{enumerate}[label=\rm(\arabic*)]\setlength{\parskip}{0pt}
	\setlength{\parindent}{\realparindent}
 \item
 Taking the partial Wilson dual with respect to $A$, $\eG \mapsto \eG
\pe\du\pe A = \eG \wi A$, interchanges b- and c-edges for edges in $A$
and leaves other edges and types unchanged.
 \item
  The operation $\eG \mapsto \eG \du A \pe B$ permutes the
edge type as (bdc) for edges in $A \cap B$ (b-edges become d-edges,
d-edges become c-edges, c-edges become b-edges), as (cd) for edges in $A
- B$, and as (bd) for edges in $B - A$.
 \item
 An alternative way to prove the parts of Theorems \ref{bcdedge} and
\ref{4regdir} for b- and d-directions is to derive them from c-direction
results (Lemma \ref{cedgefza} and Theorems \ref{cedge100} and
\ref{4reg100}) by using Theorem \ref{typechange} (possibly more than
once) with $A = E(G)$.
 \item
 We can prove results of the same flavor as Theorems \ref{HMplane} and
\ref{bippddedge}, characterizing when an embedded graph has a particular
type of twisted dual that is orientable.  The conditions for an
orientable partial dual may be summarized as follows.

 \begin{theorem}\label{oritwdexistence}
 Let $\eG$ be an embedded graph with medial checkerboard $\cM$.  Then a
twisted dual of $\eG$ is orientable if and only if there exists a
direction of $\cM$ satisfying the following conditions.
 \begin{enumerate}[label=\rm(\roman*)]\setlength{\itemsep}{0pt}
 \item The edges of $\eG$ that are unchanged or operated on by $\du$ are
b-edges,
 \item the edges of $\eG$ that are operated on by $\du\pe$ or
$\du\pe\du$ are c-edges, and
 \item the edges of $\eG$ operated on by $\pe$ or $\pe\du$ are d-edges.
 \end{enumerate}
 \end{theorem}

 \begin{proof}
 Let $\eG'$ be the twisted dual.
 If we have a balanced direction of $\cM$ satisfying the above
conditions, then by Theorem \ref{typechange} all edges become b-edges
after taking the twisted dual, so by Theorem \ref{bcdedge} $\eG'$ is
orientable.  Conversely, if $\eG'$ is orientable then there is a
b-direction of its medial checkerboard $\cM'$, and after reversing the
twisted dual we see that the edges of $\eG$ have the specified
types.
 \end{proof}

 \item
 We see that Theorem \ref{bippddedge}, and its
special case Theorem \ref{HMplane}, regarding bipartiteness of partial
duals, are really a combination of three separate ideas.
 First, an orientable embedded graph $\eG$ is bipartite if and only if
$\eG\spe$ is orientable, which is part of the Fano framework.  Second,
$\eG\spe$ is orientable if and only if $\eG$ has a medial
checkerboard direction where every edge of $\eG$ is a d-edge, by Theorem
\ref{bcdedge}.
 And third, partial duality interchanges c-edges and d-edges in the
dualized set of edges, by Theorem \ref{typechange}.
 We will see a more direct way to address bipartiteness in Subsection
\ref{ss:antidirections}.
 \end{enumerate}
 
 \subsection{Medial checkerboard antidirections and Fano properties}
 \label{ss:antidirections}

 As noted at the end of Subsection \ref{ss:bidirections}, conditions for
an embedded graph $\eG$ involving an $S$-labeling of the jewel $\cL$
where $S \cap \ec{\ca}\cL = \emptyset$ correspond to antidirections of
the medial checkerboard.
 Here we consider b-, c-, and d-antidirections of $\cM$.  We could also
consider t-antidirections, but every medial checkerboard has a
t-antidirection obtained by antidirecting each edge in the same way (all
introverted, or all extraverted).  Thus, the existence of a
t-antidirection corresponds to the always-true condition $(000)$.

 We can apply the same type of analysis to b-, c-, and d-antidirections
of $\cM$ that we applied to b-, c-, and d-directions of $\cM$ in
Subsection \ref{ss:directions}. The arguments are very similar.  For
example, we saw that a c-direction of $\cM$ corresponds to an
$\ec{\cf\cz\ca}\cL$-labeling of $\cL$.  By much the same reasoning, a
c-antidirection of $\cM$ corresponds to an $\ec{\cf\cz}\cL$-labeling of
$\cL$.  
 The results for b-, c-, and d-antidirections and the respective three
conditions $(110)$, $(011)$, and $(101)$ can be summarized by Theorems
\ref{bcdantiedge} and \ref{4regantidir}, which refer to Table
\ref{tab:bcdanti}.
 Since the proofs are very similar to those of Theorems \ref{bcdedge}
and \ref{4regdir}, we omit them, apart from a discussion of the
existence of balanced antidirections before Theorem \ref{4regantidir}.


 \def\type{$\tau$}
 \def\operation#1{$\langle$operation$\rangle$}
 \begin{theorem}\label{bcdantiedge}
 Let $\eG$ be a graph embedding with medial checkerboard $\cM$ and jewel
$\cL$.  Choose $\gtype$, $\stype$, $S$ and $(\be)$ from the same row of Table
\ref{tab:bcdanti}. 
 \begin{enumerate}[label=\rm(\alph*)]\setlength{\itemsep}{0pt}
 \item
 Let $\de$ be a bidirection of $M$ and $\la=\tbdbl(\de)$ the
corresponding binary labeling of $L$.
 Then $\de$ is a $\stype$-antidirection of $\cM$ if and only if $\la$ is
an $S$-labeling of $\cL$.
 \item
 Condition $(\be)$ holds for $\eG$ if and only if there is a
$\stype$-antidirection of $\cM$.  Such an antidirection is an
antidirection of general type $\gtype$ (here, always balanced), and is
unique up to reversal on components of $\cM$.
 \end{enumerate}
 \end{theorem}

 \def\sO{\mathcal{O}}
 \begin{table}[h]
 \def\hquad{\hskip0.5em\relax}
 \def\vstrut#1#2{\vrule height#1 depth#2 width0pt}
 \def\hdstrut{\vstrut{13pt}{5pt}}
 \begin{center}\rm
 \begin{tabular}{|l|l|l|l|l|}
 \hline
	\multicolumn{2}{|c|}{\hdstrut Antidirection of $\cM$}
	&&&
	\\
 \cline{1-2}
	\hdstrut General
        & Specific
	& $S$-labeling
	& Condition $(\be)$ and
	& Operation class $\sO$
	\\
	type $\gtype$
 	& type $\stype$
	& \quad of $\cL$
	& \quad property of $\eG$
	& \quad defining $\sG$
	\\
 \hline
 \vstrut{12pt}{4pt}%
	balanced
  & \quad b &	$\ec{\cv\cc\cf}\cL$-labeling
			& condition $(110)$, i.e., &	partial duality \\
 &&& $\eG\spe\sdu$, $\eG\sdu\spe\sdu$ bipartite & \\
 \vstrut{12pt}{4pt}%
	balanced
  & \quad c &	$\ec{\cf\cc\cz}\cL$-labeling
			& condition $(011)$, i.e., &	partial Petrie duality \\
 &&& $\eG$, $\eG\spe$ bipartite & \\
 \vstrut{12pt}{4pt}%
	balanced
  & \quad d &	$\ec{\cv\cc\cz}\cL$-labeling
			& condition $(101)$, i.e., &	partial Wilson duality \\
 &&& $\eG\sdu$, $\eG\sdu\spe$ bipartite & \\
 \hline
 \end{tabular}
 \caption{Correspondences for Theorems \ref{bcdantiedge} and
\ref{4regantidir}.}%
 \label{tab:bcdanti}
 \end{center}
 \end{table}

 Every $4$-regular graph $M$ has an even number of edges, and therefore,
by designating edges as introverted or extraverted in an alternating way
as we follow an Euler circuit in $M$, we see that $M$ has at least one
balanced antidirection.  Thus, we can use reasoning similar to that of
Theorem \ref{4reg100} and the associated discussion to construct graph
embeddings with a given medial graph $M$ and satisfying one of the
conditions $(011)$, $(101)$ or $(110)$.

 \begin{theorem}\label{4regantidir}
 Choose $\gtype$, $\stype$, $(\be)$, and $\sO$ from the same row of
Table \ref{tab:bcdanti}.
 Let $M$ be a $4$-regular graph with an antidirection $\de$ of general type
$\gtype$ (here, always balanced); $M$ always has such an antidirection.

 We can embed and face-color $M$ to obtain at least one medial
checkerboard $\cM$ in which $\de$ is a $\stype$-antidirection, i.e., so
that condition $(\be)$ holds for the corresponding embedded graph $\eG$.
 The medial checkerboards $\cM$ obtainable in this way correspond to an
equivalence class $\sG$ of embedded graphs under the operations in
$\sO$.  Any other bidirected $4$-regular graph that corresponds to $\sG$
in this way is isomorphic to $M$ with a bidirection obtained from $\de$
by reversal on a subset of components of $M$.
 \end{theorem}

 The results summarized in Theorems \ref{bcdantiedge} and
\ref{4regantidir} are, as far as we are aware, new.

 We have now associated three of the seven Fano framework properties
with a property involving medial graph antidirections.
 Each solid vertex of Figure \ref{fig:FanoFrameworkExtended} is marked
with the type of antidirection that characterizes its property.  For
example, the solid vertex on the left is marked with `c-A' to indicate
that its property, bipartiteness of $\eG$ or $\eG\spe$, corresponds to
existence of a c-antidirection of $\cM$.

 Antidirections of $M$ are really just binary labelings of the edges
of $M$ (e.g., extraverted edges are $0$, introverted edges are $1$). 
We can also think of them as signings of the edges (e.g., extraverted
edges are $-$, introverted edges are $+$) or as colorings of the edges
(e.g., extraverted edges are black, introverted edges are white). 

Thinking of antidirections as colorings gives more intuitive proofs of
Theorem \ref{bcdantiedge}. For example, if we think of a c-antidirection
of $\cM$ as a black/white coloring of the edges of $\cM$, the monochrome
cycles in $\cM$ are precisely the boundaries of the faces of $\cM$
colored $\cv$, and each vertex $v_e$ is contained in both a black cycle
and a white cycle.  In other words, we have a black/white coloring of
the vertices of $\eG$, and each edge $e$ of $\eG$ joins a black vertex
and a white vertex, meaning that $\eG$ is bipartite.

 Theorem \ref{typechange} applies to antidirections of the medial graph,
as well as to directions.  
 We can use this in conjunction with Theorem \ref{bcdantiedge} to give a
general characterization of embedded graphs with bipartite twisted
duals, similar to the characterization of orientable twisted duals in
Theorem \ref{oritwdexistence}.
 Unlike Theorems \ref{HMplane} and \ref{bippddedge}, it uses
antidirections, rather than directions, of the medial checkerboard,
it applies to nonorientable embeddings as well as orientable embeddings,
and it applies to all twisted duals, not just partial duals.

 \begin{theorem}\label{biptwdexistence}
 Let $\eG$ be an embedded graph with medial checkerboard $\cM$.  Then a
twisted dual of $\eG$ is bipartite if and only if there exists an
antidirection of $M$ satisfying the following conditions.
 \begin{enumerate}[label=\rm(\roman*)]\setlength{\itemsep}{0pt}
 \item The edges of $\eG$ that are unchanged or operated on by $\pe\du$
or $\du\pe\du$ are b-edges,
 \item the edges of $\eG$ that are unchanged or operated on by $\pe$ are c-edges, and
 \item the edges of $\eG$ operated on by $\du$ or $\du\pe$ are d-edges.
 \end{enumerate}
 \end{theorem}

 \begin{proof}
 Let $\eG'$ be the twisted dual.
 If we have a balanced antidirection of $M$ satisfying the above
conditions, then by Theorem \ref{typechange} all edges become c-edges
after taking the twisted dual, so by Theorem \ref{bcdedge} $\eG'$ is
bipartite.  Conversely, if $\eG'$ is orientable then there is a
c-antidirection of its medial checkerboard $\cM'$, and after reversing
the twisted dual we see that the edges of $\eG$ have the specified
types.
 \end{proof}

 We can state similar results for twisted duals that are
$2$-face-colorable or $2$-zigzag-colorable; we leave the details to the
reader.
 Theorem \ref{biptwdexistence} and these related results allow us to
answer a question of Ellis-Monaghan and Moffatt \cite[(1),
pp.~1566-1567]{twisteddualEMM}, as to which graph embeddings have a
checkerboard colorable (i.e., $2$-face-colorable) or bipartite twisted
dual.

 \begin{corollary}
 Every embedded graph $\eG$ has a twisted dual that is bipartite, a
twisted dual that is $2$-face-colorable, and a twisted dual that is
$2$-zigzag-colorable.
 \end{corollary}

 \begin{proof}
 Construct an arbitrary balanced antidirection $\de$ of the medial graph
$M$ of $\eG$, which we know exists.  Then apply a twisted dual operation
to the edges of $\eG$ following the specifications in Theorem
\ref{biptwdexistence}(i)--(iii).  The result is a bipartite twisted dual
of $\eG$.  Twisted duals that are $2$-face-colorable or
$2$-zigzag-colorable can be found in similar ways.
 \end{proof}

 \section{Eulerian properties}\label{sec:eulerian}

 In this subsection we consider the property of being Eulerian, or more
generally even-vertex, for an embedded graph $\eG$.  (In the literature
in this area, `Eulerian' is often used to mean `even-vertex', but we
reserve `Eulerian' for connected graphs.)  We also consider
the related properties of being even-face or even-zigzag.  
 We will henceforth refer to the properties even-vertex, even-face, and
even-zigzag as \emph{Eulerian properties}.
 The Eulerian properties are weaker than the seven Fano properties, in
the following way.
 Each of the seven properties implies at least one of these three new
properties (as we show in this section), but $\eG$ may have all three
new properties and none of the seven properties (as demonstrated later
by an example).

 We will find several ways to characterize the three Eulerian
properties.
 First we discuss how these properties can be represented in terms of
parity conditions on closed walks.
 Then we show that these properties can be described in terms of both
directions and antidirections of the medial graph.  These descriptions
allow us to show that each of the seven Fano properties implies one,
two, or all three of the Eulerian properties.

 \subsection{Parity conditions for 2-color walks}\label{ss:2colwalks}

 To extend the idea of considering the parity of closed walks in a gem
or jewel to cover the Eulerian properties, we must restrict the closed
walks we consider to \emph{$2$-color} walks, which use edges of only two
colors (any two colors are allowed).

 To consider parity conditions for sets of edges in $2$-color closed
walks, by Theorem \ref{oddwalkoddcycletheorem} we may restrict our
attention to $2$-color cycles. These are the v-gons, f-gons, and
e-squares of a gem $\cJ$ or jewel $\cL$, and for $\cL$ also the z-gons,
the $4$-cycles of color $\cv$ and $\cz$, and the $4$-cycles of color
$\cf$ and $\cz$.
 All $4$-cycles $K$ using two of the colors $\cv, \cf, \cz$ in either
$\cJ$ or $\cL$ use an even number of edges of each of the four colors
$\cv, \cf, \cz, \ca$, so we may ignore these and consider only the
v-gons, f-gons, and (for $\cL$) z-gons.
 The degree of a vertex of $\eG$ is the number of edges colored $\cv$ in
the corresponding v-gon, or equivalently the number of edges colored
$\ca$ in that v-gon.

 We can describe all three Eulerian properties using conditions on
$2$-color walks in the jewel, but only two of them using such conditions
for the gem.  We describe the conditions for the jewel first.

 The proposition below characterizes when we have one, two, or all three
of the Eulerian properties.  It partly follows from the comments in
the previous paragraph.
 To obtain the final part of each of \ref{eja}, \ref{ejb}, or \ref{ejc},
we recall that all closed $2$-color walks have even length, so the
number of edges of a given color is even if and only if the number of
edges not of that color is even.

 \begin{proposition}\label{eulerianjewel}
 Let $\eG$ be an embedded graph with jewel $\cL$.
 \begin{enumerate}[label=\rm(\alph*)]\setlength{\itemsep}{0pt}%
     \item\label{eja}
 	$\eG$ is even-vertex if and only if $v_\cL(K)$ is even for all
$2$-color closed walks $K$ in $\cL$ if and only if $(f+z+a)_\cL(K)$ is
even for all $2$-color closed walks in $\cL$.
     \item\label{ejb}
 	$\eG$ is both even-vertex and even-face if and only if
$(v+f)_\cL(K)$ is even for all $2$-color closed walks $K$ in $\cL$ if
and only if $(z+a)_\cL(K)$ is even for all $2$-color closed walks $K$ in
$\cL$.
     \item\label{ejc}
 	$\eG$ is even-vertex, even-face, and even-zigzag if and only if
$(v+f+z)_\cL(K)$ is even for all $2$-color closed walks $K$ in $\cL$ if
and only if $a_\cL(K)$ is even for all $2$-color closed walks $K$ in
$\cL$.
 \end{enumerate}
 \end{proposition}

 We can permute the colors to obtain a condition for even-face
embeddings or even-zigzag embeddings from \ref{eja}, and a condition for
even-vertex and even-zigzag, or even-face and even-zigzag, embeddings
from \ref{ejb}.  Therefore, each Eulerian property, and each nonempty
combination of those properties, can be checked using a single parity
condition for $2$-color walks in the jewel.

 If we work with the gem $\cJ$ instead of the jewel $\cL$, we have the
following.

 \begin{proposition}\label{euleriangem}
 Let $\eG$ be an embedded graph with gem $\cJ$.
 \begin{enumerate}[label=\rm(\alph*)]\setlength{\itemsep}{0pt}%
     \item\label{ega}
 	$\eG$ is even-vertex if and only if $v_\cJ(K)$ is even for all
$2$-color closed walks $K$ in $\cJ$ if and only if $(f+a)_\cJ(K)$ is
even for all $2$-color closed walks in $\cJ$.
     \item\label{egb}
 	$\eG$ is both even-vertex and even-face if and only if
$(v+f)_\cJ(K)$ is even for all $2$-color closed walks $K$ in $\cJ$ if
and only if $a_\cJ(K)$ is even for all $2$-color closed walks $K$ in
$\cJ$.
 \end{enumerate}
 \end{proposition}

 We can permute the colors to obtain a condition for even-face
embeddings from \ref{ega}.  However, we cannot obtain conditions
involving the even-zigzag property by using $2$-color walks in $\cJ$.

 Proposition \ref{eulerianjewel} tells us that certain Eulerian
properties follow from the seven Fano properties, because if a
particular parity condition holds for all closed walks in $\cL$, it
certainly holds for the $2$-color closed walks in $\cL$.  We can deduce
the following.  We already proved \ref{iea} directly when examining
condition $(100)$ in Subsubsection \ref{cond100}, and as noted earlier,
this was observed by Lins \cite[Theorem 2.9(b)]{LinsPhD}.

 \begin{corollary}\label{implyeulerian}
 Let $\eG$ be an embedded graph.
 \begin{enumerate}[label=\rm(\alph*)]\setlength{\itemsep}{0pt}%
     \item\label{iea}
 	If condition $(100)$ holds for $\eG$, i.e., $\eG\sdu\spe$ and
$\eG\sdu\spe\sdu$ are orientable, then $\eG$ is even-vertex.
     \item\label{ieb}
 	If condition $(110)$ holds for $\eG$, i.e., $\eG\spe\sdu$ and
$\eG\sdu\spe\sdu$ are bipartite, then $\eG$ is both even-vertex and
even-face.
     \item\label{iec}
 	If condition $(111)$ holds for $\eG$, i.e., the medial graph $M$
of $\eG$ is bipartite, then $\eG$ is even-vertex, even-face, and
even-zigzag.
 \end{enumerate}
 \end{corollary}

 \begin{proof} (a) If condition $(100)$ holds, then $(v+f+a)_\cL(K)$ is
even for all closed walks $K$, and hence for $2$-color $K$; therefore
$\eG$ is even-vertex by Proposition \ref{eulerianjewel}\ref{eja}.

 The proofs of (b) and (c) are similar: for (b) we use the function
$(v+f)_\cL$, and for (c) we use the function $a_\cL$.  (Each part of
Proposition \ref{eulerianjewel} gives two functions; from Subsection
\ref{ss:gemsvsjewels}) we see that only one corresponds to a Fano
condition.)
 \end{proof}

 By permuting colors in the jewel we can find eulerian properties
implied by all seven Fano properties.  The implications are shown in
Figure \ref{fig:FanoFrameworkExtended}.  We will prove these implications
again using bidirections of the medial graph, in the next
subsection.

 \subsection{Medial graph bidirections and Eulerian
properties}\label{ss:eulerianbidirections}

 Huggett and Moffatt \cite{bippartdualsHM} considered not only the
question of when a partial dual is bipartite, as described in Subsection
\ref{ss:previousdirections}, but also when it is Eulerian.
 They were able to give a sufficient condition \cite[Corollary
4.5]{bippartdualsHM} involving directions of the medial checkerboard
$\cM$ for a partial dual of a plane graph to be Eulerian, but it was not
a necessary condition (the condition was actually equivalent to $\eG$
being $2$-face-colorable).
 Metsidik and Jin \cite{MJ18} found a condition that was both necessary
and sufficient, involving bidirections of $\cM$ with both directed and
antidirected edges.
 These bidirections had c- and d-edges and satisfied additional
conditions.
 Metsidik \cite{Met21} generalized this result, using similar
bidirections with c-, d-, and t-edges to characterize when the partial
dual of an arbitrary graph embedding is even-vertex.  (Note that
\cite{Met21} employs nonstandard terminology, using `c-edge' to mean
either c- or t-edge, and `d-edge' to mean either d- or t-edge.)

 Deng, Jin, and Metsidik realized that simpler conditions could be given
just using directions of $\cM$, as follows.

 \begin{theorem}[Deng, Jin, and Metsidik {\cite[Theorem
1.5]{dengjin}}]\label{Eulerianallct}
 Let $\eG$ be a cellularly embedded graph with medial graph $M$ and
$A\subseteq E(\eG)$. Then $\eG \du A$ is even-vertex if and only if there
exists a direction of $M$ for which every element of $A$ is a d-edge or
a t-edge in $\eG$, and every element of $E(G)- A$ is a c-edge or a
t-edge in $\eG$.
 \end{theorem}

 By applying the special case of Theorem \ref{typechange} for full
duality, Deng, Jin, and Metsidik \cite[Corollary 3.5]{dengjin} also
obtained a condition for partial duals that are even-face.
 If we take $A=\emptyset$ or $A = E(G)$ in Theorem \ref{Eulerianallct},
we obtain the following. 

 \begin{corollary}\label{eulerianctedge}
 Let $\eG$ be a cellularly embedded graph with medial checkerboard
$\cM$.  Then $\eG$ is even-vertex if and only if there exists a
ct-direction of $\cM$.
 \end{corollary}

 \begin{corollary}\label{evenfacedtedge}
 Let $\eG$ be a cellularly embedded graph with medial checkerboard
$\cM$.  Then $\eG$ is even-face if and only if there exists a
dt-direction of $\cM$.
 \end{corollary}

 In keeping with our theme of breaking results down to simpler
components, we observe that Theorem \ref{Eulerianallct} is just a
consequence of Corollary \ref{eulerianctedge} and Theorem
\ref{typechange}.
 We can also use Theorem \ref{typechange} to characterize the situation
where $\cM$ has a bt-direction.

 \begin{theorem}\label{btedge}
 Let $\eG$ be a cellularly embedded graph with medial graph $M$.  Then $\eG$
is even-zigzag if and only if there exists a bt-direction of $\cM$.
 \end{theorem}

 \begin{proof}
 The zigzag walks of $\eG$ are the faces of $\eG\spe$, so $\eG$ is
even-zigzag if and only if $\eG\spe$ is even-face, which by Corollary
\ref{evenfacedtedge} means the medial graph of $\eG\spe$ has an
orientation for which all edges of $\eG\spe$ are d- or t-edges, which by
Theorem \ref{typechange} means all edges of $\eG$ are b- or t-edges.
 \end{proof}

 Consequently, situations where we have a b-, c-, d-, or t-direction of
the medial checkerboard $\cM$, representing four of the properties in
the Fano framework, imply that we have one or more of the properties
even-vertex, even-face, or even-zigzag.
 For example, if we have conditions $(001)$ or $(111)$ we have a
b-direction or a t-direction of $\cM$, which in either case is a
bt-direction, so by Theorem \ref{btedge} $\eG$ is even-zigzag.
 The implications are shown by
directed edges (magenta or orange, if color is shown) in Figure
\ref{fig:FanoFrameworkExtended}.  The results we obtain from this
argument are the same as those we obtain from part \ref{iea}, and its
generalizations, and part \ref{iec} of Corollary \ref{implyeulerian}.

 Somewhat surprisingly, the Eulerian properties also correspond to
combinations of antidirection properties, so we also get implications
for the remaining three properties of the Fano framework, shown by
directed edges (red, if color is shown) in Figure
\ref{fig:FanoFrameworkExtended}.
 We now prove the necessary results.

 \begin{theorem}\label{bdantidir}
 Let $\eG$ be a cellularly embedded graph with medial checkerboard
$\cM$.  Then $\eG$ is even-vertex if and only if there exists a
bd-antidirection of $\cM$.
 \end{theorem}

 \begin{proof}
 Let $\cJ$ be the gem of $\eG$.

 First suppose $\eG$ is even-vertex.  Then we can antidirect the edges
colored $\ca$ in each v-gon of $\cJ$ so that they are alternately
introverted and extraverted.  Thus, each edge colored $\cv$ in an
e-square $Q_e$ joins two half-edges of color $\ca$ that have opposite
directions relative to $Q_e$ and so around $Q_e$ we see either (in,
$\cv$, out, $\cf$, in, $\cv$, out, $\cf$), or (in, $\cv$, out, $\cf$,
out, $\cv$, in, $\cf$).  By Observation \ref{medch2jewel}, this means
that $e$ is either a b-edge or a d-edge in $\eG$.

 Conversely, suppose that $\cM$ has a bd-antidirection.  Then
Observation \ref{medch2jewel} guarantees that each edge of color $\cv$
in an e-square $Q_e$ of $\cJ$ joins two half-edges of color $\ca$ that
have opposite directions relative to $Q_e$.  Therefore, around each
v-gon of $J$ the edges of color $\ca$ must alternate between introverted
and extraverted, and hence each vertex of $\eG$ has even degree.
 \end{proof}

 By applying Theorem \ref{typechange} to Theorem
\ref{bdantidir} we obtain the following.

 \begin{theorem}\label{bcantidir}
 Let $\eG$ be a cellularly embedded graph with medial checkerboard
$\cM$.  Then $\eG$ is even-face if and only if there exists a
bc-antidirection of $\cM$.
 \end{theorem}

 \begin{theorem}\label{cdantidir}
 Let $\eG$ be a cellularly embedded graph with medial checkerboard
$\cM$.  Then $\eG$ is even-zigzag if and only if there exists a
cd-antidirection of $\cM$.
 \end{theorem}

 Therefore, situations where we have a b-, c-, or d-antidirection of the
medial checkerboard $\cM$, representing the remaining three properties
in the Fano framework, imply that we have two of the properties
even-vertex, even-face, or even-zigzag.
 For example, if $\eG$ satisfies condition $(011)$, i.e., $\eG$ and
$\eG\spe$ are bipartite, then $\cM$ has a c-antidirection.  Therefore,
$\cM$ has both a bc-antidirection and a cd-antidirection, meaning that
$\eG$ is both even-face and even-zigzag by Theorems \ref{bcantidir} and
\ref{cdantidir}.  The implications we obtain in this way are identical
to those given by part \ref{ieb} of Corollary \ref{implyeulerian} and
its generalizations.

 At this point we can fully explain Figure
\ref{fig:FanoFrameworkExtended}.  There are ten vertices, seven of which
represent our Fano framework properties, and three new crossed vertices
that represent Eulerian properties: whether the embedded graph $\eG$
is even-vertex, even-face, or even-zigzag.  The edges between the solid,
open and crossed vertices (teal, magenta and red, if colors are shown)
form six triangles, and each is labeled in its interior with an
embedded graph $\eH$ ($\eH=\eG$, $\eG\sdu$, etc.).  The open vertex of
the triangle represents $\eH$ being orientable, the solid vertex
represents $\eH$ being bipartite, and the crossed vertex represents
$\eH$ being even-zigzag (not the easiest property to interpret, but what
makes the figure work nicely).  For example, at top left is a triangle
with the label `$\eG$' inside it; the crossed vertex of this triangle
represents $\eG$ being even-zigzag; the open vertex represents $\eG$
being orientable (condition $(001)$); and the solid vertex represents
$\eG$ being bipartite (condition $(011)$).
 The directed edges represent implications between properties, and the
undirected edges represent the Fano framework, as in Figure
\ref{fig:FanoFramework}.

 We can combine the results above with Theorem \ref{typechange} to
obtain results about whether graphs obtained by taking an arbitrary
sequence of partial duals and partial Petrie duals are even-vertex,
even-face, or even-zigzag.

 \section{Examples}\label{sec:examples}

 We now give examples of graph embeddings satisfying all possible
combinations of the seven Fano properties.
 We will also provide an example to show that it is possible to have all
three Eulerian properties without having any of the seven Fano
properties.

 Some of our examples are graphs embedded in the projective plane
(represented as a disk with antipodal boundary points identified) or
torus (represented as a rectangle with opposite sided identified). 
Other examples are represented as \emph{rotation projections} in the
plane: the graph is drawn in the plane so that the rotation at each
vertex is just the clockwise cyclic order of half-edges; edges of
signature $-1$ are marked with an X; and edge crossings should be
ignored.  Some examples are represented as rotation projections but in
the torus rather than the plane.

 \subsection{Examples of allowable combinations of Fano properties}
 \label{ss:excombfano}

 Here we provide examples showing that all combinations of Fano
properties allowed by Metatheorem A actually occur.

 \subsubsection{All seven Fano properties}\label{s2s:allseven}

 A grid embedding of $C_{2k} \times C_{2l}$ in the torus
(Figure \ref{fig:all7}) satisfies all seven properties. It is
orientable, bipartite, and 2-face-colorable. So the embedding satisfies
conditions $(001)$, $(011)$, and $(101)$, and by Metatheorem C all seven
properties are satisfied.

 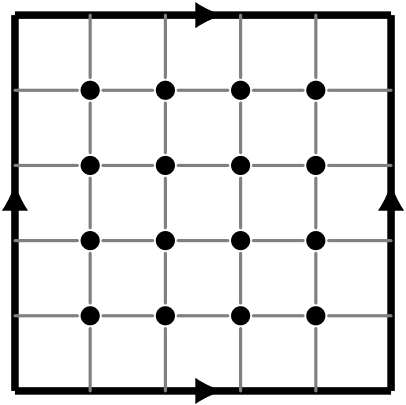
\begin{figure}
     \centering
    \begin{tikzpicture}[scale=0.5]
 	\node[vertex] at (2,2) (v0) {};
  \node[vertex] at (2,4) (v1) {};
  \node[vertex] at (2,6) (v2) {};
  \node[vertex] at (2,8) (v3) {};
  	\node[vertex] at (4,2) (v4) {};
  \node[vertex] at (4,4) (v5) {};
  \node[vertex] at (4,6) (v6) {};
  \node[vertex] at (4,8) (v7) {};
  	\node[vertex] at (6,2) (v8) {};
  \node[vertex] at (6,4) (v9) {};
  \node[vertex] at (6,6) (v10) {};
  \node[vertex] at (6,8) (v11) {};
  	\node[vertex] at (8,2) (v12) {};
  \node[vertex] at (8,4) (v13) {};
  \node[vertex] at (8,6) (v14) {};
  \node[vertex] at (8,8) (v15) {};

 	\begin{scope}[on background layer]
   \draw[line width = 1mm] (0,0) to (0,10);
    \draw[line width = 1mm] (0,0) to (10,0);
     \draw[line width = 1mm] (10,0) to (10,10);
      \draw[line width = 1mm] (0,10) to (10,10);

      \draw [arrows = {-Latex[width=10pt, length=10pt]}] (0,4.5) --
(0,5.5);
     \draw [arrows = {-Latex[width=10pt, length=10pt]}] (10,4.5) --
(10,5.5);
      \draw [arrows = {-Latex[width=10pt, length=10pt]}] (4.5,0) --
(5.5,0);
        \draw [arrows = {-Latex[width=10pt, length=10pt]}] (4.5,10) --
(5.5,10);
 \draw[edge] (v0) to (v1);
 		\draw[edge] (v1) to (v2);
 		\draw[edge] (v2) to (v3);
 		\draw[edge] (v4) to (v5);
 		\draw[edge] (v5) to (v6);
   	\draw[edge] (v6) to (v7);
 		\draw[edge] (v8) to (v9);
 		\draw[edge] (v9) to (v10);
   	\draw[edge] (v10) to (v11);
 		\draw[edge] (v12) to (v13);
 		\draw[edge] (v13) to (v14);
   	\draw[edge] (v14) to (v15);
    \draw[edge] (v0) to (v4);
 		\draw[edge] (v4) to (v8);
 		\draw[edge] (v8) to (v12);
 		\draw[edge] (v1) to (v5);
 		\draw[edge] (v5) to (v9);
   	\draw[edge] (v9) to (v13);
 		\draw[edge] (v2) to (v6);
 		\draw[edge] (v6) to (v10);
   	\draw[edge] (v10) to (v14);
 		\draw[edge] (v3) to (v7);
 		\draw[edge] (v7) to (v11);
   	\draw[edge] (v11) to (v15);
       \draw[edge] (v0) to (2,0);
 		\draw[edge] (v0) to (0,2);
 		\draw[edge] (v4) to (4,0);
 		\draw[edge] (v8) to (6,0);
 		\draw[edge] (v12) to (8,0);
   	\draw[edge] (v1) to (0,4);
 		\draw[edge] (v2) to (0,6);
 		\draw[edge] (v3) to (0,8);
   	\draw[edge] (v3) to (2,10);
 		\draw[edge] (v7) to (4,10);
 		\draw[edge] (v11) to (6,10);
   	\draw[edge] (v15) to (8,10);
    \draw[edge] (v12) to (10,2);
 		\draw[edge] (v13) to (10,4);
 		\draw[edge] (v14) to (10,6);
   	\draw[edge] (v15) to (10,8);
 	\end{scope}
 \end{tikzpicture}
     \caption{This embedding of $C_{2k} \times C_{2l}$ in the torus
satisfies all seven properties.}\label{fig:all7}
 \end{figure}

 \subsubsection{Three Fano properties (side of triangle)}
 \label{s2s:threeside}

 An embedding $\eG$ of a tree with at least two vertices in the plane is
orientable and bipartite (conditions $(001)$ and $(011)$) so it also
satisfies condition $(010)$.  However, the medial graph has a loop
corresponding to each leaf of the tree, and hence is not bipartite. 
Since condition $(111)$ does not hold, only the three conditions
$(001)$, $(011)$, and $(010)$ on the left side of the Fano triangle
hold.

 By taking the dual or the Wilson dual of such a $\eG$ we obtain
embeddings where exactly the three properties on one of the other two
sides of the Fano triangle hold.

 \subsubsection{Three Fano properties (median of triangle)}
 \label{s2s:threemedian}

 Figure \ref{fig:fanomedian} shows an example where we have twisted
one edge of a grid embedding of $C_{2k} \times C_{2l}$ in
the torus, as illustrated in Figure \ref{fig:all7}, 
 We keep the properties of being bipartite and having a bipartite medial
graph (conditions $(011)$ and $(111)$, which imply $(100)$) but lose
$2$-face-colorability (condition $(101)$).  Therefore, only the three
conditions $(011)$, $(111)$ and $(100)$ hold in this case, representing
the median of the Fano triangle from upper left to bottom right.

 By taking the dual or Wilson dual of this embeddding we obtain
embeddings where exactly the three properties on one of the other two
medians of the Fano triangle hold.

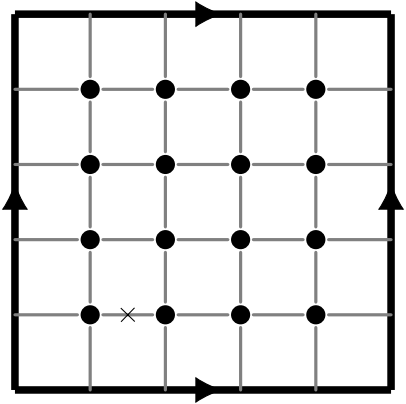
\begin{figure}
    \centering
   \begin{tikzpicture}[scale=0.5]
	\node[vertex] at (2,2) (v0) {};
 \node[vertex] at (2,4) (v1) {};
 \node[vertex] at (2,6) (v2) {};
 \node[vertex] at (2,8) (v3) {};
 	\node[vertex] at (4,2) (v4) {};
 \node[vertex] at (4,4) (v5) {};
 \node[vertex] at (4,6) (v6) {};
 \node[vertex] at (4,8) (v7) {};
 	\node[vertex] at (6,2) (v8) {};
 \node[vertex] at (6,4) (v9) {};
 \node[vertex] at (6,6) (v10) {};
 \node[vertex] at (6,8) (v11) {};
 	\node[vertex] at (8,2) (v12) {};
 \node[vertex] at (8,4) (v13) {};
 \node[vertex] at (8,6) (v14) {};
 \node[vertex] at (8,8) (v15) {};
  \node[cross=3pt, rotate=1] at (3,2) {};

	\begin{scope}[on background layer]
  \draw[line width = 1mm] (0,0) to (0,10);
   \draw[line width = 1mm] (0,0) to (10,0);
    \draw[line width = 1mm] (10,0) to (10,10);
     \draw[line width = 1mm] (0,10) to (10,10);

     \draw [arrows = {-Latex[width=10pt, length=10pt]}] (0,4.5) -- (0,5.5);
    \draw [arrows = {-Latex[width=10pt, length=10pt]}] (10,4.5) -- (10,5.5);
     \draw [arrows = {-Latex[width=10pt, length=10pt]}] (4.5,0) -- (5.5,0);
       \draw [arrows = {-Latex[width=10pt, length=10pt]}] (4.5,10) -- (5.5,10);
\draw[edge] (v0) to (v1);
		\draw[edge] (v1) to (v2);
		\draw[edge] (v2) to (v3);
		\draw[edge] (v4) to (v5);
		\draw[edge] (v5) to (v6);
  	\draw[edge] (v6) to (v7);
		\draw[edge] (v8) to (v9);
		\draw[edge] (v9) to (v10);
  	\draw[edge] (v10) to (v11);
		\draw[edge] (v12) to (v13);
		\draw[edge] (v13) to (v14);
  	\draw[edge] (v14) to (v15);
   \draw[edge] (v0) to (v4);
		\draw[edge] (v4) to (v8);
		\draw[edge] (v8) to (v12);
		\draw[edge] (v1) to (v5);
		\draw[edge] (v5) to (v9);
  	\draw[edge] (v9) to (v13);
		\draw[edge] (v2) to (v6);
		\draw[edge] (v6) to (v10);
  	\draw[edge] (v10) to (v14);
		\draw[edge] (v3) to (v7);
		\draw[edge] (v7) to (v11);
  	\draw[edge] (v11) to (v15);
      \draw[edge] (v0) to (2,0);
		\draw[edge] (v0) to (0,2);
		\draw[edge] (v4) to (4,0);
		\draw[edge] (v8) to (6,0);
		\draw[edge] (v12) to (8,0);
  	\draw[edge] (v1) to (0,4);
		\draw[edge] (v2) to (0,6);
		\draw[edge] (v3) to (0,8);
  	\draw[edge] (v3) to (2,10);
		\draw[edge] (v7) to (4,10);
		\draw[edge] (v11) to (6,10);
  	\draw[edge] (v15) to (8,10);
   \draw[edge] (v12) to (10,2);
		\draw[edge] (v13) to (10,4);
		\draw[edge] (v14) to (10,6);
  	\draw[edge] (v15) to (10,8);
	\end{scope}
\end{tikzpicture}
    \caption{This embedding is bipartite (011) and the medial graph is
bipartite (111), so we still have (100) as well, but this is not
2-face-colorable (101) so none of the other four properties can hold.} 
    \label{fig:fanomedian}
\end{figure}

 \subsubsection{Three Fano properties (circle)}
 \label{s2s:threecircle}

 Figure \ref{fig:fanocircle} shows an embedded graph $\eG$ in the
projective plane which is bipartite and $2$-face-colorable (conditions
$(011)$ and $(101)$, which imply condition $(110)$).  However, this
embedding is not orientable (condition $(001)$) and therefore it
satisfies only the three conditions corresponding to the circular line
of the Fano plane.

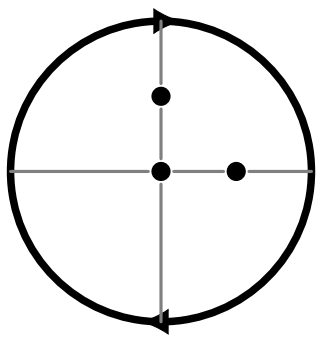
\begin{figure}
    \centering
   \begin{tikzpicture}[scale=0.5]
	\node[vertex] at (15,-10) (v0) {};
	\node[vertex] at (15,-8) (v1) {};
	\node[vertex] at (17,-10) (v2) {};
	
	\begin{scope}[on background layer]
 \draw[line width = 1mm] (15,-10) circle (4);
    \draw [arrows = {-Latex[width=10pt, length=10pt]}] (14.5,-6) -- (15.5,-6);
    \draw [arrows = {-Latex[width=10pt, length=10pt]}] (15.5,-14) -- (14.5,-14);
		\draw[edge] (v0) to (v2);
		\draw[edge] (v0) to (v1);
		\draw[edge] (v1) to (15,-6);
		\draw[edge] (v2) to (19,-10);
		\draw[edge] (v0) to (15,-14);
		\draw[edge] (v0) to (11,-10);

	\end{scope}
\end{tikzpicture}
    \caption{This graph drawn in the projective plane is bipartite,
2-face-colorable, but not orientable.}
    \label{fig:fanocircle}
\end{figure}

 \subsubsection{One Fano property (orientability)}
 \label{s2s:oneori}

 In Figure \ref{fig:justorientable} we give an embedding of $K_4$ in the
plane that is such that neither the graph nor its dual are Eulerian
(even-vertex). Therefore, it is orientable (condition $(001)$), but
satisfies none of the other seven properties.

 We can find examples that satisfy only condition $(010)$ or only
condition $(100)$ by taking the Petrie dual or Wilson dual,
respectively, of this example.

 \begin{figure}
     \centering
    \begin{tikzpicture}[scale=0.5]
 	\node[vertex] at (0,0) (v0) {};
  \node[vertex] at (0,4) (v1) {};
  \node[vertex] at (3,-3) (v2) {};
  \node[vertex] at (-3,-3) (v3) {};

 	\begin{scope}[on background layer]
 \draw[edge] (v0) to (v1);
 		\draw[edge] (v0) to (v2);
 		\draw[edge] (v0) to (v3);
 		\draw[edge] (v1) to (v2);
 		\draw[edge] (v1) to (v3);
   	\draw[edge] (v2) to (v3);
 	\end{scope}
 \end{tikzpicture}
     \caption{This embedding of $K_4$ in the plane satisfies only
$(001)$.}\label{fig:justorientable}
 \end{figure}
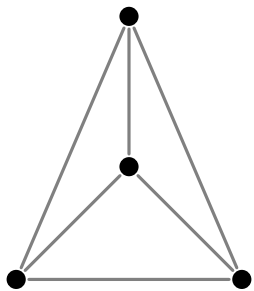

 \subsubsection{One Fano property (bipartiteness)}
 \label{s2s:onebip}

 The embedding of $K_{3,3}$ in the projective plane shown in Figure
\ref{fig:justbip} is bipartite, but it is not orientable,
2-face-colorable, or Eulerian (even-vertex).
 Thus, by Metatheorem B it only satisfies $(011)$.

 By taking a dual or Wilson dual of this example, we can find embeddings
that satisfy only condition $(101)$ or only condition $(110)$,
respectively.

 \begin{figure}
     \centering
    \begin{tikzpicture}[scale=0.5]
 	\node[vertex] at (15,-12) (v0) {};
 	\node[vertex] at (15,-8) (v1) {};
 	\node[vertex] at (13,-11) (v2) {};
 	\node[vertex] at (13,-9) (v3) {};
 	\node[vertex] at (17,-11) (v4) {};
 	\node[vertex] at (17,-9) (v5) {};
 	\begin{scope}[on background layer]
  \draw[line width = 1mm] (15,-10) circle (4);
     \draw [arrows = {-Latex[width=10pt, length=10pt]}] (14.5,-6) --
(15.5,-6);
     \draw [arrows = {-Latex[width=10pt, length=10pt]}] (15.5,-14) --
(14.5,-14);
 		\draw[edge] (v0) to (v2);
 		\draw[edge] (v0) to (v4);
 		\draw[edge] (v1) to (v3);
 		\draw[edge] (v1) to (v5);
 		\draw[edge] (v2) to (v3);
 		\draw[edge] (v4) to (v5);
   \draw[edge] (v5) to (18.464,-8);
   \draw[edge] (v4) to (18.464,-12);
   \draw[edge] (v0) to (15,-14);
     \draw[edge] (v3) to (11.536,-8);
   \draw[edge] (v2) to (11.536,-12);
 		\draw[edge] (v0) to (v4);
 		\draw[edge] (v1) to (15,-6);
 		\draw[edge] (v1) to (v5);
 		\draw[edge] (v2) to (v3);
 		\draw[edge] (v4) to (v5);
 	\end{scope}
 \end{tikzpicture}
     \caption{$K_{3,3}$ in the projective plane satisfies only
$(011)$.}\label{fig:justbip}
 \end{figure}
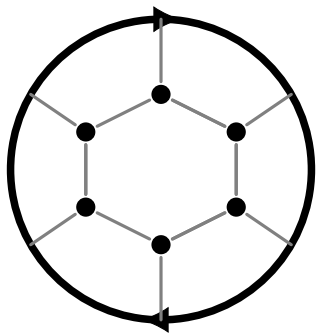

 \subsubsection{One Fano property (medial graph is bipartite)}
 \label{s2s:onemedialbip}

 By taking a twisted dual of the graph in Figure \ref{fig:all7}, we can
create the graph in Figure \ref{fig:onlymedialbipartite} which satisfies
only $(111)$.
 By taking the partial Petrie dual of one edge and the partial dual of
another, the resulting graph in Figure \ref{fig:onlymedialbipartite} is
not orientable, bipartite, or 2-face-colorable. The medial graph is
still bipartite since twisted duality does not impact the underlying
medial graph. Since the embedding satisfies $(111)$ but not $(001)$,
$(011)$, or $(101)$, by Metatheorem B none of the other three conditions
can be satisfied.

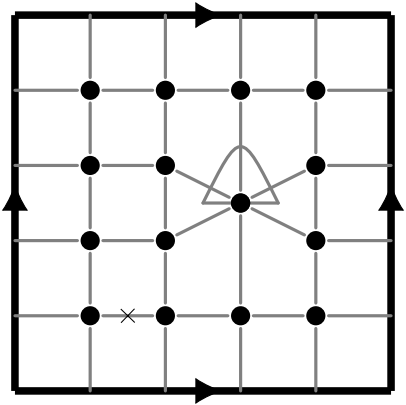
\begin{figure}
    \centering
   \begin{tikzpicture}[scale=0.5]
	\node[vertex] at (2,2) (v0) {};
 \node[vertex] at (2,4) (v1) {};
 \node[vertex] at (2,6) (v2) {};
 \node[vertex] at (2,8) (v3) {};
 	\node[vertex] at (4,2) (v4) {};
 \node[vertex] at (4,4) (v5) {};
 \node[vertex] at (4,6) (v6) {};
 \node[vertex] at (4,8) (v7) {};
 	\node[vertex] at (6,2) (v8) {};
 \node[vertex] at (6,5) (v9) {};
 \node[vertex] at (6,5.000000000001) (v10) {};
 \node[vertex] at (6,8) (v11) {};
 	\node[vertex] at (8,2) (v12) {};
 \node[vertex] at (8,4) (v13) {};
 \node[vertex] at (8,6) (v14) {};
 \node[vertex] at (8,8) (v15) {};
 \node[cross=3pt, rotate=1] at (3,2) {};

	\begin{scope}[on background layer]
  \draw[line width = 1mm] (0,0) to (0,10);
   \draw[line width = 1mm] (0,0) to (10,0);
    \draw[line width = 1mm] (10,0) to (10,10);
     \draw[line width = 1mm] (0,10) to (10,10);
        \draw [arrows = {-Latex[width=10pt, length=10pt]}] (0,4.5) -- (0,5.5);
    \draw [arrows = {-Latex[width=10pt, length=10pt]}] (10,4.5) -- (10,5.5);
     \draw [arrows = {-Latex[width=10pt, length=10pt]}] (4.5,0) -- (5.5,0);
       \draw [arrows = {-Latex[width=10pt, length=10pt]}] (4.5,10) -- (5.5,10);

\draw[edge] (v0) to (v1);
		\draw[edge] (v1) to (v2);
		\draw[edge] (v2) to (v3);
		\draw[edge] (v4) to (v5);
		\draw[edge] (v5) to (v6);
  	\draw[edge] (v6) to (v7);
		\draw[edge] (v8) to (v9);
		\draw[edge] (5,5) to (7,5);
        \draw[edge] (5,5) .. controls (6,7) .. (7,5);
  	\draw[edge] (v10) to (v11);
		\draw[edge] (v12) to (v13);
		\draw[edge] (v13) to (v14);
  	\draw[edge] (v14) to (v15);
   \draw[edge] (v0) to (v4);
		\draw[edge] (v4) to (v8);
		\draw[edge] (v8) to (v12);
		\draw[edge] (v1) to (v5);
		\draw[edge] (v5) to (v9);
  	\draw[edge] (v9) to (v13);
		\draw[edge] (v2) to (v6);
		\draw[edge] (v6) to (v10);
  	\draw[edge] (v10) to (v14);
		\draw[edge] (v3) to (v7);
		\draw[edge] (v7) to (v11);
  	\draw[edge] (v11) to (v15);
      \draw[edge] (v0) to (2,0);
		\draw[edge] (v0) to (0,2);
		\draw[edge] (v4) to (4,0);
		\draw[edge] (v8) to (6,0);
		\draw[edge] (v12) to (8,0);
  	\draw[edge] (v1) to (0,4);
		\draw[edge] (v2) to (0,6);
		\draw[edge] (v3) to (0,8);
  	\draw[edge] (v3) to (2,10);
		\draw[edge] (v7) to (4,10);
		\draw[edge] (v11) to (6,10);
  	\draw[edge] (v15) to (8,10);
   \draw[edge] (v12) to (10,2);
		\draw[edge] (v13) to (10,4);
		\draw[edge] (v14) to (10,6);
  	\draw[edge] (v15) to (10,8);
	\end{scope}
\end{tikzpicture}
    \caption{This partial twisted dual of $C_{2k} \times C_{2l}$ in the
torus satisfies only $(111)$.}
    \label{fig:onlymedialbipartite}
\end{figure}

 \subsubsection{No Eulerian or Fano properties}
 \label{s2s:none}

 The embedding in Figure \ref{fig:updatednoneof10} is not even-vertex
(the vertex degrees are $5$ and $1$), not even-face (it has a face of
degree $1$), and not even-zigzag (it has a zigzag of degree $1$).  Since
it has none of the Eulerian properties, it also has none of the Fano
properties.

\begin{figure}

\begin{center}
   \begin{tikzpicture}[scale=0.5]
\node[vertex] at (0,0) (v0) {};
 \node[vertex] at (0,5) (v1) {};
  \node[cross=6pt, rotate=2] at (2,5.6) {};

	\begin{scope}[nodes={sloped,allow upside down}][on background layer]
		\draw[edge] (v0) to (v1);
		\draw[edge] (v1) .. controls (-5,10) and (-5, 6).. (v1);
        \draw[edge] (v1) .. controls (5, 10) and (5,6).. (v1);
	
	\end{scope}
\end{tikzpicture}
   \end{center}

  \caption{A graph $G$ satisfying none of the ten properties.}
   \label{fig:updatednoneof10}
 
\end{figure}
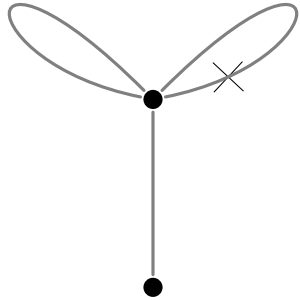
\hfill

 \subsection{Example with all Eulerian properties but no Fano properties}
 \label{ss:exalleuleriannofano}

 In Subsubsection \ref{s2s:none} above we provided an example with none
of the seven Fano properties or three Eulerian properties.  However,
there are also examples showing that no combination of the Eulerian
properties implies any of the Fano properties.

 \subsubsection{All Eulerian but no Fano properties}
 \label{s2s:alleuleriannofano}

 The embedding $\eG$ in Figure \ref{fig:eulerianpropertiesonly} has a single
vertex, a single face, and a single zigzag.  Therefore, the vertex,
face, and zigzag all have degree $6$, and the embedding has all three
Eulerian properties.  However, because it has only a single vertex,
single face, and single zigzag, it is not bipartite, $2$-face-colorable,
or $2$-zigzag-colorable.  The medial graph contains a triangle, so is
not bipartite.  Since $\eG$ contains a twisted loop it is not
orientable, and $\eG\spe$ contains two twisted loops, so is not
orientable.  Finally, $\eG$ is not directable: if we direct the edges so
that each face is bounded by a directed walk, then the half-edge
directions around the vertex would have to alternate between inward and
outward, which is not possible.

\begin{figure}
\begin{center}
   \begin{tikzpicture}[scale=0.7]
\node[vertex] at (0,0) (v0) {};
  \node[cross=6pt, rotate=2] at (0,-2.3) {};

	\begin{scope}[nodes={sloped,allow upside down}][on background layer]
		\draw[edge] (v0) .. controls (-2,-8) and (-2,8).. (v0);
		\draw[edge] (v0) .. controls (-2,3) and (2, 3).. (v0);
        \draw[edge] (v0) .. controls (-2, -3) and (2,-3).. (v0);
	
	\end{scope}
\end{tikzpicture}
   \end{center}
  \caption{A graph $G$ satisfying none of the fano properties but all
three eulerian properties.}
   \label{fig:eulerianpropertiesonly}
 
\end{figure}
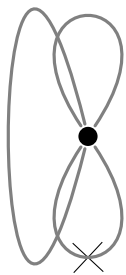
\hfill

 \section{Eulerian and single-vertex partial duals}\label{sec:allpds}

 In this section we discuss another property of embedded graphs that can
be represented by a congruence condition involving closed walks in the
jewel, and a strengthening of that property which we can also
characterize.

 \subsection{Eulerian partial duals}

 A \emph{vertex-face walk} in a jewel $\cL$ is a closed walk that alternates
between edges colored $\ca$ and edges colored either $\cv$ or $\cf$. A
\emph{vertex-face cycle} or \emph{v/f-gon} in a jewel $\cL$ is a cycle that
alternates between edges colored $\ca$ and edges colored either $\cv$ or
$\cf$. These v/f-gons represent a vertex (and equivalently a face) in
some partial dual of the cellularly embedded graph $\eG$ represented by
$\cL$.

 \begin{theorem}\label{vfgontheorem} Let $\eG$ be a connected, cellularly
embedded graph with corresponding jewel $\cL$. All partial duals of $\eG$ are
Eulerian if and only if all v/f-gons have length 0 (mod 4).
 \end{theorem}

 \begin{proof} Let $\eG$ be a connected, cellularly embedded graph with
corresponding jewel $\cL$. First assume that all partial duals of $\eG$
are Eulerian. Then for any $A\subseteq E(G)$, the v-gons of $\cL[\du A]$
must have an even number of edges colored $\cv$. This means that the
total length of the v-gon must be equal to 0 (mod 4). But, every v/f-gon
in $\cL$ represents a v-gon in $\cL[\du A]$ for some edge set $A
\subseteq E(G)$. To see this, begin with a v/f-gon $C$ in $\cL$. Note
that since $C$ is a cycle, it cannot cross an e-square in $\cL$ along
both an edge colored $\cv$ and an edge colored $\cf$. Using both an edge
colored $\cv$ and an edge colored $\cf$ in the same e-square causes a
vertex-face walk to use all three edges incident with a vertex, thus
making it not a cycle. So, for each e-square $e$ in $\cL$, $C$ crosses
$e$ along at most one color (it may use one or both edges of that color,
but it may only use one of the two colors). Let $A$ be the set of edges
in $\eG$ corresponding to e-squares in $\cL$ that $C$ crosses along
edges colored $\cf$. Then consider $\eG\du A$. In $\eG\du A$, $C$
becomes a v-gon since it is a v/f-gon with only edges colored $\ca$ and
$\cv$. Therefore $C$ must have an even number of edges colored $\cv$
since $\eG\du A$ is Eulerian by assumption and so $C$ has length 0 (mod
4). Thus, all v/f-gons have length 0 (mod 4).

 Conversely, assume that all v/f-gons have length 0 (mod 4). Let $A$ be
an arbitrarily chosen subset of $E(G)$. Let $v$ be a vertex of $\eG\du
A$. Then $v$ is represented by a v-gon in $\eG\du A$ which is itself a
v/f-gon in $\cL$ up to possibly a different coloring on the edges
colored $\cv$ and $\cf$. The length of the v-gon representing $v$ is 0
(mod 4) and thus $v$ has even degree. Since $v$ and $A$ were chosen
arbitrarily, we see that every vertex in a partial dual of $\eG$ has
even degree. Therefore, all partial duals of $\eG$ are Eulerian.  
 \end{proof}

 \begin{corollary}\label{vfgoncoro} Let $\eG$ be a connected, cellularly
embedded graph with corresponding jewel $\cL$. If for each cycle $C$ in $\cL$
either $(v+f)_\cL(C)$ is even or $a_\cL(C)$ is even, then all partial duals
of $\eG$ are Eulerian. 
 \end{corollary}

 \begin{proof} Assume that $(v+f)_\cL(C)$ is even or $a_\cL(C)$ is even for
all cycles $C$ in $\cL$. Then for each v/f-gon $C$ in $\cL$, since v/f-gons
are themselves cycles in $\cL$, $(v+f)_\cL(C)$ or $a_\cL(C)$ is even. However,
in any v/f-gon $(v+f)_\cL(C) = a_\cL(C)$. Thus the total length of $C$ is
equal to 0 (mod 4). Therefore, by Theorem \ref{vfgontheorem}, all
partial duals of $\eG$ are Eulerian. 
 \end{proof} 

 Since the functions $(v+f)_\cL$ and $a_\cL$ correspond to conditions
$(110)$ and $(111)$, respectively, we have the following.

 \begin{corollary}\label{eulerpdcoro} Let $\eG$ be a connected,
cellularly embedded graph with corresponding medial graph $M$. If $\eG$
is 2-zigzag-colorable or if $M$ is bipartite then all partial duals of
$\eG$ are Eulerian. 
 \end{corollary}

 Note that the converse of Corollary \ref{vfgoncoro} is not true. Figure
\ref{fig:EulerianPDCounterexample} shows a counterexample. To simplify
the figure we have shown only the gem $\cJ$, rather than the jewel
$\cL$; this still shows all of the v/f-gons.  This connected graph has
partial duals that have either a single vertex of degree $6$ or two
vertices of degree $4$ and $2$. So all partial duals are
Eulerian, but in this graph there exists a cycle with
$a_\cJ(C)=a_\cL(C)$ odd and $(v+f)_\cJ(C)=(v+f)_\cL(C)$ odd. One such
cycle is $(0,2,3,4,8,9,10,11)$ with three edges colored $\ca$, two edges
colored $\cf$, and three edges colored $\cv$.
 \begin{figure}

 \begin{center}
 \begin{tikzpicture}[scale=.7]
 	\node[vertex, label=above:0] at (10.709,-4.838) (v0) {};
 	\node[vertex, label=right:2] at (14.254,-6.532) (v1) {};
 	\node[vertex, label=above:1] at (12.815,-4.805) (v2) {};
 	\node[vertex, label=left:11] at (9.339,-6.583) (v3) {};
 	\node[vertex, label=above:10] at (8.488,-8.298) (v4) {};
 	\node[vertex, label=left:9] at (8.942,-9.569) (v5) {};
 	\node[vertex, label=below:8] at (9.835,-10.423) (v6) {};
 	\node[vertex, label=below:7] at (10.859,-10.908) (v7) {};
 	\node[vertex, label=below:6] at (13.141,-10.933) (v8) {};
 	\node[vertex, label=below:5] at (14.249,-10.359) (v9) {};
 	\node[vertex, label=right:4] at (14.985,-9.340) (v10) {};
 	\node[vertex, label=above:3] at (15.372,-8.231) (v11) {};
 	\begin{scope}[on background layer]
 		\draw[edge, blue] (v0) .. controls (12.858, -3.084) and (14.045,
-3.583) .. (v1);
 		\draw[edge, yellow!60!orange] (v0) to (v2);
 		\draw[edge, red] (v0) to (v3);
 		\draw[edge, red] (v1) to (v2);
 		\draw[edge, yellow!60!orange] (v1) to (v11);
 		\draw[edge, blue] (v2) .. controls (10.649, -3.057) and (9.491,
-3.650) .. (v3);
 		\draw[edge, yellow!60!orange] (v3) to (v4);
 		\draw[edge, red] (v4) to (v5);
 		\draw[edge, blue] (v4) to (v8);
 		\draw[edge, yellow!60!orange] (v5) to (v6);
 		\draw[edge, blue] (v5) to (v9);
 		\draw[edge, red] (v6) to (v7);
 		\draw[edge, blue] (v6) to (v10);
 		\draw[edge, yellow!60!orange] (v7) to (v8);
 		\draw[edge, blue] (v7) to (v11);
 		\draw[edge, red] (v8) to (v9);
 		\draw[edge, yellow!60!orange] (v9) to (v10);
 		\draw[edge, red] (v10) to (v11);
 	\end{scope}
 \end{tikzpicture} \includegraphics[]{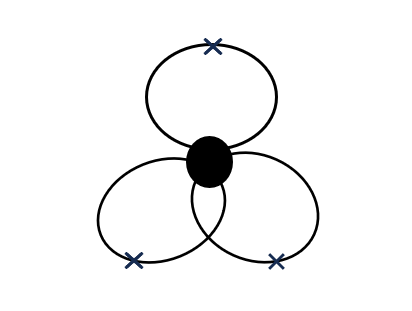}

 \end{center}
  
    \caption{A gem and its corresponding graph showing that the converse
of Corollary \ref{vfgoncoro} is false.}
    \label{fig:EulerianPDCounterexample}
 \end{figure}
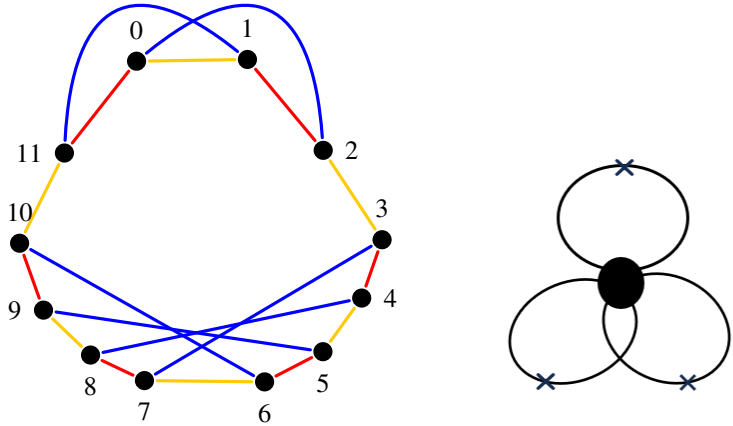

 \subsection{Single-vertex partial duals}
 We know that every connected graph embedding $\eG$ has some Eulerian
partial dual.  In particular, it is well known that if $T$ is the set of
edges of a spanning tree in $\eG$, then $\eG\du T$ has a single vertex.
This vertex is incident with both ends of every edge, and thus $\eG\du
T$ is Eulerian.

 Single-vertex embeddings are an important special class of embeddings
that have been studied from a number of perspectives, and have
connections to other ideas such as chord diagrams.  A \emph{chord
diagram} consists of a circle drawn in the plane, with a finite number
of (straight-line) chords with disjoint sets of ends.
 Since every connected embedded graph has at least one single-vertex
partial dual, it is natural, as a strengthening of the result in the
previous subsection, to ask which embeddings (a subset of those from
Theorem \ref{vfgontheorem}) have the property that all partial duals are
single-vertex.  Note that this is equivalent to all partial duals being
single-face.
 In this subsection we answer this question.

 Consider an embedded graph $\eG$ all of whose partial duals are
single-vertex.
 The effects of taking a partial dual with respect to a single edge are
detailed in \cite[Section 2.2]{GraphsonSurfEMM}.
 If an embedded graph has an untwisted loop, then taking the
partial dual with respect to that loop splits the incident vertex into
two vertices.
 So $\eG$, and all partial duals of $\eG$, contain no untwisted loops.
Clearly, all edges of $\eG$ are loops, so all edges in $\eG$ are twisted
loops. Two loops $e$ and $f$ are \emph{interlaced} if they share a
common vertex $v$ and the cyclic order of edge labels around $v$ is of
the form $e \dots f \dots e \dots f \dots$.
 If two twisted loops are interlaced, then taking the partial dual with
respect to one of the two twisted loops results in the other loop
becoming untwisted.
 So $\eG$ has no interlaced twisted loops.  If there
are no interlaced twisted loops, then taking the partial dual with
respect to a noninterlaced twisted loop keeps that edge and all other
edges as noninterlaced twisted loops. So we arrive at the following
characterization. 

 \begin{theorem}
 The partial duals of an embedded graph $\eG$ are all single-vertex (or,
equivalently, all single-face)
 if and only if $\eG$ has a single vertex, every edge of $\eG$ is a
twisted loop, and no two edges are interlaced.
 \end{theorem}

 A graph is \emph{outerplanar} if it has an embedding in the plane with
every vertex on the boundary of the outer face.  The end graph of
a graph embedding $\eG$ was defined in Subsubsection \ref{cond100};
it is $3$-regular, and at each vertex two of the edges correspond to
edges of color $\ca$ in the gem $\cJ$, while the other edges correspond
to edges of color $\cf$ in $\cJ$.  The `colored end graph', which we
will denote by $\cN$, transfers these colors to the edges of the end
graph.

 When $\eG$ is a single-vertex embedding, $\cN$ represents the chord
diagram associated with $\eG$: the edges of color $\ca$ in $\cN$ form a
cycle corresponding to the circle of the chord diagram, and the edges of
color $\cf$ in $\cN$ form a matching whose elements correspond to the
chords.  The fact that the edges of $\eG$ are not interlaced means that
the chords of the chord diagram are nonintersecting.  Since a chord
diagram without intersecting chords can be thought of as a
$2$-connected $3$-regular outerplanar graph, with the circle of the
chord diagram forming the boundary of the outer face, we obtain the
following.

 Having all partial duals single-vertex is a property preserved by
(full) duality, which swaps the end graph with the $3$-regular side
graph described in Subsubsection \ref{cond010}.  We can define a
`colored side graph', where every vertex is incident with two edges
colored $\ca$, and one edge colored $\cv$.  Duality also swaps loops
with \emph{coloops}, edges that have the same face on both sides.
 A coloop is \emph{twisted} if the face traverses the edge twice in the
same direction, and \emph{untwisted} if it traverses the edge once in
each direction.  Duality swaps twisted loops and twisted coloops.

 Putting all of this together, we have the following.

 \begin{theorem}
 Let $\eG$ be an embedded graph.  Then the following are equivalent.
 \begin{enumerate}[label=\rm(\alph*)]\setlength{\itemsep}{0pt}%
  \item
    The partial duals of $\eG$ are all single-vertex.
  \item
    The partial duals of $\eG$ are all single-face.
  \item
    Every edge of $\eG$ is a twisted loop and the colored end graph of
$\eG$ can be drawn as an outerplanar graph where all outside edges are
colored $\ca$.
  \item
    Every edge of $\eG$ is a twisted coloop and the colored side graph of
$\eG$ can be drawn as an outerplanar graph where all outside edges are
colored $\ca$.
 \end{enumerate}
 \end{theorem}

 We can also state similar theorems about when all partial Petrial duals
are single-face and single-zigzag, and when all partial Wilson duals are
single-vertex and single-zigzag.  We leave the details to the reader.

 \section{Conclusion}\label{sec:conclusion}

 There are two obvious directions in which it would make sense to try to
extend the results in this paper.

 The first is for delta-matroids.  These were introduced by Bouchet, who
showed that every embedded graph has an associated delta-matroid
\cite{Bou87, Bou89}.  Properties of the delta-matroids associated with embedded
graphs have been extensively investigated by Chun, Moffatt, Noble, and
Rueckriemen \cite{CMNR19a, CMNR19b}.  The twisted duality framework for graph
embeddings corresponds to natural operations in their delta-matroids,
and these operations can be extended to larger classes of
delta-matroids, such as binary delta-matroids.  The largest class of
delta-matroids for which these operations make sense is the class of
\emph{vf-safe} delta-matroids, defined by Brijder and Hoogeboom
\cite{BH13}.  It is natural to try to extend our Fano framework to binary
delta-matroids, or even to vf-safe delta matroids.  A step in this
direction has already been taken by Yan and Jin \cite[Theorem 3.6]{YJ22},
who showed that Theorem \ref{twoimplythethird} can be extended to binary
delta-matroids.  The overall difficulty with this direction of research
is that it is unclear what the counterpart of the medial graph is for
delta-matroids.

 Another possible direction in which our results might be extended is to
hypermaps, which are embeddings of hypergraphs in surfaces.  There is a
representation of hypermaps as edge-colored cubic graphs which
generalizes the idea of gems; see \cite{ChVT22}.  However, there is
no obvious counterpart of jewels in this setting, which means that we
may perhaps only be able to extend results that can be expressed in
terms of gems.

  \end{document}